\documentclass{amsart}
\usepackage[backref]{hyperref}
\usepackage{color}
\definecolor{mylinkcolor}{rgb}{0.5,0.0,0.0}
\definecolor{myurlcolor}{rgb}{0.0,0.0,0.75}
\hypersetup{colorlinks=true,urlcolor=myurlcolor,citecolor=myurlcolor,linkcolor=mylinkcolor,linktoc=page,breaklinks=true}
\usepackage{amsmath,amssymb,amsthm,amscd,amsfonts,fullpage}
\usepackage{array,enumerate,centernot,booktabs,multirow,bigdelim}
\usepackage{graphicx}
\usepackage{placeins}

\newcommand{\Cdot}{\raisebox{-0.25ex}{\scalebox{1.3}{$\cdot$}}}
\newcommand{\overbar}[1]{\mkern 2mu\overline{\mkern-2mu#1\mkern-2mu}\mkern 2mu}
\newcommand{\lb}{\multirow{2}{-6pt}{$\Bigl\{$}}
\newcommand{\blb}{\multirow{4}{-6pt}{$\left\{\rule{0pt}{22pt}\right.$}}

\newcommand{\Q}{\mathbf{Q}}

\newcommand{\C}{\mathbf{C}}
\newcommand{\Qbar}{\overbar{\mathbf{Q}}}
\newcommand{\Kbar}{\overbar{K}}
\newcommand{\Z}{\mathbf{Z}}
\newcommand{\Zhat}{\hat \Z}
\newcommand{\Fp}{\mathbf{F}_{\!p}}
\newcommand{\Fq}{\mathbf{F}_{\!q}}

\newcommand{\F}{\mathbf{F}}

\newcommand{\GL}{{\mathbf{GL}}}
\newcommand{\SL}{{\mathbf{SL}}}
\newcommand{\PGL}{{\mathbf{PGL}}}
\newcommand{\PSL}{{\mathbf{PSL}}}

\renewcommand{\O}{\mathcal{O}}
\newcommand{\smallmat}[4]{\left(\begin{smallmatrix}#1&#2\\#3&#4\end{smallmatrix}\right)}
\newcommand{\diagmat}[2]{\smallmat{#1}{0}{0}{#2}}
\newcommand{\unimat}{\smallmat{1}{1}{0}{1}}
\newcommand{\mat}{{\mathbf{M}}}
\newcommand{\p}{\mathfrak p}
\newcommand{\q}{\mathfrak q}
\newcommand{\Exp}{\mathbf E}
\newcommand{\End}{\mathrm{End}}
\renewcommand{\Re}{\operatorname{Re}}
\newcommand{\Frob}{\mathrm{Frob}}

\def\disc{\operatorname{disc}}
\def\Gal{\operatorname{Gal}}
\def\Aut{\operatorname{Aut}}
\def\Li{\operatorname{Li}}

\def\tr{\operatorname{tr}}
\def\det{\operatorname{det}}
\def\im{\operatorname{im}}
\def\lcm{\operatorname{lcm}}

\def\sig{\operatorname{sig}}

\def\ndiv{\centernot\mid}
\def\hairspace{\kern .05em}
\newcommand{\M}{\mathsf{M}}
\newcommand{\sym}[1]{\mathrm{S}_#1}
\newcommand{\alt}[1]{\mathrm{A}_#1}

\newcommand{\smallminus}{\scalebox{0.66}[1.0]{\( - \)}}
\newcommand{\smallplus}{\scalebox{0.66}[1.0]{\( + \)}}

\theoremstyle{plain}
\newtheorem{theorem}{Theorem}[section]
\newtheorem{lemma}[theorem]{Lemma}
\newtheorem{corollary}[theorem]{Corollary}
\newtheorem{proposition}[theorem]{Proposition}
\newtheorem{conjecture}[theorem]{Conjecture}

\theoremstyle{definition}
\newtheorem{definition}[theorem]{Definition}
\newtheorem{remark}[theorem]{Remark}
\newtheorem{example}[theorem]{Example}
\newtheorem{algorithm}{Algorithm}

\begin{document}

\title{Computing images of Galois representations\\attached to elliptic curves}
\author{Andrew V. Sutherland}
\address{Department of Mathematics\\Massachusetts Institute of Technology\\Cambridge, Massachusetts\ \  02139}
\email{drew@math.mit.edu}
\thanks{The author was supported by NSF grant DMS-1115455}

\begin{abstract}
Let $E$ be an elliptic curve without complex multiplication (CM) over a number field $K$, and let $G_E(\ell)$ be the image of the Galois representation induced by the action of the absolute Galois group of~$K$ on the $\ell$-torsion subgroup of~$E$.
We present two probabilistic algorithms to simultaneously determine $G_E(\ell)$ up to local conjugacy for all primes $\ell$ by sampling images of Frobenius elements; one is of Las Vegas type and the other is a Monte Carlo algorithm.
They determine $G_E(\ell)$ up to one of at most two isomorphic conjugacy classes of subgroups of $\GL_2(\Z/\ell\Z)$ that have the same semisimplification, each of which occurs for an elliptic curve isogenous to $E$.
Under the GRH, their running times are polynomial in the bit-size $n$ of an integral Weierstrass equation for~$E$, and for our Monte Carlo algorithm, quasi-linear in $n$.
We have applied our algorithms to the non-CM elliptic curves in Cremona's tables and the Stein--Watkins database, some 140 million curves of conductor up to~$10^{10}$, thereby obtaining a conjecturally complete list of 63 exceptional Galois images $G_E(\ell)$ that arise for $E/\Q$ without CM.
Under this conjecture we determine a complete list of 160 exceptional Galois images $G_E(\ell)$ the arise for non-CM elliptic curves over quadratic fields with rational $j$-invariants.
We also give examples of exceptional Galois images that arise for non-CM elliptic curves over quadratic fields only when the $j$-invariant is irrational.
\end{abstract}
\maketitle

\tableofcontents

\section{Introduction}
Let $E$ be an elliptic curve over a number field $K$ with algebraic closure~$\Kbar$.
For each integer $m>1$, let $E[m]$ denote the $m$-torsion subgroup of $E(\Kbar)$, which we recall is a free $\Z/m\Z$ module of rank $2$.
The absolute Galois group $\Gal(\Kbar/K)$ acts on $E[m]$ via its action on the coordinates of its points, and this action induces a Galois representation (a continuous homomorphism):
\[
\rho_{E,m}\colon \Gal(\Kbar/K)\to \Aut(E[m])\simeq \GL_2(m):=\GL_2(\Z/m\Z).
\]
We regard the image of $\rho_{E,m}$ as a subgroup $G_E(m)$ of $\GL_2(m)$ that is determined only up to conjugacy, since the isomorphism $\Aut(E[m])\simeq \GL_2(m)$ depends on a choice of basis.
For fixed $E$ and varying $m$, the representations $\rho_{E,m}$ form a compatible system, and we have the adelic Galois representation
\[
\rho_E\colon \Gal(\Kbar/K)\to \GL_2(\hat \Z) = \varprojlim_m \GL_2(m),
\]
whose image we denote $G_E$.

By Serre's open image theorem (see \cite[\S IV.3.2]{serre68} and \cite{serre72}), so long as $E$ does not have complex multiplication (CM), the adelic image $G_E$ has finite index in $\GL_2(\hat \Z)$.
In particular, there is a minimal positive integer~$m_E$ for which $G_E$ is the full inverse image of $G_E(m_E)$, and a finite set $S_E$ of exceptional primes~$\ell$ for which $G_E(\ell)$ is properly contained in $\GL_2(\ell)$.
Each such $\ell$ necessarily divides~$m_E$, but the converse is not true in general (and almost never true for elliptic curves over $\Q$).
Nevertheless, a first step toward computing~$m_E$ and~$G_E(m_E)$ is to determine the set $S_E$ and the groups $G_E(\ell)$ for $\ell\in S_E$.

A related motivating question is this: for a given number field $K$, which exceptional groups $G_E(\ell)$ can arise for a non-CM elliptic curve $E/K$?
Serre's theorem implies that for any fixed $E$ this is a finite list, and Serre has asked whether this is still true when only $K$ is fixed and $E/K$ is allowed to vary; it is expected that the answer is yes.
This can be regarded as a generalization of Mazur's results \cite{mazur77,mazur78}, which determine the primes $\ell$ for which an elliptic curve $E/\Q$ may admit a rational point of order $\ell$, or a rational isogeny of degree~$\ell$.
Both of these properties are determined by $G_E(\ell)$, but the converse does not hold: $G_E(\ell)$ may be exceptional when~$E$ does not admit a rational isogeny of degree~$\ell$, and even when~$E$ has a rational point of order~$\ell$, many different $G_E(\ell)$ may occur.
Serre's question remains open for all number fields $K$, but there has been some recent progress in the case $K=\Q$: for $\ell >37$ any exceptional~$G_E(\ell)$ must lie in the normalizer of a non-split Cartan group in $\GL_2(\ell)$, as shown in \cite{bpr13}, and for $\ell\le 11$ the possible $G_E(\ell)$ have been completely determined \cite{zywina15}.
Little is known for number fields other than $\Q$.

We are thus led to the problem at hand: given an elliptic curve $E/K$ without CM, determine the set~$S_E$ of exceptional primes $\ell$ and the groups $G_E(\ell)$ for each prime $\ell\in S_E$.
Serre's open image theorem can be made effective, and under the generalized Riemann hypothesis (GRH) reasonably good bounds on the exceptional primes~$\ell$ are known; quasi-linear in the norm of the conductor of $E$, by \cite{lv14}.
This leaves the problem of computing $G_E(\ell)$.
In principle this is straight-forward: pick a basis for $E[\ell]$ and compute the action of $\Gal(\Kbar/K)$ on this basis.
This approach can be made completely effective.
The points in $E[\ell]$ are defined over the $\ell$-torsion field $K(E[\ell])$, which is an extension of the splitting field of the $\ell$-division polynomial $f_{E,\ell}(x)$ whose roots are the $x$-coordinates of the non-trivial $\ell$-torsion points.
Using well-known formulas for $f_{E,\ell}(x)$ one can explicitly construct its splitting field and take a quadratic extension if necessary to obtain the $y$-coordinates of the points in~$E[\ell]$ (a quadratic extension always suffices, see Lemma~\ref{lem:quadsplit}).
One then finds generators for $\Gal(K(E[\ell])/K)$ and applies them to a basis for $E[\ell]$.
Using the algorithm in~\cite{landau85}, this computation can be accomplished in deterministic polynomial time;
a Magma \cite{magma} script that implements this procedure is available at the author's website \cite{sut15}.

Unfortunately this is feasible only for very small $\ell$.
While $\Gal(K(E[\ell])/K)$ can be computed in time polynomial in $\ell$, the exponents involved are quite large; indeed, the necessary first step of factoring~$f_{E,\ell}(x)$ is already non-trivial, even when $K=\Q$.
For $\ell>2$ the polynomial $f_{E,\ell}$ has degree $(\ell^2-1)/2$ and coefficients with bit-size $O(\ell^2)$, which gives an $O(\ell^{12+o(1)})$ time for factoring $f_{E,\ell}\in \Z[x]$ using the best known bounds for polynomial factorization~\cite{schonhage84}.
More generally, the time to factor $f_{E,\ell}$ in $K[x]$ given in~\cite{landau85} is $O(\ell^{18+o(1)}[K\!:\!\Q]^{9+o(1)})$, and the time to compute its splitting field my be substantially larger.
By contrast, the Monte Carlo algorithm presented in this article computes~$G_E(\ell)$ up to local conjugacy (as defined below) in time that is quasi-linear in both $\ell$ and $[K\!:\!\Q]$; in fact it does this simultaneously for all primes in $S_E$ in time quasi-linear in $\max(S_E)$.

Two Galois representations $\rho_1,\rho_2\colon \Gal(\Kbar/K)\to \GL_2(m)$ are said to be \emph{locally conjugate} if $\rho_1(\sigma)$ and $\rho_2(\sigma)$ are conjugate in $\GL_2(m)$ for every $\sigma\in \Gal(\Kbar/K)$ (not necessarily by the same matrix in each case).
We call two subgroups $G$ and $H$ of $\GL_2(m)$ locally conjugate if there is a bijection of sets that maps each $g\in G$ to an element $h\in H$ that is conjugate to $g$ in $\GL_2(m)$; equivalently, $(\GL_2(m),G,H)$ is a (non-trivial) Gassmann-Sunada triple \cite{gassmann26,sunada85}.
Local conjugacy defines an equivalence relation on subgroups of $\GL_2(m)$.

We present two probabilistic algorithms to determine the exceptional primes~$\ell$ for a given elliptic curve $E/K$ and to determine the groups $G_E(\ell)$ up to local conjugacy.
The algorithms work by computing the images in $G_E(\ell)$ of Frobenius elements (conjugacy classes) $\Frob_\p$ for unramified primes $\p\ndiv\ell$ of $K$ where~$E$ has good reduction, either for all~$\p$ of bounded norm, or for randomly chosen~$\p$ with norms in a bounded interval.
This implies that our algorithms can only determine $G_E(\ell)$ up to local conjugacy, but we show that this imposes very strong constraints on $G_E(\ell)$.
In particular, we prove that every local conjugacy class of subgroups of $\GL_2(\ell)$ consists of at most two conjugacy classes of subgroups of $\GL_2(\ell)$ that are isomorphic as abstract groups and have the same semisimplification.
Moreover, we prove that whenever $G_E(\ell)$ is locally conjugate to a subgroup~$G'$ of $\GL_2(\ell)$, there is an isogenous elliptic curve $E'/K$ for which $G_{E'}(\ell)=G'$; see Theorem~\ref{thm:locconjisog}.
We also describe some global methods for efficiently distinguishing pairs of locally conjugate but non-conjugate Galois images that are applicable in most (but not all) cases, including every case that we encountered in our computations; see Section~\ref{sec:distinguish}.

To compute the conjugacy class $\rho_{E,m}(\Frob_\p)$ for unramified primes $\p$ of $K$ that do not divide $m$ we rely on three fundamental algorithms for elliptic curves over finite fields that we apply to the reduction $E_\p/\F_p$ of $E$ modulo~$\p$; here $\F_\p:=\O_K/\p$ is the residue field, a finite field with $q:=N(\p)$ elements.
The first is Schoof's algorithm \cite{schoof85,schoof95}, which computes the trace $t\in\Z$ of the Frobenius endomorphism in time polynomial in $\log q$.
The second is a Las Vegas algorithm to compute the endomorphism ring $\End(E_\p)$ when $E_\p$ is ordinary, due to Bisson and the author \cite{bisson11,bs11}; under the GRH its expected running time is subexponential in $\log q$.
It follows from a theorem of Duke and T\'oth \cite{dt02} that the pair ~$(t,\End(E_\p))$ determines an integer matrix~$A_\p$ whose reduction modulo $m$ lies in the conjugacy class $\rho_{E,m}(\Frob_\p)$ for every positive integer~$m$.
The third is Miller's algorithm to compute the Weil pairing \cite{miller04}, which we use to compute the rank of the $\ell$-torsion subgroup of $E_\p(\F_\p)$ in quasi-cubic time.
This allows us to determine the dimension of the 1-eigenspace of $\rho_{E,\ell}(\Frob_\p)$ without computing~$A_\p$, providing an efficient method to distinguish unipotent elements of~$G_E(\ell)$, which are not distinguished by their characteristic polynomials.

In order to bound the norms of the primes $\p$ that we use, we rely on explicit Chebotarev bounds that depend on the GRH.
In principle our algorithms can be implemented so that they do not rely on this hypothesis, but the running times would increase exponentially.
The GRH also gives us bounds on the largest exceptional prime $\ell$ that can occur for a given elliptic curve $E/K$; the results of Larson and Vaintrob~\cite{lv14} give bounds that are quasi-linear in $\log N_E$, where $N_E$ is the absolute value of the norm of the conductor of~$E$.
Together these allow us to bound the norms of the primes $\p$ that we must consider by a polynomial in $\log \Vert f\Vert$, where $\Vert f\Vert$ denotes the maximum of the absolute values of the norms of the coefficients appearing in an integral Weierstrass equation $y^2=f(x)$ for~$E$.

We now state our two main results.
The first is a Las Vegas algorithm that, given an elliptic curve $E/K$ specified by an integral Weierstrass equation, outputs a complete list $S_E$ of the primes $\ell$ for which $G_E(\ell)\ne\GL_2(\ell)$ and for each $\ell\in S_E$ a subgroup of $\GL_2(\ell)$ specified by generators that is locally conjugate to $G_E(\ell)$; see Algorithm~\ref{alg:lvfull}.
Under the GRH its expected running time is bounded by
\[
(\log\Vert f\Vert)^{11+o(1)}
\]
in Theorem~\ref{thm:lvfull}.
Our second main result is a Monte Carlo algorithm that has the same output as our Las Vegas algorithm and is correct with probability at least 2/3; see Algorithm~\ref{alg:mcfull}.  Its error is one-sided in the sense that each subgroup of $\GL_2(\ell)$ output by the algorithm is guaranteed to be locally conjugate to a subgroup of $G_E(\ell)$, but it may be a proper subgroup.
By running the algorithm repeatedly the error probability can be made arbitrarily small.
Under the GRH its running time is bounded by
\[
(\log \Vert f\Vert)^{1+o(1)},
\]
which is quasi-linear in the size of the input, the equation $y^2=f(x)$; see Theorem~\ref{thm:mcfull}.

An essential ingredient to both of our algorithms is the ability to distinguish and explicitly construct subgroups of $\GL_2(\ell)$ based on a compact representation of a subset of their element conjugacy classes.
The classification of the possible images of subgroups of $\GL_2(\ell)$ in $\PGL_2(\ell)$ has long been known \cite{dickson01}, but for our we work we require a complete list of the subgroups of $\GL_2(\ell)$ up to conjugacy, and a precise understanding of the element conjugacy classes each contains.
We address these questions in Section~\ref{sec:GL2}, in which we obtain exact formulas for the number of subgroups of $\GL_2(\ell)$ up to conjugacy (and for subgroups of various types) that may be of independent interest.
We also give a quasi-linear time algorithm to enumerate these subgroups with explicit generators for each; see Algorithm~\ref{alg:enum}.

We have applied our algorithms to various databases of elliptic curves over~$\Q$, including all non-CM curves of conductor up to $350,000$ listed in Cremona's tables \cite{cremona}, and the non-CM curves in the Stein-Watkins database \cite{sw02}, which includes a large proportion of all elliptic curves over $\Q$ of conductor up to $10^8$, and of prime conductor up to~$10^{10}$; some 140 million elliptic curves in all.
We also analyzed parameterized families of elliptic curves that are known to have exceptional Galois images, and large families of elliptic curves of bounded height (more than $10^9$ curves).
In each case we were able to compute a complete list $S_E$ of the exceptional primes $\ell$ and the subgroups $G_E(\ell)$ up to conjugacy (not just local conjugacy), using the methods described in \S\ref{sec:distinguish}.
This work yields a conjecturally complete list of 63 exceptional subgroup conjugacy classes that arise as $G_E(\ell)$ for some non-CM elliptic curve $E/\Q$ and prime~$\ell$; these are listed in Tables~\ref{table:Qpart1} and~\ref{table:Qpart2} of Section \ref{sec:results}.
Thanks to recent work by Zywina~\cite{zywina15}, we have been able to independently verify our results for all the non-CM elliptic curves in Cremona's tables, and in every case we found that the output of our Monte Carlo algorithm (which we executed repeatedly in order to amplify its success probability) was correct.
This motivates the following conjecture:

\begin{conjecture}\label{conj:Qexcep}
Let $E/\Q$ be an elliptic curve without complex multiplication and let $\ell$ be a prime.
Then $G_E(\ell)$ is either equal to $\GL_2(\ell)$ or conjugate to one of the $63$ groups listed in Tables~\ref{table:Qpart1} and~\ref{table:Qpart2}.
\end{conjecture}

Under this conjecture we determine a complete list of 160 exceptional Galois images $G_E(\ell)$ not containing $\SL_2(\ell)$ that arise for non-CM elliptic curves with rational $j$-invariants over quadratic fields; these include the 63 groups that already arise over~$\Q$ along with~68 new groups that arise for base changes of elliptic curves over $\Q$, and 29 that arise for quadratic twists of these curves but not for any base change from~$\Q$; see Theorem~\ref{thm:basechange} and Tables~\ref{table:Qbasechange1}-\ref{table:Qbasechangetwist2}.
A key ingredient to this result is an analysis of how $G_{E^F}(\ell)$ varies within a family of quadratic twists $E^F$ of a fixed elliptic curve $E/K$ as $F$ varies over quadratic extensions of $K$; this appears in \S\ref{sec:twists}.
We find that for any odd prime~$\ell$, up to~3 non-conjugate groups $G_{E^F}(\ell)$ may arise in such a family and we give an explicit method to determine quadratic extensions $F/K$ that realize every possibility.

We have also run our algorithms on tables of elliptic curves defined over quadratic fields that have recently been made available in the $L$-functions and Modular Forms Database (LMFDB) \cite{lmfdbbeta}, including the five real quadratic fields and five imaginary quadratic fields of least absolute discriminant.
Examples of exceptional Galois images $G_E(\ell)$ that occur only for non-CM elliptic curves with irrational $j$-invariants over these fields are listed in the tables at the end of Section~\ref{sec:results}, as well as examples over the cubic field of discriminant $-23$.

In principle our algorithms can also be used to determine $G_E(m)$ up to local conjugacy for any positive integer $m$, but the situation is more complicated when~$m$ is composite for three reasons: (1) local conjugacy imposes fewer constraints when $m$ is composite, for example locally conjugate subgroups of $\GL_2(m)$ need not be isomorphic; (2) the integers $m$ for which $G_E(m)$ is exceptional and not the full inverse image of  $G_E(m')$ for some $m'|m$ may be exponentially larger than the largest exceptional prime $\ell$; (3) our understanding of the subgroup structure of $\GL_2(m)$ is less refined that it is for $\GL_2(\ell)$.
In spite of these obstacles, it is entirely feasible to apply our algorithms when $m$ is small, and if we set the more modest goal of simply computing the index of $G_E(m)$ in $\GL_2(m)$, this can be done quite efficiently.
This suggests a practical method for computing $m_E$ and the index of $G_E$ in $\GL_2(\hat\Z)$ for a non-CM elliptic curve $E/K$ that we plan to address in a future article.

\subsection{Acknowledgements}
The author is grateful to Nicholas Katz for his support, and for asking the questions that motivated this research, and to David Zywina for several fruitful discussions.  The author also wishes to thank John Cremona and William Stein for their assistance with many of the computations.

\section{Notation and Terminology}\label{sec:notation}
Throughout this article the symbols $\ell$ and $p$ denote rational primes, and $r$, $m$ and $n$ denote positive integers.
We use $\tau(n)$ to denote the number of positive divisors of an integer $n$ and $\phi(n):=\#(\Z/n\Z)^\times$ for the Euler function.
For any prime power $q$, we use $\Fq$ to denote the field with $q$ elements.
For sets $S$ and $T$ we write $S-T$ for the set of elements of $S$ that do not lie in $T$.

For any ring $R$, we use $\mat_r(R)$, to denote the ring of $r\times r$ matrices, $\GL_r(R)$ for its multiplicative subgroup of invertible matrices, $\SL_r(R)$ for the kernel of the determinant map $\det\colon \GL_r(R)\to \GL_1(R)$, and $\PGL_r(R)$ for the quotient of $\GL_r(R)$ by its center.
For each integer $m>1$ we define the notations
\begin{align*}
\Z(m):=\Z/m\Z,\qquad\qquad\qquad \mat_r(m):=\mat_r(\Z/m\Z),\\
\SL_r(m):=\SL_r(\Z/m\Z),\qquad \GL_r(m):=\GL_r(\Z/m\Z),\qquad \PGL_r(m)&:=\PGL_r(\Z/m\Z).
\end{align*}
The center of $\GL_2(m)$ consists of the subgroup of scalar matrices $\diagmat{z}{z}$, which we denote $Z(m)$; when there is no risk of ambiguity we may identify $Z(m)\simeq \Z(m)^\times$ and use $z$ to denote $\diagmat{z}{z}$.
The scalar matrices form the kernel of the canonical projection
\[
\pi\colon \GL_2(m)\twoheadrightarrow \PGL_2(m)
\]
which we denote by $\pi$ throughout.

In our identification of $\Aut(E[m])$ with $\GL_2(m)$, we view elements of $\GL_2(m)$ as $2\times 2$ matrices acting on column vectors by multiplication on the \textbf{left}, and distinguish subgroups of $\GL_2(m)$ only up to conjugacy.
For an elliptic curve $E$ over a number field $K$, composing the $2$-dimensional representation
\[
\rho_E\colon \Gal(\Kbar/K)\to \GL_2(\Zhat)
\]
with the determinant map $\GL_2(\Zhat)\to\Zhat^\times$ induces a $1$-dimensional representation 
\[
\det\circ\, \rho_E\colon \Gal(\Kbar/K)\to \GL_1(\Zhat)=\Zhat^\times.
\]

Throughout this article we use $\p$ to denote a prime of $K$, by which we mean a nonzero prime ideal in its ring of integers $\O_K$, and we use $\F_\p$ to denote the residue field $\O_K/\p$.
For each prime $\p\ndiv m$ that is unramified in $K(E[m])/K$ (all but finitely many $\p$), the value of $\det\circ\rho_E$ on the Frobenius element $\Frob_\p$ (which we recall is defined only up to conjugacy) is $N(\p):=[\O_K:\p]$.
The image of $\det\circ\rho_E$ thus depends only on $K$, not on $E$; in fact, it depends only on the intersection of~$K$ with the maximal cyclotomic extension $\Q^{\rm cyc}$ of $\Q$ in $\Kbar$, and $\det\circ\rho_{E,\ell}$ is surjective for all but finitely many~$\ell$.


Our complexity bounds always count bit operations.
We use $\M(n)$ to denote the time to multiply two $n$-bit integers, which we may bound by
\[
\M(n) = n(\log n)^{1+o(1)}
\]
via \cite{ss71}; see \cite{hhl14} for a more precise bound.
This bound implies that arithmetic operations in finite fields $\Fq$ can be performed in $(\log q)^{1+o(1)}$ time, which we assume throughout (we refer the reader to \cite{gg13} for details).

Many of the algorithms we present are probabilistic algorithms, which we recall are typically classified as one of two types.
\emph{Las Vegas algorithms} produce output that is guaranteed to be correct but have potentially unbounded running times that may depend on probabilistic choices; for such algorithms we bound their expected running time, which is required to be finite.
\emph{Monte Carlo algorithms}, by contrast, have bounded running times but may produce outputs that are incorrect with probability bounded by some $c < 1/2$; we use $c=1/3$.
Assuming the correct output is unique, by running a Monte Carlo algorithm repeatedly and choosing the output produced most frequently, the probability of error can be made arbitrarily close to zero at a rate exponential in the number of repetitions.

For integers $n\ge 3$, we use $\alt{n}$ and $\sym{n}$ to denote the alternating and symmetric groups on $n$ elements, respectively.
For the purpose of this article we consider the non-cyclic group of order 4 (the Klein group) to be a dihedral group.

\section{Subgroups of \texorpdfstring{$\GL_2(\F_\ell)$}{GL(2,Z/\textit{l}Z)}}\label{sec:GL2}

The classification of subgroups of $\PGL_2(\ell)$ is well known (see Proposition~\ref{prop:subgroups} below).
Our algorithms require a more refined classification of the subgroups of $\GL_2(\ell)$, up to conjugacy in $\GL_2(\ell)$, that allows us to distinguish subgroups by sampling element conjugacy classes corresponding to Frobenius elements.
In this section we obtain such a classification, as well as explicit formulas to count subgroups of $\GL_2(\ell)$ up to conjugacy and an efficient algorithm to enumerate them.
Many of the proofs in this section are elementary, but as our algorithms depend crucially on these results, we give at least a sketch of each proof.
Except when the case $\ell=2$ is specifically noted, we assume throughout this section that $\ell$ is an odd prime.

For any $g\in\GL_2(\ell)$ we define the discriminant
\[
\Delta(g):=\tr(g)^2-4\det(g)\in\Z(\ell),
\]
and the Legendre symbol
\[
\chi(g):=\left(\frac{\Delta(g)}{\ell}\right)\in \{-1,0,1\}.
\]

For ease of reference we list the element conjugacy classes of $\GL_2(\ell)$ in Table~\ref{table:conjclasses}.
Here and throughout,~$\varepsilon$ denotes a fixed non-square element of $\Z(\ell)^\times$; for the sake of concreteness, let $\varepsilon$ be the least positive integer that generates $\Z(\ell)^\times$.
We note that $\diagmat{x}{y}$ and $\diagmat{y}{x}$ are conjugate via $\smallmat{0}{1}{1}{0}$, and $\smallmat{x}{\varepsilon y}{y}{x}$ and $\smallmat{x}{-\varepsilon y}{-y}{x}$ are conjugate via $\diagmat{1}{-1}$, which explains the restrictions on $y$ in Table~\ref{table:conjclasses} below.
\vspace{-8pt}

\begin{table}[htb]
\setlength{\extrarowheight}{4pt}
\setlength{\tabcolsep}{10pt}
\begin{tabular}{llcccrl}
representative & size & number & $\det$ & $\tr$ & $\chi$ & order \\\hline
$\diagmat{x}{x}\quad 0<x<\ell$ & 1 &$ \ell-1$ & $x^2$ & $2x$ & 0 & divides $\ell-1$\\
$\smallmat{x}{1}{0}{x}\quad 0<x<\ell$ & $\ell^2-1$ &$ \ell-1$ & $x^2$ & $2x$ &  0 & divisible by $\ell$\\
$\diagmat{x}{y}$\hspace{10pt}$0<x< y<\ell$ & $\ell^2+\ell$ &$\binom{\ell-1}{2}$ & $xy$ & $x+y$& $+1$ & divides $\ell-1$\\\vspace{2pt}
$\smallmat{x}{\varepsilon y}{y}{x}\ \ 0<y\le\frac{\ell-1}{2}$ & $\ell^2-\ell$ & $\binom{\ell}{2}$ & $x^2-\varepsilon y^2$ & $2x$ & $-1$ & divides $\ell^2-1$\\\hline
\end{tabular}
\vspace{8pt}

\caption{Element conjugacy classes in $\GL_2(\ell)$ for primes $\ell >2$.}\label{table:conjclasses}
\end{table}
\vspace{-12pt}

For any $g\in \GL_2(\ell)$ and positive integer $n$, the trace of $g^n$ can be computed as $\tr g^n = a_n$, where $a_n$ is defined by the recurrence:
\begin{equation}\label{eq:trrec}
a_0:=2,\qquad a_1:=\tr(g),\qquad a_{n+2}:=a_1a_{n+1}-a_n\det g.
\end{equation}
This implies that for elements $g$ whose order $|g|$ is not divisible by $\ell$, we can derive $|g|$ from $(\det g,\tr g)$.
We are also interested in the order of the image of $g$ in $\PGL_2(\ell)$.
For this purpose we define
\[
u(g) := \frac{\tr(g)^2}{\det(g)}\in \Z(\ell).
\]
If $|g|$ is divisible by $\ell$ then $g$ is conjugate to some $\smallmat{x}{1}{0}{x}$ and $u(g)=4$.
Otherwise, the order $r$ of $\pi(g)$ in $\PGL_2(\ell)$ is prime to $\ell$ and we have
\begin{equation}\label{eq:projord}
u(g)=\zeta_r+\zeta_r^{-1}+2, 
\end{equation}
for some primitive $r$th root of unity for which $\zeta_r+\zeta_r^{-1}\in \F_\ell^\times$, as explained in \cite[p.~190]{lang76}.
Note that $\zeta_r$ may lie in a quadratic extension $\F_\ell$, but in any case $r$ divides either $\ell-1$ or $\ell+1$ and is uniquely determined by~$u(g)$; this allows $|\pi(g)|=r$ to be unambiguously determined from $u(g)$, and hence from $(\det g,\tr g)$ whenever $|g|$ is prime to $\ell$.
This implies, in particular, that the elements of $\GL_2(\ell)$ that have order $2$ in $\PGL_2(\ell)$ are precisely the elements of trace zero.

For each odd prime $\ell$ we define
the \emph{split Cartan group} $C_s(\ell)$ and \emph{non-split Cartan group} $C_{ns}(\ell)$ by
\begin{align*}
C_s(\ell) &:= \left\{\smallmat{x}{0}{0}{y}: xy\ne 0\right\}\subseteq\GL_2(\ell),\\
C_{ns}(\ell) &:=\left\{\smallmat{x}{\varepsilon y}{y}{x}:(x,y)\ne (0,0)\right\}\subseteq\GL_2(\ell),
\end{align*}
and note that $C_s(\ell)\simeq \F_\ell^\times\times\F_\ell^\times$ and $C_{ns}(\ell)\simeq \F_{\ell^2}^\times$.
Both $C_s(\ell)$ and $C_{ns}(\ell)$ have index 2 in their normalizers
\[
C_s^+(\ell):= C_s(\ell)\cup\smallmat{0}{1}{1}{0}C_s, \qquad C_{ns}^+(\ell):= C_{ns}\cup\diagmat{1}{-1}C_{ns}(\ell).
\]
Elements of $C_s^+(\ell)$ that are conjugate in $\GL_2(\ell)$ are conjugate in $C_s^+(\ell)$, and similarly for $C_{ns}^+(\ell)$.
We define $C_s(2)$ as the trivial group, and $C_{ns}(2)$ as the kernel of the sign homomorphism $\GL_2(2)\simeq\sym{3}\twoheadrightarrow\{\pm 1\}$; both are normal in $\GL_2(2)$.

We refer to the conjugates of $C_s(\ell)$ and $C_{ns}(\ell)$ in $\GL_2(\ell)$ as  \emph{split} and \emph{non-split Cartan groups}, respectively.
For $\ell>2$ all elements in the non-trivial coset of a Cartan group in its normalizer have trace zero, and the square of such an element~$g$ is the scalar matrix $\diagmat{z}{z}$, where $z=-\det g$.

The \emph{Borel group} $B(\ell)\subseteq\GL_2(\ell)$ is the subgroup of upper triangular matrices; we refer to its conjugates in $\GL_2(\F_\ell)$ as \emph{Borel groups}.
For $\ell>2$, the group $B(\ell)$ is nonabelian, and its commutator subgroup $B(\ell)'$ is the cyclic group of order $\ell$ generated by $\unimat$.
The split Cartan subgroup $C_s(\ell)$ is contained in $B(\ell)$,  and for $\ell>2$ it is isomorphic to the abelian quotient $B(\ell)/[B(\ell),B(\ell)]$.
We also note that
\[
Z(\ell)=C_s(\ell)\cap C_{ns}(\ell)\subseteq B(\ell).
\]

We now recall the classification of subgroups of $\GL_2(\ell)$ in terms of their images in $\PGL_2(\ell)$, originally due to Dickson \cite{dickson01}.

\begin{proposition}\label{prop:subgroups}
Let $\ell$ be an odd prime and let $G$ be a subgroup of $\GL_2(\ell)$ with image $H$ in $\PGL_2(\ell)$.
If~$G$ contains an element of order $\ell$ then $G\subseteq B(\ell)$  or $\SL_2(\ell)\subseteq G$.
Otherwise, one of the following holds:
\begin{enumerate}
\setlength{\itemsep}{0pt}
\item[{\rm (1)}] $H$ is cyclic and $G$ lies in a Cartan group;
\item[{\rm (2)}] $H$ is dihedral and $G$ lies in the normalizer of a Cartan group, but not in any Cartan group;
\item[{\rm (3)}] $H$ is isomorphic to $\alt{4}$, $\sym{4}$, or $\alt{5}$ and $G$ is not contained in the normalizer of any Cartan group.
\end{enumerate}
\end{proposition}
\begin{proof}
See \cite[Lem.\ 2]{sd72} or \cite[\S 2]{serre72}.
\end{proof}

\begin{remark}
In the exceptional case (3), if $G$ contains an element whose determinant is not a square, then~$H$ contains a subgroup of index 2, which rules out $H\simeq \alt{4}$ and $H\simeq \alt{5}$.
This applies when $G=G_E(\ell)$ arises from an elliptic curve $E$ over a number field $K$ that does not contain the quadratic subfield of the cyclotomic field $\Q(\zeta_\ell)$, which includes $K=\Q$.
\end{remark}

\subsection{Borel cases}\label{sec:borel}

In this section we address subgroups of the Borel group $B(\ell)$ that contain an element of order $\ell$ (hence do not lie in $C_s(\ell)$), where $\ell$ is an odd prime.

\begin{lemma}\label{lem:borel}
Let $\ell$ be an odd prime and let $G$ be a subgroup of $B(\ell)$ that contains an element of order $\ell$.
Then $G$ contains $t=\unimat$ and is equal to the internal semidirect product
\[
G = \langle t\rangle\rtimes \bigl(G\cap C_s(\ell)\bigr),
\]
which is a direct product if and only if $G\cap C_s(\ell) \subseteq Z(\ell)$.
\end{lemma}
\begin{proof}
If $G$ contains an element $g=\smallmat{a}{b}{0}{d}$ of order divisible $\ell$ then $g^{\ell-1}=\smallmat{1}{x}{0}{1}$ for some nonzero $x$, and for $ex\equiv 1\bmod \ell$ we have $g^{e\ell-e}=t\in G$.
For any $g=\smallmat{a}{b}{0}{d}$ the product
$gt^e = \smallmat{a}{ae+d}{0}{d}$
is diagonal if and only if $e\equiv -d/a\bmod \ell$.
Thus every coset of $\langle t\rangle$ in $G$ contains a unique element of $H=G\cap C_s(\ell)$.
Thus $G=\langle t\rangle \rtimes H$, since $\langle t\rangle$ is normal in $G$, and the action of $H$ on $\langle t\rangle$ is trivial if and only if $H\subseteq Z(\ell)$.
\end{proof}

Formulas to count subgroups of a given finite abelian group are well known; see \cite{birkhoff35}, for example.
In the case of interest here, the answer is particularly simple.
The lemma below is a special case of \cite[Thm.\ 4.1]{toth14}.

\begin{lemma}\label{lem:alpha}
Let $n$ be a positive integer.
There is a one-to-one correspondence between triples $(a,b,i)$ with $a,b|n$, $0\le i<\gcd(a,b)$ and subgroups of $\Z(n)\times \Z(n)$ given by
\[
(a,b,i)\mapsto \bigl\langle (a,-a),\, (ic,d-ic)\bigr\rangle,
\]
where $c=a/\gcd(a,b)$ and $d=n/b$.  The number of distinct subgroups of $\Z(n)\times\Z(n)$ is thus
\[
\alpha(n) := \sum_{a,b|n}\gcd(a,b).
\]
\end{lemma}
\begin{proof}
For each subgroup of $H\subseteq\Z(n)\times\Z(n)$ there is a triple $(a,b,i)$ with $a,b|n$ and $0\le i < \gcd(a,b)$ determined by the generator $x=(a,-a)$ of the trace zero subgroup $H_0\subseteq H$, the order $b$ of $H/H_0$, and the least $i\ge 0$ for which $y:=(ia/\gcd(a,b),\ n/b-  ia/\gcd(a,b))\in H$.
Conversely, each such triple $(a,b,i)$ determines a subgroup $H=\langle x,y\rangle$; we thus have a bijection and $\alpha(n)$ counts the triples $(a,b,i)$.
\end{proof}

\begin{corollary}\label{cor:borel}
Let $\ell$ be an odd prime.
The number of non-conjugate subgroups of $\GL_2(\ell)$ that lie in a Borel group and contain an element of order $\ell$ is $\alpha(\ell-1)$, the number of subgroups of $C_s(\ell)$.
\end{corollary}
\begin{proof}
It suffices to consider subgroups $G\subseteq B(\ell)$ containing $\unimat$ up to conjugation in $B(\ell)$, since $B(\ell)$ is self-normalizing in $\GL_2(\ell)$.
Every such $G$ is normal in $B(\ell)$: the subgroup generated by $\unimat$ is normal and the $B(\ell)$-conjugates of $G\cap C_s(\ell)$ all lie in $G$.
This gives a one-to-one correspondence between subgroups of $B(\ell)$ containing $\unimat$, all of which are non-conjugate, and subgroups of $C_s(\ell)\simeq \F_\ell^\times\times\F_\ell^\times\simeq  \Z(\ell-1)\times\Z(\ell-1)$.
\end{proof}

\begin{lemma}\label{lem:isoborel}
Let $\ell$ be an odd prime.
Let $G$ and $H$ be conjugate subgroups of $\GL_2(\ell)$ that lie in $C_s(\ell)$ and let $t=\unimat$.
The groups $G':=\langle G,t\rangle$ and $H':=\langle H,t\rangle$ of $B(\ell)$ are locally conjugate in $\GL_2(\ell)$ and isomorphic.
\end{lemma}
\begin{proof}
When $G=H$ the lemma clearly holds, so we assume $G\ne H$, in which case $G$ and $H$ are conjugate via $s=\smallmat{0}{1}{1}{0}$.
We have $G'=\langle t\rangle\rtimes G$ and $H'=\langle t\rangle\rtimes H$, by Lemma~\ref{lem:borel},
and the bijection from $G'$ to $H'$ given by swapping diagonal entries preserves conjugacy classes (but is typically not a homomorphism), hence $G'$ and $H'$ are locally conjugate.
Let $g\in G$ be an element with maximal projective order $e$, and let $z\in G$ be a generator for $G\cap Z(\ell)$ with order $f$; then $g^e=z^d$ for some integer $d\in [1,\ell-1]$.
We have $gtg^{-1}=t^n$, where $n\in (\Z/\ell\Z)^\times$ is the ratio of the diagonal entries of $g$, while $z$ commutes with $t$ and $g$.  Thus $G'$ is isomorphic to the abstract group
\[
\mathcal{G}:=\langle t,g,z:t^\ell=g^ez^{-d}=z^f=ztz^{-1}t^{-1}=zgz^{-1}g^{-1}=gtg^{-1}t^{-n}=1\rangle.
\]
We now note that $t$ lies in $H'$, and $z$ generates $H'\cap Z(\ell)$.
The element $h=sgs$ of~$H$ has maximal projective order~$e$, with $h^e=z^d$, and $hth^{-1}=t^{1/n}$, where $1/n$ is the inverse of $n$ and has order $e$ in $(\Z/\ell\Z)^\times$.
The action of $h'=h^{e-1}$ on $t$ is thus identical to that of $g$, and there exists a $z'\in H'\cap Z(\ell)$ of the same order $f$ as $z$ for which $(h')^e = (z')^d$.
It follows that $H'$ is also isomorphic to $\mathcal{G}$.
\end{proof}

\begin{remark}\label{rem:onlycase}
The situation in Lemma~\ref{lem:isoborel} is the only case where non-conjugate but locally conjugate subgroups can arise; see Corollary~\ref{cor:sigzlc}.
\end{remark}

\subsection{Cyclic cases}\label{sec:cyclic}

We now consider the subgroups of $\GL_2(\ell)$ with cyclic image in $\PGL_2(\ell)$.

\begin{lemma}\label{lem:beta}
Let $n=\prod_p p^{e_p}$ be a positive integer.
The number of subgroups of $\Z(n)\times\Z(n)$ that are fixed by the automorphism $\sigma\colon (x,y)\mapsto (y,x)$ is
\[
\beta(n):=\beta_2(n)\prod_{p \ne 2}(e_p+1)^2,
\]
where $\beta_2(n)=2(e_2^2-e_2)+3$ if $n$ even and $\beta_2(n)=1$ if $n$ is odd.
\end{lemma}
\begin{proof}
Let $G$ be a subgroup of $\Z(n)\times\Z(n)$ fixed by $\sigma$.
The automorphism $\sigma$ fixes each $p$-Sylow subgroup of $G$, so it suffices to consider the case $\#G=p^e$.
The map $\varphi$ defined by $g\mapsto \sigma(g)-g$ is an endomorphism of $G$ with kernel $D:=\{(x,y)\in G:x=y\}$ and image contained in $T:=\{(x,y)\in G: x+y=0\}$.

If $p$ is odd then $\varphi(G)=T$ and $D\cap T$ is trivial, so $G=D\times T$.
Conversely, every product of a diagonal and trace zero subgroup of $\Z(n)\times \Z(n)$ is fixed by $\sigma$,
and there are $(e_p+1)^2$ such subgroups.

For $p=2$ we have $\beta(2)=3$, and $\beta(2^{n+1})=\beta(2^n)+4n$, where the $4n$ new groups all have exponent $2^{n+1}$: one is the full group, one is the even trace subgroup of index 2, two are index 4 subgroups $\langle (1,1),(0,4)\rangle$ and $\langle(1,-1),(0,4)\rangle$, and four are subgroups of index $2^i$, for $i$ from $3$ to $n+1$, of the form $\langle (1,\pm1),(0,2^i)\rangle$,\ $\langle (1,2^{i-1}\pm 1),(0,2^i)\rangle$.  The formula for $\beta_2(n)$ then follows by induction.
\end{proof}

\begin{remark}\label{rem:alphabeta}
In terms of the bijection given by Lemma~\ref{lem:alpha}, two triples $(a,b,i)$ and $(a,b,j)$ correspond to subgroups in the same $\sigma$-orbit if and only if $i+j\equiv (n/\lcm(a,b))\bmod \gcd(a,b)$; in particular, the subgroups fixed by $\sigma$ are those corresponding to triples $(a,b,i)$ with $2i\lcm(a,b)\equiv n\bmod \gcd(a,b)$.
\end{remark}

\begin{corollary}\label{cor:splitnormal}
Let $\ell$ be an odd prime.
There are $\beta(\ell-1)$ subgroups $H$ of $C_s(\ell)$ that are normal in $C_s^+(\ell)$.
\end{corollary}
\begin{proof}
The split Cartan group $C_s(\ell)\simeq\Z(\ell-1)\times\Z(\ell-1)$ is abelian and has index 2 in its normalizer $C_s^+(\ell)=\langle C_s(\ell),s\rangle$ where $s=\smallmat{0}{1}{1}{0}$.
It follows that a subgroup $H$ of $C_s(\ell)$ is normal in $C_s^+(\ell)$ if and only if it is fixed under conjugation by $s$, which acts on $H$ by swapping the diagonal entries of each element.
\end{proof}

\begin{corollary}\label{cor:splitcyclic}
Let $\ell$ be an odd prime.
The number of non-conjugate subgroups of $\GL_2(\ell)$ that lie in a split Cartan group is
\[
\frac{\alpha(\ell-1)+\beta(\ell-1)}{2}.
\]
\end{corollary}
\begin{proof}
It suffices to count $\GL_2(\ell)$-conjugacy classes of subgroups of $C_s(\ell)$, and it is enough to consider $C_s^+(\ell)$ conjugacy classes, since $C_s^+(\ell)$ is the normalizer of $C_s(\ell)$.
The orbit of each subgroup $G\subseteq C_s(\ell)$ under conjugation by $C_s^+(\ell)$ has order 1 or 2, depending on whether $G$ is fixed by the action of $\smallmat{0}{1}{1}{0}$, which swaps the diagonal entries.
The counting formula then follows from Corollary~\ref{lem:alpha} and Lemma~\ref{lem:beta}.
\end{proof}

\begin{lemma}\label{lem:nonsplitcyclic}
Let $\ell$ be an odd prime.
The number of non-conjugate subgroups of $\GL_2(\ell)$ that lie in a non-split Cartan group is $\tau(\ell^2-1)$, where $\tau(n)$ counts the positive divisors of $n$.
\end{lemma}
\begin{proof}
This is clear: $C_{ns}(\ell)\simeq\F_{\ell^2}^\times$ is cyclic of order $\ell^2-1$ and therefore contains a subgroup of order $n$ for each divisor $n$ of $\ell^2-1$, none of which are conjugate.
\end{proof}

\subsection{Dihedral cases}\label{sec:dihedral}

Next up are the subgroups of $\GL_2(\ell)$ with dihedral image in $\PGL_2(\ell)$; as above we assume that $\ell$ is an odd prime and recall that we consider the Klein group to be dihedral.

If $G$ is a subgroup of $\GL_2(\ell)$ with dihedral image in $\PGL_2(\ell)$, then $G$ lies in the normalizer $C^+$ of a Cartan group $C$ and it contains the abelian subgroup $H=G\cap C$ with index $2$. Let $Z=G\cap Z(\ell)\subseteq H$ denote the scalar subgroup of $G$.
The subgroup $H$ is normal in $G$ and in $C$, hence in $C^+=GC$, and it follows that each non-scalar element $h$ of $H$ has a distinct conjugate $\bar h\in H$; indeed, $\bar h=\smallmat{0}{1}{1}{0}h\smallmat{0}{1}{1}{0}$ if $C=C_s(\ell)$ and $\bar h=\smallmat{1}{0}{0}{-1}h\smallmat{1}{0}{0}{-1}$ if $C=C_{ns}(\ell)$).
To better understand the relationship between $G$ and~$H$ we consider the following maps:
\medskip

\begin{tabular}{lll}
$\qquad\qquad\quad\ \ H \to Z$ &  & $\qquad\qquad\qquad\ \ \ (G-H)\to Z$\\
$\qquad\qquad\qquad h\mapsto h\bar h =\diagmat{\det h}{\det h}$ & & $\qquad\qquad\qquad\qquad\qquad g\mapsto g^2=\diagmat{-\det g}{-\det g}.$
\end{tabular}
\medskip

\noindent
There are two possibilities, depending on whether the set $\det(H)$ and the set
\[
-\det(G-H):=\{-\det g:g \in G-H\}\subseteq \GL_1(\ell)
\]
coincide or not.

\begin{lemma}\label{lem:dihedral}
Let $\ell$ be an odd prime.
Let $G$ be a subgroup of $\GL_2(\ell)$ with dihedral image in $\PGL_2(\ell)$ that lies in the normalizer $C^+$ of a Cartan group $C$, let $H=G\cap C$, and let $Z=G\cap Z(\ell)$.
Then $H$ is normal in $C^+$ and one of the following holds:
\begin{enumerate}
\item[{\rm (2a)}] $\det(H)$ and $-\det(G-H)$ coincide, in which case $G=\langle H,\gamma\rangle$ for some $\gamma\in G-H$ with $\det\gamma=-1$.
\item[{\rm (2b)}] $\det(H)$ and $-\det(G-H)$ are disjoint, in which case $\det(H)=\det(Z)$ and $H$ contains $-1$.
\end{enumerate}
If $G'$ is another subgroup of $C^+$ with dihedral image in $\PGL_2(\ell)$ with $H=G'\cap C$ and $-\det(G'-H)=-\det(G-H)$, then $G$ and $G'$ are conjugate in~$C^+$.
\end{lemma}
\begin{proof}
We have $[G:H]=2$, so $H$ is normal in $G$, and its normalizer in $C^+$ contains the abelian group~$C$ and is therefore equal to $C^+$; so $H$ is normal in $C^+$.

If $\det(H)$ and $-\det(G-H)$ intersect then we may pick $g\in G-H$ and $h\in H$ so $\gamma=g/h\in G-H$ has $\det\gamma=-1$.
Then $G-H=\gamma H$ and $\det(H)=-\det(G-H)$.

Otherwise $\det(H)$ and $-\det(G-H)$ are disjoint.
The image of the map $H\to Z$ is then an even index subgroup of~$Z$, and its index is at most $2$, since the image of the subgroup $Z\subseteq H$ has index $2$.
It follows that $\det(H)=\det(Z)$ corresponds to an index $2$ subgroup of $Z$,  and since $Z$ has even order, it contains $-1$.

Now suppose $G'$ is another subgroup of $C^+$ with dihedral image in $\PGL_2(\ell)$ for which $H=G'\cap C$ and $-\det(G'-H)=-\det(G-H)$.
In case (2a) we have $G'=\langle H,\gamma'\rangle$ for some $\gamma'\in G-H$ with $\det\gamma'=-1$, and then $\gamma'$ is conjugate to $\gamma$ in~$C^+$, and therefore $G'=\langle H,\gamma'\rangle$ is conjugate to $G=\langle H,\gamma\rangle$.
In case (2b) the image of the map $g\mapsto g^2$ from $(G-H)\to Z$ is the non-trivial coset of $\im(h\mapsto h\bar h)$ in $Z$, thus we may pick $\diagmat{z}{z}\in Z$ that is the square of some $\gamma\in G-H$ with $\det\gamma=-z$.
There must then be a $\gamma'\in G'-H$ with $\det\gamma'=-z$ that is conjugate to $\gamma$ in $C^+$, and therefore $G'=\langle H,\gamma'\rangle$ is conjugate to $G=\langle H,\gamma\rangle$.
\end{proof}

\begin{remark}\label{rem:cc}
For an elliptic curve $E$ over a number field with a real embedding, the group $G_E(\ell)$ necessarily contains an element $\gamma$ with $\tr\gamma=0$ and $\det\gamma=-1$ corresponding to complex conjugation.
This implies that $G_E(\ell)\not\subseteq C_{ns}(\ell)$ for $\ell>2$ (although $G_E(\ell)\subseteq C_s(\ell)$ is possible).
Indeed, $C_{ns}(3)$ is the unique subgroup $G\subseteq\GL_2(3)$ with $\det(G)=\F_3^\times$ that does not arise for any elliptic curve $E/\Q$; the corresponding modular curve $X_{ns}(3)$ has genus zero but no non-cuspidal rational points.
\end{remark}

\begin{remark}
For composite $m$ and elliptic curves $E$ over a number field with a real embedding, the criterion that $G_E(m)$ contains an element $\gamma$ with $\tr\gamma=0$ and $\det\gamma=-1$ is necessary but not sufficient.
A stronger criterion is that~$\gamma$ must also fix an order-$m$ element of $\Z(m)\times \Z(m)$.
When $m$ is prime this is already implied by $\tr\gamma = 0$ and $\det\gamma=-1$, but not in general.
This explains why, for example, $G_E(4)\ne \bigl\langle \smallmat{1}{2}{2}{3},\diagmat{3}{3}\bigr\rangle$ for any elliptic curve $E/\Q$, even though this group contains an element $\gamma$ with $\tr\gamma=0$ and $\det\gamma=-1$.
As in the previous remark, the corresponding modular curve has genus $0$ but no non-cuspidal rational points.
More generally, the ten pointless conics noted in~\cite{rzb14} that are models of modular curves associated to subgroups of $\GL_2(2^n)$ lack rational points for this reason.
\end{remark}

The following lemma determines the cases in which $\GL_2(\ell)$-conjugate subgroups of the normalizer $C^+$ of a Cartan group $C$ have intersections with $C$ that are not $\GL_2(\ell)$-conjugate.
This can occur only when~$C$ is a split Cartan group with $\ell\equiv 1\bmod 4$ and the projective images of the subgroups are isomorphic to the Klein group of order 4.

\begin{lemma}\label{lem:kleinexcep}
Let $\ell$ be an odd prime.
Let $G_1$ and $G_2$ be $\GL_2(\ell)$-conjugate subgroups of the normalizer $C^+$ of a Cartan group $C$ with dihedral images in $\PGL_2(\ell)$ such that $H_1:=G_1\cap C$ and $H_2:= G_2\cap C$ are not conjugate in $\GL_2(\ell)$.
Then $C$ is a split Cartan group, $\ell\equiv 1\bmod 4$, $|\pi(G_1)|=|\pi(G_2)|=4$, $Z:=\bigl\langle\diagmat{z}{z}\bigr\rangle:= G_1\cap Z(\ell)$ contains $-1$ with $z=x^2$ square, and $\{H_1,H_2\}=\{\bigl\langle\diagmat{x}{-x}\bigr\rangle, \bigl\langle \diagmat{z}{z},\diagmat{1}{-1}\bigr\rangle\}$.
Conversely, whenever $\ell\equiv 1\bmod 4$ there is a pair of conjugate $G_1$ and $G_2$ as above for each scalar subgroup $\bigl\langle\diagmat{z}{z}\bigr\rangle$ that contains $-1$ with $z$ square.
\end{lemma}
\begin{proof}
Let $G_2=gG_1g^{-1}$, let $Z:=\bigl\langle \diagmat{z}{z}\bigr\rangle=G_1\cap Z(\ell)=G_2\cap Z(\ell)$, and choose $h_1\in H_1$ so that $H_1=\langle h_1,Z\rangle$.
The group $H_1$ is normal in $C$, and thus contains all the $\GL_2(\ell)$-conjugates of $h_1$ that lie in~$C$, none of which lie in $H_2$ (otherwise $H_1$ and $H_2$ would coincide).
Thus $\gamma_2:=gh_1g^{-1}$ lies in $G_2-H_2$, and therefore both $h_1$ and $\gamma_2$ have trace zero, and we can similarly choose $h_2\in H_2$ so that $\gamma_1:=g^{-1}h_2g$ lies in $G_1-H_1$.
We then have $G_1=\langle h_1,\gamma_1,Z\rangle$ and $G_2=\langle h_2,\gamma_2,Z\rangle$ with $h_1,h_2,\gamma_1,\gamma_2$ all elements of trace zero and order 2 in $\PGL_2(\ell)$, thus $\pi(G_1)$ and $\pi(G_2)$ are both isomorphic to the Klein group.
And~$Z$ must contain $-1=h_1\bar h_1^{-1}=h_2\bar h_2^{-1}$.

Since $H_1$ and $H_2$ are non-conjugate we must have $\det h_1\ne \det h_2$ (no matter which $h_1$ and $h_2$ we pick); thus one of them is cyclic, say $H_1$, and the other, $H_2$, is not. This rules out the non-split Cartan case, so we now assume $C=C_s(\ell)$.
We can assume $h_1^2$ generates $Z$, so $z$ must be square, and we can assume $h_1=\diagmat{x}{-x}$;
and we must have $h_2^2=h^2$ for some scalar $h\in Z$, so we can assume $h_2=\diagmat{1}{-1}$.

Since $\gamma_1$ is conjugate to $h_2$ and $\gamma_2$ is conjugate to $h_1$, so may assume that $G_1=\bigl\langle \diagmat{x}{-x},\smallmat{0}{1}{1}{0}\bigr\rangle$ and $G_2=\bigl\langle\diagmat{z}{z},\diagmat{1}{-1},\smallmat{0}{x}{x}{0}\bigr\rangle$; this shows
whenever $\ell\equiv 1\bmod 4$, for each square $z\in \Z(\ell)^\times$ of even order we can construct conjugate $G_1$ and $G_2$ with $H_1$ and $H_2$ non-conjugate as above.
\end{proof}

\begin{corollary}\label{cor:nonsplitdihedral}
Let $\ell$ be an odd prime, let $\gamma=\diagmat{1}{-1}$, and let $\delta$ be a generator for $C_{ns}(\ell)$.
For each subgroup $H\subseteq C_{ns}(\ell)$ not in $Z(\ell)$ the group $G_1=\langle H,\gamma\rangle\subseteq C_{ns}^+(\ell)$ has dihedral image in $\PGL_2(\ell)$ and satisfies $H=G_1\cap C_{ns}(\ell)$ with $\det(H)=-\det(G_1-H)$.
If $H$ satisfies $\det(H)=\det(H\cap Z(\ell))$ and $-1\in H$, then for $e:=[Z(\ell)\colon H\cap Z(\ell)]$, the group $G_2:=\langle H,\gamma \delta^e\rangle\subseteq C_{ns}^+(\ell)$ has dihedral image in $\PGL_2(\ell)$ and satisfies $H=G_2\cap C_{ns}(\ell)$ with $\det(H)$ and $-\det(G_2-H)$ disjoint.
Up to conjugacy in $\GL_2(\ell)$, this accounts for all subgroups $G$ that lie in the normalizer of a non-split Cartan group and have dihedral image in $\PGL_2(\ell)$.
The number of such $G$ is
\[
\tau(\ell^2-1)-\tau(\ell-1)+\tau\left(\frac{\ell^2-1}{4}\right)-\tau\left(\frac{\ell-1}{2}\right).
\]
\end{corollary}
\begin{proof}
It is clear that $G_1$ and $G_2$ both have dihedral image in $\PGL_2(\ell)$ and intersect $C_{ns}(\ell)$ in $H$, since $\gamma$ and $\gamma r^e$ both lie in $C_{ns}^+(\ell)$ but not $C_{ns}(\ell)$ and their squares lie in $H\cap Z(\ell)$.
For $G_1$ it is clear that $\det(H)=-\det(G_1-H)$, and for $G_2$ we note that $(\gamma \delta^e)^2$ generates $H\cap Z(\ell)$, by construction, and if $\det(H)=\det(H\cap Z(\ell))$ then $-\det(\gamma r^e)\not\in \det (H)$, and by Lemma~\ref{lem:dihedral}, the sets $\det(H)$ and $-\det(G_2-H)$ must then be disjoint.

Every subgroup $H\subseteq C_{ns}(\ell)$ is normal in $C_{ns}^+(\ell)$ and has no non-trivial $\GL_2(\ell)$-conjugates in $C_{ns}^+(\ell)$.
It follows from Lemmas~\ref{lem:dihedral} and~\ref{lem:kleinexcep} that up to conjugacy in $\GL_2(\ell)$, each $G_1,G_2$ arises for exactly one $H$.

The first two terms in the formula count subgroups $H\subseteq C_{ns}(\ell)$ not in $Z(\ell)$.
Among these, those that satisfy $\det(H)=\det(H\cap Z(\ell))$ and $-1\in H$ are precisely those that lie in the index 2 subgroup of $C_{ns}(\ell)$ (squares) and contain a subgroup of order $2$, which accounts for the last two terms in the formula.
\end{proof}

The split dihedral case is slightly more complicated due to the fact that $C_s(\ell)$ contains subgroups $H$ that are not normal in $C_s^+(\ell)$, and Lemma~\ref{lem:kleinexcep} implies that even when $H$ is normal in $C_s^+(\ell)$ it may have distinct $\GL_2(\ell)$-conjugates that also lie in $C_s^+(\ell)$.

\begin{corollary}\label{cor:splitdihedral}
Let $\ell$ be an odd prime, let $\gamma=\smallmat{0}{1}{1}{0}$, and let $\delta\in C_s(\ell)$ be a coset representative of a generator for $C_s(\ell)/(C_s(\ell)\cap \SL_2(\ell))$.
For each subgroup $H\subseteq C_s(\ell)$ not in $Z(\ell)$ that is normal in $C_s^+(\ell)$, the group $G_1=\langle H,\gamma\rangle\subseteq C_s^+(\ell)$ satisfies $H=G_1\cap C_s(\ell)$ with $\det(H)=-\det(G_1-H)$.
If $H$ satisfies $\det(H)=\det(H\cap Z(\ell))$ and $-1\in H$, then for $e:=[Z(\ell):H\cap Z(\ell)]$, the group $G_2:=\langle H,\gamma \delta^e\rangle\subseteq C_s^+(\ell)$ satisfies $H=G_2\cap C_s(\ell)$ with $\det(H)$ and $-\det(G_2-H)$ disjoint.
Up to conjugacy in $\GL_2(\ell)$, this accounts for all subgroups $G$ that lie in the normalizer of a split Cartan group and have dihedral image in $\PGL_2(\ell)$.
The total number of such $G$ is
\[
\beta(\ell-1)-\tau(\ell-1)+\tau\left(\frac{\ell-1}{2}\right)^2-\tau\left(\frac{\ell-1}{2}\right)-\frac{1}{2}\left(1+\left(\frac{-1}{\ell}\right)\right)\tau\left(\frac{\ell-1}{4}\right).
\]
\end{corollary}
\begin{proof}
The argument that $G_1$ and $G_2$ have the claimed properties is identical to that in the proof of Corollary~\ref{cor:nonsplitdihedral}, as is the argument that they are uniquely determined by $H$.

The first two terms in the formula count the normal subgroups $H$ of $C_s(\ell)$ not in $Z(\ell)$, via Corollary~\ref{cor:splitnormal}, each of which gives rise to a $G_1$; these $G_1$ are all non-conjugate so long as we are not in the exceptional case of Lemma~\ref{lem:kleinexcep}.  The last term in the formula is a correction factor for double-counting the exceptional cases.

The third and fourth terms in the formula account for subgroups $H$ that satisfy $\det(H)=\det(H\cap Z(\ell))$ and $-1\in H$.  To see this note that in the proof of Lemma~\ref{lem:beta}, adding the restriction $\det(H)=\det(H\cap Z(\ell))$ replaces the factor $\beta_2(n)$ with $(e_2+1)^2$ and the modified formula for $\beta(n)$ is then $\tau(n)^2$; using $n=(\ell-1)/2$ accounts for the constraint $-1\in H$.
Each such $H$ gives rise to a $G_2$, and these are all non-conjugate.
\end{proof}

\begin{lemma}\label{lem:splitnonsplitdihedral}
Let $\ell$ be an odd prime and let $G$ be a subgroup of $\GL_2(\ell)$ with dihedral image in $\PGL_2(\ell)$.
Then $G$ is contained in both the normalizer of a split Cartan group and the normalizer of a non-split Cartan group if and only if $G$ is conjugate to a subgroup of the form
\[
H_z:=\left\langle\smallmat{0}{1}{z}{0},\diagmat{1}{-1}\right\rangle,
\]
where $z\in \Z(\ell)^\times$ is not a square, in which case the image of $G$ in $\PGL_2(\ell)$ is the Klein group of order $4$.
There is exactly one such $H_z$ for each odd divisor of $\ell-1$.
\end{lemma}
\begin{proof}
Every non-scalar element of $G$ lies in the non-trivial coset of a subgroup of a Cartan group in its normalizer, hence has trace zero and order $2$ in $\PGL_2(\ell)$.
It follows that the image of $G$ in $\PGL_2(\ell)$ has order 4, and we can write $G=\langle g_1,g_2\rangle$ with $\tr g_1=\tr g_2=0$, and $\det g_1$ square, while $\det g_2$ is not square.

If $\ell\equiv 1 \bmod 4$, then after multiplication by a scalar, we can assume $\det g_1=-1$, and $G$ is then conjugate to $H_z\subseteq C_s^+(\ell)$ via an action that sends $g_1$ to $\diagmat{1}{-1}$ and $g_2$ to $\smallmat{0}{1}{z}{0}$, with $z=-\det g_2$ not a square.

If $\ell\equiv 3\bmod 4$, then after multiplication by a scalar we can assume $\det g_2 = -1$ and $G$ is then conjugate to $H_z\subseteq C_{ns}^+(\ell)$ via an action that sends $g_2$ to $\diagmat{1}{-1}$ and $g_1$ to $\smallmat{0}{1}{z}{0}$, with $z=-\det g_1$ not a square.

Conversely, for each non-square $z\in \Z(\ell)^\times$ the subgroup $H_z$ lies in $C_s^+(\ell)\cap C_{ns}^+(\ell)$ and has dihedral image in $\PGL_2(\ell)$.
If we fix a generator $r$ for $\Z(\ell)^\times$, the distinct groups $H_z$ that can arise are precisely those with $z=r^e$, where $e$ is an odd divisor of $\ell-1$.
\end{proof}

\begin{remark}
Not every $G\subseteq \GL_2(\ell)$ with projective image isomorphic to the Klein group is contained in both the normalizer of a split Cartan group and the normalizer of a non-split Cartan group; this occurs if and only if $G$ contains elements $g,h$ with $\chi(g)=1$ and $\chi(h)=-1$.
\end{remark}

\subsection{Exceptional cases}
We now consider the exceptional case (3) of Proposition~\ref{prop:subgroups}.
In all of these cases the group $G\subseteq \GL_2(\ell)$ is determined up to conjugacy by three criteria: the isomorphism class of its image in $\PGL_2(\ell)$, the cardinality of its scalar subgroup $Z:=G\cap Z(\ell)$, and the index $[\det(G):\det(Z)]$.

\begin{lemma}\label{lem:excep}
Let $\ell\ge 5$ be prime, and suppose that $G$ is a subgroup of $\GL_2(\ell)$ with projective image isomorphic to $H \in \{\alt{4},\sym{4},\alt{5}\}$ and scalar subgroup $Z:=G\cap Z(\ell)$ containing $-1$.
\begin{enumerate}
\setlength{\itemsep}{4pt}
\item[{\rm (3a)}] If $H=\alt{4}$ then one of the following holds:
\begin{enumerate}[{\rm(i)}]
\item $[\det(G):\det(Z)]=1$,
\item $[\det(G):\det(Z)]=3$ and $\ell\equiv 1\bmod 3$ with $[Z(\ell):Z]$ divisible by $3$.
\end{enumerate}
\item[{\rm (3b)}] If $H=\sym{4}$ then one of the following holds:
\begin{enumerate}[{\rm(i)}]
\item $[\det(G):\det(Z)]=1$ and $\ell\equiv \pm 1\bmod 8$.
\item $[\det(G):\det(Z)]=2$ and $\ell\equiv 1\bmod 8$ with $[Z(\ell):Z)]$ divisible by $2$.
\item $[\det(G):\det(Z)]=2$ and $\ell\equiv 3\bmod 8$.
\item $[\det(G):\det(Z)]=2$ and $\ell\equiv 5\bmod 8$ with $\#Z$ divisible by $4$.
\end{enumerate}
\item[{\rm (3c)}] If $H=\alt{5}$ then $[\det(G):\det(Z)]=1$ and $\ell\equiv\pm 1\bmod 5$.
\end{enumerate}
Moreover, every case listed above arises for exactly one conjugacy class of subgroups $G$ in $\GL_2(\ell)$.
\end{lemma}
\begin{proof}
The lemma follows from the classification in \cite{fo05}; see Theorems 5.5, 5.8, and 5.11.  It can also be derived from the analysis in \cite[\S 5.2]{anni13}.
\end{proof}

The explicit classification of primitive subgroups of $\GL_2(\ell)$ in \cite{fo05} also provides a method for constructing a subgroup $G\subseteq\GL_2(\ell)$ that satisfies Lemma~\ref{lem:excep} for given values of $H$, $Z$, and $[\det(G):\det(Z)]$, whenever such a $G$ exists (if it exists, it is unique up to conjugacy, by the previous lemma).
The complexity of this algorithm is important to what follows, so we give it in detail and then bound its complexity.
The construction given in \cite{fo05} gives generators for a subgroup $\tilde{G}$ of $\GL_2(\F_{\ell^2})$ that is conjugate to our desired $G\subseteq\GL_2(\ell)$; we then use the algorithm of \cite{gh97} to efficiently conjugate $\tilde{G}$ to $G$.

\begin{algorithm}\label{alg:excep}
Given a prime $\ell \ge 5$, a group $H\in \{\alt{4},\sym{4},\alt{5}\}$, a subgroup $Z\subseteq\Z(\ell)$ containing $-1$ generated by $\lambda$, and  $i\in \{1,2,3\}$, output generators for a group $G\subseteq\GL_2(\ell)$ with projective image isomorphic to $H$, and scalar subgroup $Z\subseteq Z(\ell)$ such that $[\det(G):\det(Z)]=i$, or report that no such $G$ exists.
\end{algorithm}
\begin{enumerate}[{\bf 1.}]
\item Let $\omega\in \F_{\ell^2}$ be a primitive fourth root of unity, let $s:=\frac{1}{2}\left(\begin{smallmatrix}\omega-1&\omega-1\\\omega+1&-(\omega+1)\end{smallmatrix}\right)$, and let $t:=\diagmat{\omega}{-\omega}$.
\item If $H=\alt{4}$ then
\begin{enumerate}[{\bf a.}]
\setlength{\itemsep}{0pt}
\item If $i=1$ let $\tilde{G}:=\langle s,t,\lambda\rangle$.
\item If $i=3$ and $\ell\equiv 1\bmod 3$ with $3|[Z(\ell):Z]$ let $\tilde{G}:=\langle \mu s,t,\lambda\rangle$ where $\mu\in Z(\ell)-Z$ satisfies $\mu^3=\lambda$.
\item Otherwise, report that no such $G$ exists and terminate.
\end{enumerate}
\item If $H=\sym{4}$ then
\begin{enumerate}[{\bf a.}]
\item Let $\alpha\in \F_{\ell^2}$ be a square root of $2$ and let $u:=\diagmat{1+\omega}{1-\omega}$.
\item If $i=1$ and $\ell\equiv \pm 1 \bmod 8$ let $\tilde{G}:=\langle s,\frac{u}{\alpha},\lambda\rangle$.
\item If $i=2$ and $\ell\equiv 1\bmod 8$ with $[Z(\ell):Z]$ even, let $\tilde{G}:=\langle s,\frac{\mu}{\alpha} u,\lambda\rangle$ where $\mu\in Z(\ell)$ satisfies $\mu^2=\lambda$.
\item If $i=2$ and $\ell\equiv 3\bmod 8$ let $\tilde{G}:=\langle s,\frac{\mu}{\alpha} u,\lambda\rangle$ where $\mu\in Z(\ell)$ satisfies $\mu^2=\lambda$.
\item If $i=2$ and $\ell\equiv 5\bmod 8$ with $4|\#Z$, let $\tilde{G}:=\langle s,\frac{\mu}{\alpha} u,\lambda\rangle$ where $\frac{\mu}{\alpha}\in Z(\ell)$ satisfies $(\frac{\mu}{\alpha})^2=\frac{\lambda}{2}$
\item Otherwise, report that no such $G$ exists and terminate.
\end{enumerate}
\item If $H=\alt{5}$ then
\begin{enumerate}[{\bf a.}]
\item If $i=1$ and $\ell\equiv \pm 1\bmod 5$ then let $v:=\frac{1}{4}\smallmat{2\omega}{1-\omega-\sqrt{5}(1+\omega)}{\sqrt{5}(1-\omega)-1-\omega}{-2\omega}$, and let $\tilde{G}=\langle s, t, v, \lambda\rangle$.
\item Otherwise, report that no such $G$ exists and terminate.
\end{enumerate}
\item By solving a linear system in 4 variables and at most 16 equations, construct a matrix $C\in \GL_2(\F_{\ell^2})$ for which $gC=Cg^\sigma$ holds for all $g\in \tilde{G}$, where $\langle\sigma\rangle = \Gal(\F_{\ell^2}/\F_\ell)$.
\item Generate random matrices $X\in \mathbf{M}_2(\F_{\ell^2})$ until $A:= X+CX$ is invertible.
\item Output $G:=A^{-1}\tilde{G}A\subseteq \GL_2(\ell)$ and terminate.
\end{enumerate}

The last 3 steps of Algorithm~\ref{alg:excep} implement a special case of the probabilistic (Las Vegas) algorithm in~\cite{gh97} which, given a subgroup $\tilde{G}$ of $\GL_r(\F_{p^n})$, finds a conjugate subgroup $G$ in $\GL_r(\F_{p^m})$ with $m|n$ minimal.
The correctness of Algorithm~\ref{alg:excep}, including the fact that a subgroup $G\subseteq\GL_2(\ell)$ conjugate to $\tilde{G}\subseteq\GL_2(\F_{\ell^2})$ necessarily exists, is guaranteed by Theorems 5.5, 5.8, and 5.11 of \cite{fo05}.
We now analyze its complexity.

\begin{proposition}\label{prop:algexceptime}
The expected running time of Algorithm~\ref{alg:excep} is $O(\M(\log \ell)\log \ell)$.
\end{proposition}
\begin{proof}
Using standard probabilistic root-finding algorithms we can find the roots of any polynomial of bounded degree over $\F_{\ell}$ or $\F_{\ell^2}$ in $O(\M(\log\ell)\log\ell)$ expected time \cite{gg13}.
Every other operation in Algorithm~\ref{alg:excep} takes $O(\M(\log \ell))$ time, including the linear algebra in step 5, since the dimensions of the system are bounded.  The expected number of random matrices needed in step 6 is at most $4$; see \cite[p.\ 1707]{gh97}.
\end{proof}

\subsection{Counting and enumerating subgroups}\label{sec:gl2sum}
As a result of our classification we can now count the number of subgroups of $\GL_2(\ell)$ up to conjugacy.
For $\ell=2$ there are four non-conjugate subgroups of $\GL_2(2)$, namely, $C_s(2)$, $C_{ns}(2)$, $B(2)$, and $\GL_2(2)=\SL_2(2)$.
For primes $\ell>2$, every subgroup of $\GL_2(\ell)$ is conjugate to at least one of the groups enumerated on the next page.
The 11 cases that appear are disjoint except for $C_s$ and $C_{ns}$, which intersect in $Z$, and $C_s^+$ and $C_{ns}^+$, which intersect in $C_{s\cap ns}^+$.
Other than these intersections all of the groups listed are non-conjugate in $\GL_2(\ell)$.

We thus obtain an explicit formula for the number of non-conjugate subgroups of $\GL_2(\ell)$ by summing the formulas for the 11 listed cases with the counts for $Z$ and $C_s\cap C_{ns}^+$ negated.
Table~\ref{tab:gl2counts} lists this data for odd primes $\ell<200$ along with several larger primes.
These formulas can easily be adapted to count subgroups of $\SL_2(\ell)$ instead.\vspace{-12pt}

\begin{enumerate}
\setlength{\itemsep}{2pt}
\setlength{\itemindent}{8pt}
\item[$\SL_2$:] $\tau(\ell-1)$ subgroups that contain $\SL_2(\ell)$;
\item[$B$:] $\alpha(\ell-1)$ subgroups of $B(\ell)$ that contain an element of order $\ell$;
\item[$C_s$:] $\frac{1}{2}\bigl(\alpha(\ell-1)+\beta(\ell-1)\bigr)$ subgroups of $C_s(\ell)$;
\item[$C_{ns}$:] $\tau(\ell^2-1)$ subgroups of $C_{ns}(\ell)$;
\item[$Z$:] $\tau(\ell-1)$ subgroups of $C_s(\ell)\cap C_{ns}(\ell)=Z(\ell)$;
\item[$C_s^+$:] $\beta(\ell-1)-\tau(\ell-1)+\tau\bigl(\frac{\ell-1}{2}\bigr)^2-\tau\bigl(\frac{\ell-1}{2}\bigr)-\frac{1}{2}\bigl(1+\bigl(\frac{-1}{\ell}\bigr)\bigr)\tau\bigl(\frac{\ell-1}{4}\bigr)$ subgroups of $C_s^+(\ell)$ not in $C_s(\ell)$;
\item[$C_{ns}^+$:] $\tau(\ell^2-1)-\tau(\ell-1)+\tau\bigl(\frac{\ell^2-1}{4}\bigr)-\tau\bigl(\frac{\ell-2}{2}\bigr)$ subgroups of $C_{ns}^+(\ell)$ not in $C_{ns}(\ell)$;
\item[$C_{s\cap ns}^+$:] $\tau\bigl((\ell-1)/2^{v_2(\ell-1)}\bigr)$ subgroups of $C_s^+(\ell)\cap C_{ns}^+(\ell)$ not contained in $C_s(\ell)$ or $C_{ns}(\ell)$;
\item[$\alt{4}$:] $\tau\bigl(\frac{\ell-1}{2}\bigr)+\frac{1}{2}(1+(\frac{-3}{\ell})\tau\bigl(\frac{\ell-1}{6}\bigr)$ subgroups $G\not\supseteq\SL_2(\ell)$ with $\pi(G)\simeq\alt{4}$;
\item[$\sym{4}$:] $\bigl(1-\frac{1}{4}\bigl(1-\bigl(\frac{2}{p}\bigr)\bigr)\bigl(1-\bigl(\frac{-1}{p}\bigr)\bigr)\bigr)\tau\bigl(\frac{\ell-1}{2}\bigr) +\frac{1}{2}\bigl(1+\bigl(\frac{-1}{p}\bigr)\bigr)\tau\bigl(\frac{\ell-1}{4}\bigr)$ subgroups $G\not\supseteq\SL_2(\ell)$ with $\pi(G)\simeq\sym{4}$;
\item[$\alt{5}$:] $\frac{1}{2}\bigl(1+\bigl(\frac{5}{p}\bigr)\tau\bigl(\bigl(\frac{\ell-1}{2}\bigr)\bigr)$ subgroups $G\not\supseteq\SL_2(\ell)$ with $\pi(G)\simeq\alt{5}$.
\end{enumerate}
\smallskip

\begin{remark}\label{rem:linearcount}
From the formulas for $\alpha(n)$ and $\beta(n)\le\alpha(n)$, and the bound $\tau(n)=2^{O(\log n/\log\log n)} = n^{o(1)}$, one may deduce that the number of subgroups of $\GL_2(\ell)$ is quasi-linear in $\ell$.
Indeed, the lower bound $\alpha(n)=\Omega(n)$ is immediate, and the upper bound $\alpha(n)=O(n\log\log\log n)$ is easy to prove.
\end{remark}

\begin{table}
\begin{tabular}{rrrrrrrrrrrrr}
$\ell$ & $\SL_2$ & $B$ & $C_s$ & $C_{ns}$ &$Z$ & $C_s^+$ & $C_{ns}^+$ & $C_{sns}^+$ & $\alt{4}$ & $\sym{4}$ & $\alt{5}$ & $\GL_2$\\\toprule
3 & 2 & 5 & 4 & 4 & 2 & 1 & 3 & 1 & 0 & 0 & 0 & 16 \\
5 & 3 & 15 & 11 & 8 & 3 & 5 & 7 & 1 & 2 & 1 & 0 & 48 \\
7 & 4 & 30 & 21 & 10 & 4 & 10 & 10 & 2 & 3 & 2 & 0 & 84 \\
11 & 4 & 40 & 26 & 16 & 4 & 10 & 18 & 2 & 2 & 2 & 2 & 114 \\
13 & 6 & 90 & 59 & 16 & 6 & 32 & 14 & 2 & 6 & 2 & 0 & 217 \\
17 & 5 & 83 & 55 & 18 & 5 & 31 & 21 & 1 & 4 & 7 & 0 & 218 \\
19 & 6 & 115 & 71 & 24 & 6 & 27 & 27 & 3 & 5 & 3 & 3 & 272 \\
23 & 4 & 70 & 41 & 20 & 4 & 10 & 26 & 2 & 2 & 2 & 0 & 169 \\
29 & 6 & 150 & 89 & 32 & 6 & 32 & 38 & 2 & 4 & 2 & 4 & 349 \\
31 & 8 & 240 & 144 & 28 & 8 & 52 & 36 & 4 & 6 & 4 & 4 & 510 \\
37 & 9 & 345 & 204 & 24 & 9 & 81 & 21 & 3 & 10 & 3 & 0 & 685 \\
41 & 8 & 296 & 178 & 40 & 8 & 78 & 50 & 2 & 6 & 10 & 6 & 662 \\
43 & 8 & 300 & 174 & 32 & 8 & 52 & 36 & 4 & 6 & 4 & 0 & 600 \\
47 & 4 & 130 & 71 & 24 & 4 & 10 & 34 & 2 & 2 & 2 & 0 & 271 \\
53 & 6 & 240 & 134 & 32 & 6 & 32 & 38 & 2 & 4 & 2 & 0 & 480 \\
59 & 4 & 160 & 86 & 32 & 4 & 10 & 42 & 2 & 2 & 2 & 2 & 334 \\
61 & 12 & 720 & 416 & 32 & 12 & 152 & 28 & 4 & 12 & 4 & 8 & 1368 \\
67 & 8 & 420 & 234 & 32 & 8 & 52 & 36 & 4 & 6 & 4 & 0 & 780 \\
71 & 8 & 400 & 224 & 60 & 8 & 52 & 84 & 4 & 4 & 4 & 4 & 828 \\
73 & 12 & 851 & 493 & 30 & 12 & 189 & 27 & 3 & 15 & 15 & 0 & 1617 \\
79 & 8 & 480 & 264 & 48 & 8 & 52 & 68 & 4 & 6 & 4 & 4 & 922 \\
83 & 4 & 220 & 116 & 32 & 4 & 10 & 42 & 2 & 2 & 2 & 0 & 422 \\
89 & 8 & 518 & 289 & 60 & 8 & 78 & 82 & 2 & 6 & 10 & 6 & 1047 \\
97 & 12 & 1062 & 617 & 42 & 12 & 242 & 50 & 2 & 15 & 18 & 0 & 2044 \\
101 & 9 & 675 & 369 & 48 & 9 & 81 & 57 & 3 & 6 & 3 & 6 & 1242 \\
103 & 8 & 600 & 324 & 40 & 8 & 52 & 52 & 4 & 6 & 4 & 0 & 1074 \\
107 & 4 & 280 & 146 & 32 & 4 & 10 & 42 & 2 & 2 & 2 & 0 & 512 \\
109 & 12 & 1140 & 626 & 64 & 12 & 152 & 76 & 4 & 14 & 4 & 8 & 2080 \\
113 & 10 & 830 & 469 & 48 & 10 & 148 & 62 & 2 & 8 & 14 & 0 & 1577 \\
127 & 12 & 1150 & 629 & 54 & 12 & 126 & 78 & 6 & 10 & 6 & 0 & 2047 \\
131 & 8 & 640 & 344 & 64 & 8 & 52 & 84 & 4 & 4 & 4 & 4 & 1192 \\
137 & 8 & 740 & 400 & 40 & 8 & 78 & 50 & 2 & 6 & 10 & 0 & 1322 \\
139 & 8 & 780 & 414 & 64 & 8 & 52 & 84 & 4 & 6 & 4 & 4 & 1404 \\
149 & 6 & 600 & 314 & 48 & 6 & 32 & 62 & 2 & 4 & 2 & 4 & 1064 \\
151 & 12 & 1350 & 729 & 60 & 12 & 126 & 78 & 6 & 9 & 6 & 6 & 2358 \\
157 & 12 & 1440 & 776 & 32 & 12 & 152 & 28 & 4 & 12 & 4 & 0 & 2440 \\
163 & 10 & 1185 & 630 & 40 & 10 & 85 & 45 & 5 & 9 & 5 & 0 & 1994 \\
167 & 4 & 430 & 221 & 40 & 4 & 10 & 58 & 2 & 2 & 2 & 0 & 761 \\
173 & 6 & 690 & 359 & 32 & 6 & 32 & 38 & 2 & 4 & 2 & 0 & 1155 \\
179 & 4 & 460 & 236 & 48 & 4 & 10 & 66 & 2 & 2 & 2 & 2 & 824 \\
181 & 18 & 2760 & 1506 & 96 & 18 & 360 & 114 & 6 & 20 & 6 & 12 & 4868 \\
191 & 8 & 880 & 464 & 64 & 8 & 52 & 100 & 4 & 4 & 4 & 4 & 1568 \\
193 & 14 & 2202 & 1227 & 32 & 14 & 360 & 30 & 2 & 18 & 22 & 0 & 3889 \\
197 & 9 & 1125 & 594 & 72 & 9 & 81 & 93 & 3 & 6 & 3 & 0 & 1971 \\
199 & 12 & 1610 & 859 & 90 & 12 & 126 & 126 & 6 & 10 & 6 & 6 & 2827\\
$10^3+9$ & 30 & 19090 & 10031 & 144 & 30 & 1476 & 186 & 6 & 40 & 42 & 24 & 31027 \\
$10^4+7$ & 4 & 25030 & 12521 & 60 & 4 & 10 & 90 & 2 & 2 & 2 & 0 & 37713 \\
$10^5+3$ & 16 & 715200 & 357696 & 128 & 16 & 232 & 168 & 8 & 12 & 8 & 0 & 1073436\\
$10^6+3$ & 8 & 5000100 & 2500074 & 96 & 8 & 52 & 132 & 4 & 6 & 4 & 0 & 7500460 \\\bottomrule
\end{tabular}
\bigskip

\caption{Subgroups of $\GL_2(\ell)$ up to conjugacy. See \S\ref{sec:gl2sum} for an explanation of the column headings.}\label{tab:gl2counts}
\end{table}

We now give an efficient Las Vegas algorithm to enumerate the subgroups of $\GL_2(\ell)$ up to conjugacy.
It outputs a short list of $O(1)$ generators for each subgroup and has a total expected running time that is quasi-linear in $\ell$, hence in the size of its output.

\begin{algorithm}\label{alg:enum}
Given a prime $\ell$, output a list of the subgroups of $\GL_2(\ell)$ up to conjugacy as follows:
\end{algorithm}
\begin{enumerate}[{\bf 1.}]
\setlength\itemsep{2pt}
\item (\textbf{even} $\boldsymbol{\ell}$) If $\ell=2$ then output $\langle\rangle,\ \langle\smallmat{1}{1}{0}{1}\rangle,\ \langle\smallmat{1}{1}{1}{0}\rangle,\ \langle\smallmat{1}{1}{0}{1},\smallmat{1}{1}{1}{0}\rangle$ and terminate.
\item Compute a generator $r$ for $\Z(\ell)$, a generator $g$ for $C_{ns}(\ell)$, lists of the divisors of $\ell-1$ and $\ell^2-1$, and a lookup table $T:=\{(u(g),|\pi(g)|):g \in C_s(\ell)\cup C_{ns}(\ell)\}$ indexed by $u(g):=\tr(g)^2/\det(g)$.
\item (\textbf{contains} $\boldsymbol{\SL_2(\ell)}$) For each $e$ dividing $\ell-1$ output $\langle \smallmat{1}{1}{0}{1},\smallmat{1}{0}{1}{1},\diagmat{1}{r^e}\rangle$.
\item (\textbf{in} $\boldsymbol{B(\ell)}$) For each triple $(a,b,i)$ with $a,b|(\ell-1)$ and $0\le i < \gcd(a,b)$, output
\[
\left\langle\begin{pmatrix}r^a&0\\0&1/r^a\end{pmatrix},\begin{pmatrix}r^{ic}&0\\0&r^{d-ic}\end{pmatrix},\begin{pmatrix}1&1\\0&1\end{pmatrix}\right\rangle.
\]
where $c=a/\gcd(a,b)$ and $d=n/b$.
\item (\textbf{exceptional cases})  If $\ell\ge 5$ then call Algorithm~\ref{alg:excep} for each $H\in\{\alt{4},\sym{4},\alt{5}\}$, $i\in \{1,2,3\}$, and  $Z=\langle\diagmat{r^n}{r^n}\rangle$ with $n$ dividing $(\ell-1)/2$.
\item (\textbf{cyclic cases})
\begin{enumerate}[{\bf a.}]
\item (\textbf{split}) For each $(a,b,i)$ with $a,b|(\ell-1)$ and $0\le i < \gcd(a,b)$, put $c=a/\gcd(a,b)$ and $d=(\ell-1)/b$, and if there is no integer $j\in [0,i-1]$ satisfying $jc\equiv d-ic\bmod a$ then output
\[
H_{a,b,i}:=\left\langle\begin{pmatrix}r^a&0\\0&1/r^a\end{pmatrix},\begin{pmatrix}r^{ic}&0\\0&r^{d-ic}\end{pmatrix}\right\rangle.
\]
\item (\textbf{nonsplit}) For each $n|(\ell^2-1)$ not divisible by $\ell+1$ output $H_n:=\langle g^n\rangle$, where $C_{ns}(\ell)=\langle g\rangle$.
\end{enumerate}
\item (\textbf{dihedral cases})
\begin{enumerate}[{\bf a.}]
\item (\textbf{split}) Let $\gamma:=\smallmat{0}{1}{1}{0}$ and $\delta:=\diagmat{1}{r}$.
For each $H_{a,b,i}$ as in step 6.a with $2ic\equiv d\bmod a$:
\begin{enumerate}[{\bf i.}]
\item Compute $Z_{a,b,i}:=H_{a,b,i}\cap Z(\ell)$ using the table $T$ as described below.
\item Unless $-1\in Z_{a,b,i}$,\ $[H_{a,b,i}:Z_{a,b,i}]= 2$, and $\diagmat{1}{-1}\in H_{a,b,i}$, output $\langle H_{a,b,i},\gamma\rangle$.
\item If $-1\in Z_{a,b,i}$ and $\det(H_{a,b,i})=\det(Z_{a,b,i})$ then output $\langle H,\gamma\delta^e\rangle$, where $e:=[Z(\ell):Z_{a,b,i}]$.
\end{enumerate}
\item (\textbf{nonsplit}) Let $\gamma=\diagmat{1}{-1}$.  For each $H_n=\langle g^n\rangle$ as in step 6.b:
\begin{enumerate}[{\bf i.}]
\item Compute $Z_n:=H_n\cap Z(\ell)$ using the table $T$ as described below.
\item If $-1\in Z_n$ and $\det(H_n)=\det(Z_n)$ then output $\langle H_n,\gamma g^e\rangle$, where $e:=[Z(\ell):Z_n]$.
\item Output $\langle H_n,\gamma\rangle$.
\end{enumerate}
\end{enumerate}
\end{enumerate}

The scalar subgroup $Z_{a,b,i}:=H_{a,b,i}\cap Z(\ell)$ computed in step 7.a.ii is uniquely determined by its order, which we can compute as $\#H_{a,b,i}/\#\pi(H_{a,b,i})$, where $\pi\colon \GL_2(\ell)\twoheadrightarrow\PGL(\ell)$ is the canonical projection.
Since $\pi(H_{a,b,i})$ is cyclic, we may compute its order as the least common multiple of the projective orders of the generators of $H_{a,b,i}$, which may be determined using the lookup table $T$ computed in step 2.
Similar comments apply to computing $Z_n:=H_n\cap Z(\ell)$ in step 7.b.ii.

The correctness of Algorithm~\ref{alg:enum} follows from Proposition~\ref{prop:subgroups}, the correctness of Algorithm~\ref{alg:excep}, and the analysis in \S\ref{sec:borel}, \S\ref{sec:cyclic}, and \S\ref{sec:dihedral}.
The constraint on $i$ in step 6.a ensures that we pick just one of the two possible conjugacy class representatives of a subgroup of $C_s(\ell)$, and the constraint on $H_{a,b,i}$ in step 7.a.ii uses Lemma~\ref{lem:kleinexcep} to pick just one of the two possible conjugacy class representatives of a subgroup of $C_s(\ell)^+$ with projective image isomorphic to the Klein group.

\begin{proposition}\label{prop:algenumtime}
The expected running time of Algorithm~\ref{alg:enum} is $\ell^{1+o(1)}$.
\end{proposition}
\begin{proof}
We first consider step 2.
We can compute the generators $r$ and $g$ in $(\log \ell)^{2+o(1)}$ expected time using probabilistic algorithms.
We can compute the divisors of $\ell-1$ and $\ell+1$ in $\ell^{1+o(1)}$ time using a sieve, and these lists can then be used to construct a complete list of the divisors of $\ell^2-1=(\ell-1)(\ell+1)$ in $\ell^{o(1)}$ time (here we are using the the fact that an integer $n$ has at most $n^{o(1)}$ divisors).
To compute the table $T$, we note that for $C_s(\ell)$ it suffices to compute $(u(a^e),(\ell-1)/e)$ using $a=\diagmat{1}{r}$ for $1\le e\le \ell-1$, and for $C_{ns}(\ell)$ it suffices to compute $(u(g^e),(\ell+1)/e)$ using the generator $g$ for $C_{ns}(\ell)$ for $1\le e\le \ell+1$.
Thus step 2 takes $\ell^{1+o(1)}$ time.

Step 3 clearly takes $\ell^{o(1)}$ time.  For step 4 we note that the number of triples $(a,b,i)$ is given by
\[
\alpha(\ell-1)=\sum_{a,b|(\ell-1)}\gcd(a,b) = \prod_p\left(\sum_{0\le i\le v_p(\ell-1)} (2(v_p(\ell-1)-i)+1)p^i\right) = \ell^{1+o(1)},
\]
and the time to compute generators for each individual subgroup of $B(\ell)$ is $\ell^{o(1)}$.
There are $\ell^{o(1)}$ calls to Algorithm~\ref{alg:excep} in step 5, each of which takes $\ell^{1+o(1)}$ expected time, by Proposition~\ref{prop:algexceptime}.
The number of subgroups $H_{a,b,i}$ in step 6.a is bounded by $\alpha(\ell-1)=\ell^{1+o(1)}$, and each takes $\ell^{o(1}$ time to compute, while step 6.b takes $\ell^{o(1)}$ time.
The number of groups arising in step 7 is similarly bounded by $\ell^{1+o(1)}$, and the time for each group is $(\log \ell)^{2+o(1)}$, using the table $T$ to compute the projective orders of $H_{a,b,i}$ and $H_n$ as described above in order to determine their scalar subgroups.
\end{proof}

A Magma \cite{magma} script implementing Algorithm~\ref{alg:enum} is available from the author's website \cite{sut15}.
In practical terms, it typically takes just a few seconds for $\ell\approx 10^3$ and less than an hour for $\ell\approx 10^6$, computations that would be infeasible using the \texttt{Subgroups} function in Magma, or similar functionality in GAP \cite{GAP}.

\subsection{Subgroup signatures}\label{sec:grpsig}

\begin{definition}
For each $g\in \GL_2(\ell)$ we define
\[
\sig(g):=(\det(g),\tr(g),\dim_1(g))
\]
where $\dim_1(g)\in \{0,1,2\}$ is the dimension of the $1$-eigenspace of $g$.
For each subgroup $G\subseteq \GL_2(\ell)$ we define the \emph{signature} of $G$ to be the set
\[
\sig(G):=\{\sig(g):g \in G\}.
\]
\end{definition}

\begin{lemma}\label{lem:borel121}
Let $\ell$ be an odd prime, and let $G$ be a subgroup of $\subseteq\GL_2(\ell)$.
Then $(1,2,1)\in \sig (G)$ if and only if $G$ contains an element of order $\ell$.
\end{lemma}
\begin{proof}
If $G$ contains an element $g$ of order $\ell$ then it is conjugate to $\smallmat{x}{1}{0}{x}$ and $\sig(g^{\ell-1})=(1,2,1)\in G$.
Conversely, if $(1,2,1)\in \sig(G)$ then $G$ contains an element conjugate to $\smallmat{1}{1}{0}{1}$, which has order $\ell$.
\end{proof}

\begin{lemma}\label{lem:sig}Suppose $G$ and $H$ are non-conjugate subgroups of $\GL_2(\ell)$ for which $\sig(G)=\sig(H)$, with $\#G\ge \#H$.
Up to conjugacy in $\GL_2(\ell)$ exactly one of the following holds:
\begin{enumerate}
\item[{\rm (a)}] $G=\langle C,\smallmat{1}{1}{0}{1}\rangle$ and $H=\langle C',\smallmat{1}{1}{0}{1}\rangle$ where $C,C'\subseteq C_s(\ell)$ are distinct $C_s^+(\ell)$-conjugates.
\item[{\rm (b)}] $G\subseteq C_s^+(\ell)$ with $\det(G)\subseteq \F_\ell^{\times 2}$ and $H=G \cap C_s(\ell)\subsetneq G$; in this case $\ell\equiv 1\bmod 4$.
\item[{\rm (c)}] $G\subseteq C_{ns}^+(\ell)$ with $\det(G)\subseteq \F_\ell^{\times 2}$ and $H=G \cap C_{ns}(\ell)\subsetneq G$; in this case $\ell\equiv 3\bmod 4$.
\item[{\rm (d)}] the images of $G$ and $H$ in $\PGL_2(\ell)$ are isomorphic to $\alt{4}$ and $\sym{3}$, respectively.
\end{enumerate}
For every subgroup $G\subseteq \GL_2(\ell)$ there is at most one conjugacy class of non-conjugate subgroups $H$ that have the same signature.
\end{lemma}

\begin{proof}
The four conjugacy classes of subgroups in $\GL_2(2)$ all have distinct signatures, in which case the lemma is vacuously true, so we assume $\ell$ is odd.
The group $G$ contains $\SL_2(\ell)$ if and only if $\sig(G)$ contains $(1,2,1)$ and a triple $(1,t,0)$ with $t^2-4$ not square, and in this case the conjugacy class of $G$ is then determined by $\det(G)$, which is also determined by $\sig(G)$.
The same applies to $H$, so this case cannot arise.
Lemma~\ref{lem:borel121} implies that either $G$ and $H$ both contain an element of order $\ell$, or neither do, and if the former holds than we must be in case (a), by Lemma~\ref{lem:isoborel} and its proof.

We now assume neither $G$ nor $H$ contain an element of order $\ell$.
The scalar subgroup $G\cap Z(\ell)$ of $G$ and the possible orders of all $g\in G$ and $h\in \pi(G)$ are determined by $\sig(G)$, and they must be the same as for $H$.
The groups $\pi(G)$ and $\pi(H)$ cannot both be cyclic, since Corollary~\ref{cor:splitcyclic} and Lemma~\ref{lem:nonsplitcyclic} imply that in this case the conjugacy classes of $G$ and $H$ are determined by their signatures.
Similarly, Corollaries~\ref{cor:nonsplitdihedral}, ~\ref{cor:splitdihedral}, and Lemma~\ref{lem:splitnonsplitdihedral} imply that $\pi(G)$ and $\pi(H)$ cannot both by dihedral.

The group $\sym{4}$ (resp.\ $\alt{5}$) may be distinguished from any cyclic or dihedral group by the fact that it contains elements of order 3 and 4 (resp.\ 3 and 5), but no element of order 12 (resp.\ 15).
For the group $\alt{4}$, the only cyclic or dihedral group with the same set of element orders is $\sym{3}$.
By Lemma~\ref{lem:excep}, the conjugacy class of $G$ in $\GL_2(\ell)$ with $\pi(G)$ isomorphic to $\alt{4}$, $\sym{4}$, or $\alt{5}$ is determined by $\det(G)$ and $G\cap Z(\ell)$, thus the only case that can arise in which $G$ or $H$ has an exceptional projective image is case (d) of the lemma.

The only remaining possibility is that $\pi(G)$ is dihedral and $\pi(H)$ is cyclic (since we assume $\#G\ge \#H$), and $\pi(H)$ cannot be trivial, so $H$ is contained in either a split Cartan group or a non-split Cartan group, but not both.
We have $G\cap Z(\ell)=H \cap Z(\ell)$ with $G$ is distinguished up to conjugacy among subgroups with dihedral projective image by its signature and $H$ distinguished up to conjugacy among subgroups with cyclic projective image image by its signature, and this implies that $G$ must contain an index 2 subgroup conjugate to $H$.
So without loss of generality we assume $H=G\cap C$, where $C$ is either $C_s(\ell)$ or $C_{ns}(\ell)$, and let $\gamma H$ be the non-trivial coset of $H$ in $G$, for some $\gamma \in G-H$.
Now $\pi(H)$ contains an element of order~$2$, since $\pi(G)$ does and their signatures coincide, so $H$ contains a trace-zero element $h$, and every trace-zero element of $H$ is a scalar multiple of $h$.
It follows that either all or none of the trace zero elements in $H$ (and hence in $G$) have square determinants, depending on whether $\det h$ is square or not.

Suppose $\det h$ is not a square.
The same must be true of every element of $\gamma H$ (since they all have trace zero), including $\gamma$, so every element of $\gamma\gamma H=H$ has square determinant; but this includes $h$, a contradiction.
So $h$ and every element of $\gamma H$ has square determinant, including $\gamma $, and the same holds for $\gamma\gamma H=H$ and hence for $G$; thus $\det(G)\subseteq\F_\ell^{\times 2}$, as claimed.

If $H\subseteq C_s(\ell)$ then $h=\diagmat{x}{-x}$ for some $x\in \Z(\ell)^\times$; thus $\det h = -x^2$ is square only if $-1$ is square in $\Z(\ell)^\times$, in which case $\ell\equiv 1\bmod 4$.
If $H\subseteq C_{ns}(\ell)$ then $h=\smallmat{0}{\varepsilon y}{y}{0}$ for some $y\in \Z(\ell)^\times$ with $\varepsilon$ not square; thus $\det h=-\varepsilon y^2$ is square only if $-1$ is not square in $\Z(\ell)^\times$, in which case $\ell\equiv 3\bmod 4$.
\end{proof}

We note that when $\det(G)$ is not contained in the subgroup of squares in $\Z(\ell)^\times$ only case (a) of Lemma~\ref{lem:sig} can arise, and in this case $G$ and~$H$ are isomorphic, by Lemma~\ref{lem:isoborel}.
This yields the following corollary.

\begin{corollary}
Let $E$ be an elliptic curve over a number field $K$ and let $\ell$ be a prime for which $K\cap \Q(\zeta_\ell)=\Q$ (any prime if $K=\Q$).
Then $G_E(\ell)$ is determined up to isomorphism by its signature.
\end{corollary}

To address cases (b), (c), (d) of Lemma~\ref{lem:sig} that may arise when $\det(G)\subseteq \F_\ell^{\times 2}$ we need an additional datum.
For any subgroup $G\subseteq\GL_2(\ell)$, let
\[
z(G):= \frac{\#\{g:g\in G, \tr g = 0\}}{\#G}
\]
denote the proportion of trace-zero elements in $G$.

\begin{lemma}\label{lem:zdiff}
Let $G$ and $H$ be as in Lemma~\ref{lem:sig} and suppose we are not in case {\rm (a)}.
Then
\[
\bigl|z(G)-z(H)\bigr|\ge \frac{1}{4}.
\]
\end{lemma}
\begin{proof}
If we are in case (b) or (c) of Lemma~\ref{lem:sig}, then $H$ lies in a Cartan group $C$ and has index 2 in $G$, and the non-trivial coset $gH$ of $H$ in $G$ does not intersect $C$.
In this case every element of $gH$ has trace zero, while at most half the elements of $H$ can have trace zero, thus
\[
z(G)-z(H)=\frac{1+z(H)}{2} - z(H) = \frac{1-z(H)}{2}\ge \frac{1}{4}.
\]
In case (d) we have $z(G)=1/4$ and $z(H)=1/2$, thus $z(H)-z(G)=1/4$.
\end{proof}

\begin{corollary}\label{cor:sigzlc}
If $G$ and $H$ are subgroups of $\GL_2(\ell)$ with $\sig(G)=\sig(H)$ and $z(G)=z(H)$ then either $G$ and $H$ are conjugate or case $(a)$ of Lemma~\ref{lem:sig} applies.
In particular, $G$ and $H$ are locally conjugate and isomorphic.
\end{corollary}
\begin{proof}
This follows from the previous lemma and Lemma~\ref{lem:isoborel}.
\end{proof}


We now give an efficient algorithm to determine a set of generators for a subgroup $G$ of $\GL_2(\ell)$ that satisfies $\sig(G)=s$ and $z(G)=z$, given the signature $s=\sig(G')$ and trace-zero ratio $z=z(G')$ of some subgroup $G'$ of $\GL_2(\ell)$.
By Corollary~\ref{cor:sigzlc}, the group $G$ must be locally conjugate to $G'$.
In order to do this more efficiently, we note that each signature $s$ is uniquely determined by a small subset of its triples.
It suffices to retain a subset $\overline{s}$ of $s$ of signatures $\sig(g)$ for $g\in G'$ that includes

\begin{itemize}
\item the triple $(1,2,1)$ if $\#G'$ is divisible by $\ell$;
\item a triple $\sig(g)$ for which $\langle\det(g)\rangle = \det(G')=:\det(s)$;
\item a triple $\sig(g)$ for which $\langle g\rangle = Z(G')=:Z(s)$;
\item a triple $\sig(g)$ for which $|\pi(g)|=\max\{|\pi(h)|:h\in G'\}=:m(s)$;
\item triples $\sig(g_i)$ for which $\lcm|\pi(g_i)|=\lcm\{|\pi(h):h\in G'\}=:\lambda(s)$;
\item triples $\sig(g_i)$ for which $\{\chi(g_i)\}=\{\chi(h):h\in G'\}=:\chi(s)$;
\item if $\pi(G')$ is not cyclic, triples $\sig(g_1)$ and $\sig(g_2)$ with $|\pi(g_1)|=|\pi(g_2)|=2$ but $\pi(g_1)\ne\pi(g_2)$.
\end{itemize}

Given any signature $s=\sig(G')$ we can always reduce $s$ to a subset $\overline{s}$ of at most 11 elements that satisfy all of the criteria above.
Alternatively, as we shall do in Section \ref{sec:algorithms}, we can construct $\overline{s}$ by randomly sampling a sufficiently large subset of $s$, without ever needing to store more than $O(\log\ell)$ triples, which requires just $O(\log^2 \ell)$ bits of space, as opposed to $O(\ell^2\log\ell)$ for the entire signature.
More importantly, with the algorithm below we can obtain generators for a subgroup $G$ locally conjugate to $G'$ in expected time polynomial in $\log\ell$ rather than $\ell$, an exponential improvement.
For any subgroup $G$ of $\GL_2(\ell)$ let $Z(G)$ denote the subgroup of scalar elements, and similarly let $Z(s)$ denote the subset of signatures of scalar elements $(d,t,n)$ with $n\in \{0,2\}$ and $t^2-4d=0$.

\begin{algorithm}\label{alg:grpsig}
Given a subset $\overline{s}$ of the signature $s$ of a subgroup $G'$ of $\GL_2(\ell)$ satisfying the requirements above and a rational number $z\in [0,1]$ with denominator at most $\#\GL_2(\ell)$ satisfying $|z(G')-z|<1/8$, output a set of generators for a subgroup $G$ of $\GL_2(\ell)$ that is locally conjugate to $G'$ as follows:
\end{algorithm}
\begin{enumerate}[{\bf 1.}]
\setlength\itemsep{2pt}
\item (\textbf{even $\boldsymbol\ell$}) If $\ell=2$ then output $G=\langle \diagmat{1}{1}\rangle,\ \langle\smallmat{0}{1}{1}{0}\rangle,\ \langle\smallmat{1}{1}{1}{0}\rangle,$ or $\langle\smallmat{0}{1}{1}{0},\smallmat{1}{1}{1}{0}\rangle$ when $\overline{s}$ is equal to $\{(1,0,2)\}$, $\ \{(1,0,2),(1,0,1)\}$, $\ \{(1,0,2),(1,1,0)\}$, or $\ \{(1,0,2),(1,0,1),(1,1,0)\}$, respectively, then terminate.
\item (\textbf{cases with order divisible by $\boldsymbol\ell$}) If $\overline{s}$ contains the triple $(1,2,1)$ then:
\begin{enumerate}[{\bf a.}]
\item (\textbf{contains $\boldsymbol{\SL_2(\ell)}$}) If $-1\in \chi(s)$ output $G=\langle \smallmat{1}{1}{0}{1}, \smallmat{1}{0}{1}{1}, \smallmat{1}{0}{0}{d}\rangle$ with $\langle d\rangle=\det(s)$ and terminate.
\item (\textbf{in $\boldsymbol{B(\ell)}$}) Output $G=\langle \smallmat{1}{1}{0}{1}, g,c\rangle\subseteq B(\ell)$, with $g\in C_s(\ell)$ satisfying $|\pi(g)|=m(s)$ and $\langle c\rangle = Z(s)$, and terminate.
\end{enumerate}
\item (\textbf{exceptional cases}) Check for projective image $\alt{4},\sym{4},\alt{5}$ as follows:
\begin{enumerate}[{\bf a.}]
\item ($\boldsymbol{\alt{4}}$) If $z < 3/8$, $m(s)=3$ and $\lambda(s)=6$, use Algorithm~\ref{alg:excep} to construct $G$ with $\pi(G)\simeq \alt{4}$, $Z(G)=Z(s)$, and $[\det(G):\det(Z(G))]=[\det(s)=\det(Z(s))]$.  Output $G$ and terminate.
\item ($\boldsymbol{\sym{4}}$) If $m(s)=4$ and $\lambda(s)=12$ use Algorithm~\ref{alg:excep} to construct $G$ with $\pi(G)\simeq \sym{4}$, $Z(G)=Z(s)$, and $[\det(G):\det(Z(G))]=[\det(s)=\det(Z(s))]$.  Output $G$ and terminate.
\item ($\boldsymbol{\alt{5}}$) If $m(s)=5$ and $\lambda(s)=30$ use Algorithm~\ref{alg:excep} to construct $G$ with $\pi(G)\simeq \alt{5}$, $Z(G)=Z(s)$, and $[\det(G):\det(Z(G))]=[\det(s)=\det(Z(s))]$.  Output $G$ and terminate.
\end{enumerate}
\item (\textbf{trivial cases}) If $\chi(s)=\{0\}$ output $Z(s)$ and terminate.
\item (\textbf{cyclic cases}) Construct a maximal $H\subset C_s(\ell)\cup C_{ns}(\ell)$ with $\pi(H)$ cyclic such that $\sig(H)\subseteq s$:
\begin{enumerate}[{\bf a.}]
\item Let $\langle c\rangle=Z(s)$ let $g\in C_s(\ell)\cup C_{ns}(\ell)$ satisfy $|\pi(g)|=m(s)$ and $\sig(g)\in \overline{s}$, and set $H=\langle g,c\rangle$.
\item If $s\subseteq \sig(H)$ and $|z(H)-z|<1/8$ then output $G=H$ and terminate.
\end{enumerate}
\item (\textbf{dihedral cases}) Determine the unique $G\supseteq H$ with $\pi(G)$ dihedral such that $\sig(G)=s$:
\begin{enumerate}[{\bf a.}]
\item Let $e=[Z(\ell):H\cap Z(\ell)]$, where $H$ is as in step 5.
\item If $\chi(g)=1$ let $\gamma=\smallmat{0}{1}{1}{0}$ and $r=\smallmat{1}{0}{0}{\varepsilon}$, otherwise let $\gamma=\diagmat{1}{-1}$ and let $r$ be a generator for $C_{ns}(\ell)$.
\item Output whichever of $G=\langle H,\gamma\rangle$ or $G=\langle H, \gamma r^e\rangle$ satisfies $\overline{s}\subseteq \sig(G)$.
\end{enumerate}
\end{enumerate}

The correctness of Algorithm~\ref{alg:grpsig} follows from Proposition~\ref{prop:subgroups}, Lemma~\ref{lem:borel121}, and Corollaries~\ref{cor:nonsplitdihedral},~\ref{cor:splitdihedral}, and~\ref{cor:sigzlc}.
Note that in the dihedral case $\overline{s}$ is guaranteed to contain the signature of some $h\in G-H$, since we retain two projectively distinct elements of order 2 in this case, and $\det h$ will determine whether $\det(G)=-\det(G-H)$ or not, which determines which of the two possible subgroups $G$ to output in step~6c, by Corollaries~\ref{cor:nonsplitdihedral} and~\ref{cor:splitdihedral}.

\begin{proposition}\label{prop:alggrpsigtime}
The expected running time of Algorithm~\ref{alg:grpsig} is $O(\M(\log\ell)\log\ell)$.
\end{proposition}
\begin{proof}
All the individual arithmetic operations in the algorithm involve $O(\log\ell)$-bit integers, including the numerator and denominator of $z$, and can be accomplished using $O(\M(\log\ell)\log\log\ell)$ bit operations (including any field inversions).
The subset $\overline{s}$ contains just $O(1)$ elements, there are $O(1)$ steps in the algorithm, and each can be completed in $O(\M(\log\ell)\log\ell)$ expected time, including the calls to Algorithm~\ref{alg:excep}, by Proposition~\ref{prop:algexceptime}, and the time to obtain a generators $\varepsilon$ for $\Z(\ell)^\times$ and $r$ for $C_{ns}(\ell)$ using a Las Vegas algorithm.
\end{proof}

\subsection{Locally conjugate subgroups}

We conclude this section with a theorem that precisely characterizes the circumstances in which we may have an elliptic curve $E/K$ for which $G_E(\ell)$ is locally conjugate but not conjugate to another subgroup of $\GL_2(\ell)$.

\begin{theorem}\label{thm:locconjisog}
Let $\ell$ be a prime and let $E$ be an elliptic curve over a number field $K$ for which there exists a subgroup $G'$ of $\GL_2(\ell)$ that is locally conjugate to $G_E(\ell)$ but not conjugate to $G_E(\ell)$.
Then $G'$ arises as $G_{E'}(\ell)$ for an elliptic curve $E'/K$ that is related to $E$ by a cyclic isogeny whose degree is a power of $\ell$; the curve $E'$ is unique up to isomorphism.
\end{theorem}
\begin{proof}
It follows from the classification of \S\ref{sec:GL2} that up to conjugacy, $G=G_E(\ell)$ and $G'$ are of the form $G=H\rtimes \langle t\rangle$ and $G'=H'\rtimes\langle t\rangle$, where $t=\smallmat{1}{1}{0}{1}$ and $H$ and $H'$ are distinct subgroups of $C_s(\ell)$ that are conjugate in $\GL_2(\ell)$ via $\smallmat{0}{1}{1}{0}$.  This implies that neither $H$ nor $H'$ lie in $Z(\ell)$.

The group $G$ lies in $B(\ell)$ but not $C_s(\ell)$, so $E$ admits a rational isogeny $\varphi_1$ of degree $\ell$ that is unique up to isomorphism.
Let $E_1=\varphi_1(E)$ and let $G_1=G_{E_1}(\ell)$.
The isogeny $\varphi_1$ induces a homomorphism $G\to G_1$ with kernel $\langle t\rangle$.
The existence of the dual isogeny implies that the order of $G_1$ is either equal to that of~$G$ or smaller by a factor of $\ell$ (it cannot be larger because $\ell^2$ does not divide $\#\GL_2(\ell)$).
In the latter case,~$G_1$ lies in a split Cartan group but is not contained in $Z(\ell)$ (since $H$ is not), and $E_1$ admits exactly two distinct rational $\ell$-isogenies, one of which is the dual of $\varphi_1$.

If we let $\varphi_2\colon E_1\to E_2$ be the rational $\ell$-isogeny that is not dual to $\varphi_1$ and put $G_2=G_{E_2}(\ell)$, then either~$G_2$ also lies in a split Cartan group but not $Z(\ell)$ and we can repeat the same argument, or~$G_2$ has the same order as~$G$.
The isogeny class of $E$ is finite, so by following a chain of $\ell$-isogenies whose composition $\varphi$ has a cyclic kernel of $\ell$-power order, we must eventually reach an elliptic curve $E_n=\varphi_n(E)$ for which $G_n:=G_{E_n}(\ell)$ has the same order as $G$.
We may thus assume that $G_n$ lies in $B(\ell)$ but not $C_s(\ell)$, and therefore has the form $H_n\rtimes\langle t\rangle$, where $H_n$ is a subgroup of $C_s(\ell)$.
The isogeny $\varphi_n$ induces a group homomorphism $\phi_n\colon G\to G_n$ with kernel $\langle t\rangle$.
We can pick bases $(P,Q)$ and $(P',Q')$ for $E[\ell]$ and $E_n[\ell]$ (respectively) so that $\varphi_n(P)=0$ and $\varphi_n(Q)=Q'$, while for the dual isogeny $\hat\varphi_n$ we have $\hat\varphi_n(Q')=0$ and $\hat\varphi_n(P')=P$.
It follows that $\phi_n$ restricts to an isomorphism from $H$ to $H_n$ that corresponds to conjugation by $\smallmat{0}{1}{1}{0}$ (swapping the diagonal elements).
We therefore have $H_n=H'$ and $G_n=G'$.
The curve $E':=E_n$ is determined up to isomorphism by the kernel of the separable isogeny $\varphi_n$, which is in turn determined up to isomorphism by $E$.
\end{proof}

\begin{remark}
The theorem allows for the possibility that $E/K$ has CM, but rarely applies in this case.
When $E/K$ has CM the hypothesis of the theorem is satisfied only when when $\ell$ is ramified in the CM field and the ideal above $\ell$ in the CM field is non-principal (and thus has order 2 in the class group).
This corresponds to an $\ell$-volcano that consists of a single edge; see \cite{sut13a}.
\end{remark}

\begin{example}
Consider the chain of $5$-isogenies $E\longleftrightarrow E_1\longleftrightarrow E'$, where $E$, $E_1$, $E'$ are the elliptic curves over $\Q$ with Cremona labels \href{http://www.lmfdb.org/EllipticCurve/Q/11a3}{\texttt{11a3}}, \href{http://www.lmfdb.org/EllipticCurve/Q/11a1}{\texttt{11a1}}, \href{http://www.lmfdb.org/EllipticCurve/Q/11a2}{\texttt{11a2}}, respectively.
In this example the groups~$G=\langle\diagmat{1}{2}\rangle$ and $H=\langle\diagmat{2}{1}\rangle$ are both conjugate to $G_{E_1}(5)$, while the groups $G_E(5)=\langle G,t\rangle$ and $G_{E'}(5)=\langle H,t\rangle$ are non-conjugate but locally conjugate and isomorphic (as required by Lemma~\ref{lem:isoborel}).
As can be seen from the groups $G_E(5)$ and $G_{E'}(5)$, the elliptic curve $E$ has a rational 5-torsion point, but $E'$ does not.
\end{example}

\section{GRH Bounds}\label{sec:GRHbounds}

By the generalized Riemann hypothesis (GRH) we refer to the assumption that the non-trivial zeros of the Dedekind zeta function of a number field all lie on the critical line $\{s\in \C:\Re(s) = 1/2\}$.
We also recall the logarithmic integral $\Li(x):=\int_2^x dt/\log t$.

\begin{proposition}[Lagarias--Odlyzko, Serre]\label{prop:GRHchebotarev}
Assume the GRH.
Let $L$ be a finite Galois extension of a number field $K$ with Galois group $G=\Gal(L/K)$, let $n_L:=[L:\Q]$, and $d_L:=|\disc(L)|$.
For each nonempty subset $C$ of $G$ stable under conjugation define
\[
\pi_C(x):= \#\left\{\p:\ \left(\frac{L/K}{\p}\right)\subseteq C,\ N(\p)\le x\right\},
\]
where $\p$ ranges over the primes of $K$ that are unramified in $L$, $N(\p)$ is its absolute norm, and $\left(\frac{L/K}{\Cdot}\right)$ is the Artin symbol.
There are absolute effective constants $c_1$ and $c_2$ such that
\[
\left|\pi_C(x)-\frac{\#C}{\#G}\Li(x)\right|\le c_1\frac{\#C}{\#G}\sqrt{x}(\log d_L + n_L\log x)
\]
holds for all $x\ge 2$, and $\pi_C(x) \ge 1$ for all $x\ge c_2\log^2 d_L$.
\end{proposition}
\begin{proof}
The first bound is \cite[Thm.\ 4]{serre81}, which sharpens \cite{lo77}.
The second is \cite[Thm.\ 5]{serre81}, which is also sketched in \cite{lo77}.
For the third bound, see the remark regarding an improvement to Corollary~1.2 in \cite{lo77}.
\end{proof}

\begin{remark}
As noted in \cite{serre81}, Oesterl\'e announced the explicit values $c_1=2$ and $c_2=70$ in \cite{oesterle79}.
Proofs of these values have not been published, but in \cite{winckler13} one can find proofs that use somewhat larger constants (one can take $c_1=185$ via \cite[Thm.\ 1.2]{winckler13}; if one assumes $d_L$ is sufficiently large one can take $c_1\approx 32$).
\end{remark}

\begin{proposition}[Larson--Vaintrob]\label{prop:GRHellbound}
Assume the GRH.
Let $E$ be an elliptic curve without CM defined over a number field $K$, and let $N_E$ be the absolute value of the norm of its conductor.
There is an effective constant $c_K$ depending only on $K$ such that $G_E(\ell)\ne \GL_2(\ell)$ only occurs for primes
\[
\ell \le c_K\log N_E(\log\log N_E)^3.
\]
\end{proposition}
\begin{proof}
See \cite[Thm.\ 2]{lv14}.
\end{proof}

\begin{remark}
Without the GRH the best known bounds on $\ell$ are exponentially worse.
Even in the case $K=\Q$ the best unconditional bound known is quasi-linear in $N_E$ \cite{cojocaru05}.
For elliptic curves over $\Q$ with no primes of multiplicative reduction, an $O(\sqrt{N_E})$ bound is given in \cite{zywina14a}, which also gives much stronger bounds (logarithmic in the discriminant) for elliptic curves with non-integral $j$-invariants.
\end{remark}


\begin{proposition}\label{prop:tordiscbound}
Let $E$ be an elliptic curve defined over a number field $K$, and let $N_E$ be the absolute value of the norm of the conductor of $E$.
Let $m>1$ be an integer, let $L:=K(E[m])$ be the $m$-torsion field of $E$, and let $d_L:=|\disc(L)|$, $d_K:=|\disc(K)|$, and $n_K:=[K:\Q]$, 
Then
\[
\log d_L \le  m^4d_K(4n_K\log_2 m+d_K+1)\log(mN_E).
\]
\end{proposition}
\begin{proof}
We have
\[
d_L = d_K^{[L:K]}|N_{K/\Q}(d_{L/K})|,
\]
where $d_{L/K}$ denotes the relative discriminant of $L/K$.
The extension $L/K$ has degree at most $\#\GL_2(m)$ which is less than $ m^4$, and is unramified at all primes $\p$ of $K$ that do not divide $m$ and for which $E$ has good reduction; see \cite[Thm.\ 1]{dt02}.
The ramification index $e$ of any prime $\q|\p$ cannot exceed $[L:K]< m^4$, therefore the multiplicity of any prime~$\q$ in the relative different $\mathcal{D}_{L/K}$ cannot exceed
\[
e-1+v_\p(e)e <  e(n_K\log_2 e+1) < m^4(4n_K \log_2 m+1) =: B.
\]
The multiplicity of any prime $\p$ in the relative discriminant $d_{L/K}=N_{L/K}(\mathcal{D}_{L/K})$ is also bounded by $B$, and since every ramified prime divides $mN_E$, we have
\[
|N_{K/\Q}(d_{L/K})| \le mN_E.
\]
Thus
\[
\log d_L \le m^4d_K + B\log (mN_E) = m^4(4n_K\log_2 m + d_K + 1)\log (mN_E).\qedhere
\]
\end{proof}

\begin{remark}
The conductor norm $N_E$ can be replaced by its squarefree part in the proposition above.
\end{remark}

\begin{corollary}\label{cor:GRHsamplebounds}
Assume the GRH.
Let $E$ be an elliptic curve defined over a number field $K$ and let $N_E$ be the absolute value of the norm of its conductor.
Let $\ell$ be a prime and let $L=K(E[\ell])$.
There is an effective constant $c_K'$ depending only on $K$ such that every conjugacy class in $G_E(\ell)$ arises as the image of a Frobenius element of $\Gal(L/K)$ for a prime $\p\ndiv \ell$ of good reduction for~$E$ with absolute norm $N(\p)\le x$, provided that
\[
x \ge c_K'\ell^8(\log\ell \log(\ell N_E))^2.
\]
For $\ell \le c_K \log N_E(\log\log N_E)^3$ as in Proposition~\ref{prop:GRHellbound}, it suffices to have
\[
x \ge c_K'(\log N_E)^{10}(\log\log N_E)^4(\log\log\log N_E)^{24},
\]
Moreover, if a good prime $\p\ndiv \ell$ is chosen uniformly at random from the set $\{\p:N(\p)\in[P,2P]\}$ with $P\ge x\log\log x$ and $x$ as above,
then for any nonempty subset $C$ of $G_E(\ell)$ stable under conjugation the probability that $\Frob_\p$ lies in $C$ is
\[
\bigl(1+o(1)\bigr)\frac{\#C}{\#G}\,,
\]
where the implied constant in $o(1)$ is effective.
\end{corollary}
\begin{proof}
Applying Proposition~\ref{prop:tordiscbound} with $n=\ell$ yields $\log d_L=O(\ell^4\log\ell\log(\ell N_E))$, where the implied constant is effective and depends only on $K$.
We then apply the last part of Proposition~\ref{prop:GRHchebotarev} to get the first lower bound on $x$.
The second bound on $x$ follows immediately, and the last statement follows from the upper and lower bounds on $\pi_C(x)$ given by Proposition \ref{prop:tordiscbound} (we just need $P$ to grow strictly faster than $x$).
\end{proof}

\begin{remark}
Analogous results that do not depend on the GRH are known (see \cite{lo77} and \cite{lmo79}, for example), but the bounds are typically polynomial in the absolute discriminant $d_L$, rather than its logarithm.
\end{remark}

\section{Algorithms and Applications}\label{sec:algorithms}

All the fields $k$ that we shall consider are either number fields $K$ or finite fields $\Fq$ of odd characteristic~$p$; in both cases $k$ is a finite extension of its prime field $k_0$ and can be explicitly represented as $k_0[\alpha]/(F(\alpha))$ for some fixed monic polynomial $F\in \Z[\alpha]$ of degree $[k:k_0]$ whose image in $k_0[\alpha]$ is irreducible.
For the purpose of explicit computation, we assume that elements of $k$ are uniquely represented as integer polynomials of degree less than $[k:k_0]$, with coefficients in the interval $[0,p-1]$ in the case that $k_0$ is the finite field $\Fp$.

For number fields $K=\Q[\alpha]/(F(\alpha))$, we assume that the polynomial $F$ is fixed in advance, 
and that elliptic curves $E/K$ are specified by an integral Weierstrass equation $y^2=f(x)$, where $f\in \Z[\alpha][x]$ is a cubic polynomial whose coefficients in $\Z[\alpha]$ represent elements of $K$ as described above.
For each prime~$\p$ of $K/\Q$ that does not divide $\disc(F)$ we may represent the residue field $\F_\p$ of $\p$ as $\Fp[\alpha]/(G(\alpha))$, where $p=\p\cap \Z$ and $G$ divides the image of $F$ in $\Fp[\alpha]$; such a $G$ can be efficiently obtained by factoring $F$ over $\Fp$ (indeed, this is how the $\p|p$ are typically determined; see \cite[\S 4.8.2]{cohen93}, for example).
If $\p$ is a prime of good reduction for $E$, we may compute $E_\p:= E\bmod \p$ by reducing the $\Z[\alpha]$-coefficients of $f(x)$ modulo $(p,G(\alpha))$ to obtain elements of $\F_\p$.

\begin{remark}
We do not assume $O_K=\Z[\alpha]$ (which is possible only when~$\O_K$ is monogenic), so $\disc(F)$ may be divisible by primes that do not divide $\disc(K)$.
Such primes $\p$ are finite in number and there is no harm in ignoring them for the purpose of computing $G_E(\ell)$.
More generally, as we are only interested in primes $\p$ of bounded norm, there is no loss of generality in assuming that $N(\p)=p$ is prime, so that we have $\deg G = 1$ and $\F_\p\simeq \F_p$; this accounts for all but a negligible proportion of the primes $\p$ with $N(\p)\le B$ for any sufficiently large bound $B$.
Doing so simplifies the practical implementation of our algorithms.
\end{remark}

\subsection{Computing Frobenius triples}\label{sec:frobtriples}

Our strategy is to determine the signature of $G_E(\ell)$ by computing the images of Frobenius elements $\Frob_\p$ under $\rho_{E,\ell}$ for primes $\p$ of good reduction for $E$ that do not divide~$\ell$ or $\disc(F)$ (such primes are unramified in both $K(E[\ell])/K$ and $K/\Q$).
This requires us to compute the determinant, trace, and $1$-eigenspace dimension of $\rho_{E,\ell}(\Frob_\p)$.
If we put $q:=N(\p)$, then for any prime $\ell$ not divisible by $\p$, the Frobenius triple
\begin{equation}\label{eq:frobtriple}
\bigl(\det \rho_{E,\ell}(\Frob_\p),\ \tr \rho_{E,\ell}(\Frob_\p),\ \dim_1(\rho_{E,\ell}(\Frob_\p)\bigr)
\end{equation}
of $E/K$ at $\p$ is given by
\[
\bigl(q\bmod \ell,\ \tr \pi_{E_\p} \bmod \ell,\ \ \log_\ell \#E_\p[\ell](\F_\p)\bigr),
\]
where $\tr \pi_{E_\p}:=q+1-\#E_\p(\F_\p)$ is the trace of the Frobenius endomorphism $\pi_{E_\p}$ of $E_\p$.
We can efficiently compute $\tr \pi_{E_\p}$ using Schoof's algorithm \cite{schoof85,schoof95}, which runs in time $(\log q)^{5+o(1)}$ (see \cite[Cor.~11]{ss15} for a sharp bound when $q$ is prime; up to factors of $\log\log q$, the non-prime case is the same).
To compute $\#E_\p[\ell](\F_\p)$ we rely on Miller's algorithm \cite{miller04} for computing the Weil pairing.
Recall that for an elliptic curve $E$ over any field $k$ any prime $\ell\ne {\rm char}(k)$, the Weil pairing
\[
\omega_\ell\colon E[\ell]\times E[\ell]\to \mu_\ell
\]
is a non-degenerate alternating bilinear pairing.
This implies that for any $P,Q\in E[\ell]\simeq \Z(\ell)\times\Z(\ell)$, the points $P$ and $Q$ generate $E[\ell]$ if and only if $\omega_\ell(P,Q)\ne 1$.
In \cite{miller04}, Miller gives an efficient algorithm to compute $\omega_\ell$; when $k=\Fq$ is a finite field and $P,Q$ lie in $E(\Fq)$ it runs in time $(\log q)^{3+o(1)}$.

We now give a Las Vegas algorithm to compute Frobenius triples for a set $S$ of primes $\ell$ for a given reduction $E_\p$ of $E/K$ at an unramified prime $\p$ of norm $q$.
The algorithm can be applied to any elliptic curve over a finite field, but in order to keep the context clear we denote the curve $E_\p/\F_\p$, since we have in mind a reduction of our fixed elliptic curve $E/K$.

\begin{algorithm}\label{alg:frobtriples}
Given an elliptic curve $E_\p$ over a finite field $\F_\p$ of characteristic $p$ and cardinality $q$, and a finite set $S$ of primes $\ell\ne p$,
compute $T=\{\left(\ell,\ q \bmod\ell,\ \tr \pi_{E}\bmod \ell,\ \log_\ell \#E_\p[\ell](\F_\p)\right):\ell \in S\}$ as follows:
\begin{enumerate}[{\bf 1.}]
\setlength{\itemsep}{4pt}
\item Use Schoof's algorithm to compute $t = q+1-\#E_\p(\F_\p)$ and put $N:=q+1-t$.
\item Initialize $T$ to $\{\}$ and for each prime $\ell\in S$:
\begin{enumerate}[{\bf a.}]
\setlength{\itemsep}{2pt}
\item Put $e:=v_\ell(N)$.
\item If $e=0$ then add $(\ell, q\bmod\ell, t\bmod\ell,0)$ to $T$ and proceed to the next prime $\ell\in S$.
\item If $e=1$ or $q\not\equiv 1\bmod \ell$ then add $(\ell,q\bmod \ell,\ t\bmod\ell,\ 1)$ to $T$ and proceed to the next prime $\ell\in S$.
\item Repeat the following:
\begin{enumerate}[{\bf i.}]
\item Generate random points $P_1,P_2\in E_\p(\F_\p)$ and compute $Q_1:=(N/\ell^e)P_1$ and $Q_2:=(N/\ell^e)P_2$.
\item For $i=1,2$, determine the least $e_i\in [0,e]$ such that $\ell^{e_i}Q_i=0$.
\item If $\max(e_1,e_2)=e$ then add $(\ell,\ q\bmod\ell,\ t\bmod\ell,\ 1)$ to $T$ and proceed to the next prime $\ell\in S$.
\item Use Miller's algorithm to compute $\zeta:=\omega_\ell(\ell^{e_1-1}Q_1,\ell^{e_2-1}Q_2)$.
\item If $\zeta\ne 1$ then add $(\ell,\ q\bmod\ell,\ t\bmod\ell,\ 2)$ to $T$ and proceed to the next prime $\ell \in S$.
\end{enumerate}
\end{enumerate}
\item Output $T$ and terminate.
\end{enumerate}
\end{algorithm}

Steps 2.b and 2.c of the algorithm allow us to quickly treat cases where we can immediately determine the $\ell$-rank $r:=\log_\ell\#E_\p[\ell](\F_\p)$: if $\ell$ does not divide $N=\#E_\p(\F_\p)$ (so $e=0$), then clearly $r=0$; if $\ell$ divides $N$ then $r\ge 1$, and we can have $r>1$ only if $\ell^2$ divides $\#E_\p(\F_\p)$ (so $e>1$) and $q\equiv 1\bmod \ell$.

\begin{proposition}\label{prop:frobtriples}
The expected running time of Algorithm~\ref{alg:frobtriples} is
\[
O\left((\log q)^{5+o(1)} + \#S\cdot(\log q)^{3+o(1)}\right).
\]
\end{proposition}
\begin{proof}
As noted above, the cost of running Schoof's algorithm in step 1 is bounded by $(\log q)^{5+o(1)}$.
Generating uniformly random non-trivial points $P\in E_\p(\F_\p)$ in step 2.d.i can be accomplished by repeatedly choosing uniformly random $x_0\in \F_\p$ and attempting to find a root $y_0$ of $y^2-f(x_0)\in \F_\p[y]$; to obtain a uniform distribution over $E_\p(\F_\p)-\{0\}$ one picks the sign of $y_0$ at random and discards points with $y_0=0$ with probability $1/2$.
The expected time per random point $(x_0,y_0)$ is $(\log q)^{1+o(1)}$, which matches the cost of step 2.d.ii.
The time for step 2.d.iv is $(\log q)^{3+o(1)}$, and this dominates the total cost of step 2.d, which we expect to execute less than twice, on average, for each $\ell\in S$.
If $E_\p(\F_\p)[\ell]$ has order $\ell$, then with probability at least $1-1/\ell^2$ one of $Q_1$ or $Q_2$ will be a generator and the algorithm will then proceed to the next $\ell\in S$ in step 2.d.iii; otherwise we have $E_\p[\ell]\subseteq E_\p(\F_\p)$, and with probability at least $1-1/\ell$ the points $Q_1$ and $Q_2$ generate $E_\p[\ell]$ and the algorithm proceeds to the next $\ell \in S$ in step 2.d.v.
The expected time for step 2.d is thus $(\log q)^{3+o(1)}$ for each prime $\ell$, and the total time for step 2 is $\#S\cdot (\log q)^{3+o(1)}$.
\end{proof}

\begin{remark}
By using the Schoof--Elkies--Atkin (SEA) algorithm in step 2 of Algorithm~\ref{alg:frobtriples}, under the GRH one obtains a tighter bound on its average running time for reductions of a fixed elliptic curve $E/K$ modulo primes~$\p$ of $K$ with norm contained in any dyadic interval $[x,2x]$.
An extension of \cite[Cor.\ 3]{ss15} yields an average expected time of
\[
O\left((\log x)^{4+o(1)} + \#S\cdot (\log x)^{3+o(1)}\right)
\]
per prime.
This also applies if we restrict to degree-1 primes, or to primes in an arithmetic progression with a sufficiently small modulus.
\end{remark}

\subsection{Computing Frobenius conjugacy classes}\label{sec:frobconj}
We now give an asymptotically slower algorithm that instead of computing Frobenius triples for a given set of primes computes a single integer matrix
\[
A_\p :=\begin{pmatrix} (a_\p + b_\p\delta_\p)/2& b_\p\\b_\p(\Delta_\p-\delta_\p)/4 & (a_\p-b_\p\delta_\p)/2\end{pmatrix}\in \mat_2(\Z)
\]
whose reduction modulo $m$ lies in the conjugacy class $\rho_{E,m}(\Frob_\p)$ for all integers $m>1$ prime to $\p$ (including all primes~$\ell$ not divisible by $\p$).
The quantities $a_\p, b_\p, \Delta_\p,\delta_\p$ appearing in $A_\p$ are defined as follows.
Let $R_\p:=\End(E_\p)\cap \Q(\pi_{E_\p})$; if $\pi_{E_\p}\in \Z$ then $R_\p=\Z$ and otherwise $R_\p$ is  the centralizer of $\pi_{E_\p}$ in $\End(E_\p)$ and isomorphic to an order in an imaginary quadratic field.
We then define
\[
\Delta_\p := \disc (R_\p),\qquad \delta_\p :=0,1\text{ as }\Delta_\p\equiv 0,1\bmod 4,\qquad a_\p := \tr \pi_{E_\p},\qquad b_\p := \sqrt{(a_\p^2-4N(\p))/\Delta_\p}.
\]
Note that $b_\p=0$ if $R_\p=\Z$ (and in this case $A_\p$ is a scalar matrix), otherwise $b_\p$ is the index of $\Z[\pi_{E_\p}]$ in~$R_\p$.
In either case, we always have
\[
4N(\p)=a_\p^2-b_\p^2\Delta_\p,
\]
with $\tr A_\p = a_\p$ and $\det A_\p = N(\p)\ne 0$.

\begin{theorem}[Duke--T\'oth]
Let $E$ be an elliptic curve over a number field $K$ and let $\p$ be a prime of good reduction for $E$.
For any integer $m$ for which $\p\ndiv m$ is unramified in $K(E[m])$ the reduction of $A_\p$ modulo~$m$ lies in the conjugacy class of $\rho_{E,m}(\Frob_\p)$ in $\GL_2(m)$.
\end{theorem}
\begin{proof}
See \cite[Thm.\ 2.1]{dt02}.
\end{proof}

When $E_\p$ is supersingular, the matrix $A_\p$ is determined by $N(\p)$ and $a_\p$.
This follows from the fact that in this case $\End(E_\p)$ is a maximal order in the quaternion algebra $\End(E)\otimes\Q$, by \cite{deuring41}, hence either $R_\p = \Z$ or $R_\p$ is the maximal order of $\Q(\sqrt{-p})$, where $\p|p$.
In the former case $b_\p=0$ and in the latter case $\Delta_\p=\disc(\Q(\sqrt{-p}))$ and we compute $b_\p$ as above.

To treat the ordinary case, we rely on the algorithm in \cite{bisson11}, which gives a GRH-based Las Vegas algorithm to compute the index $u_\p$ of $\End(E_\p)$ in the maximal order of the imaginary quadratic field $\End(E_\p)\otimes\Q$ with an expected running time of
\[
\mathsf{L}(N(\p))^{1+o(1)},
\]
where
\[
\mathsf{L}(x):=\exp\sqrt{\log x\log\log x}.
\]
The first step of this algorithm is to compute $a_\p$ via Schoof's algorithm and factor $a_\p^2-4N(\p)$ in order to determine the discriminant $D:=\disc(\Q((a_\p^2-4N(\p))^{1/2})$.
Once the index $u_\p$ has been determined we may compute $b_\p=(a_\p^2-4N(\p))/(u_\p^2 D)$.
This yields the following theorem.

\begin{theorem}\label{thm:algAp}
Let $E$ be an elliptic curve over a number field $K$ and let $\p$ be a prime of good reduction for~$E$.
Under the GRH there is a Las Vegas algorithm to compute $A_\p$ in $\mathsf{L}(N(\p))^{1+o(1)}$ expected time.
\end{theorem}

\begin{remark}
An exponential-time algorithm for computing $A_\p$ using Hilbert class polynomials $H_D$ whose discriminants $D$ divide $a_\p^2-4N(\p)$ is given in \cite{centeleghe14}; the running time is not explicitly analyzed in \cite{centeleghe14}, but we note that there are several algorithms to compute Hilbert class polynomials whose running times are quasi-linear in $|D|$, which is close to the bit-size of $H_D$ \cite{bbel08}.
The fastest of these relies on the GRH \cite{sut10}, but the algorithm in \cite{enge09} does not,
and as noted in \cite[Rem.\ 1.1]{streng14}, the heuristics used in \cite{enge09} can be removed.
This gives an unconditional deterministic algorithm to compute $A_\p$ in time $N(\p)^{1+o(1)}$, but this is too slow to be useful to us here (and we require the GRH in any case).
\end{remark}

In terms of its complexity in $q=N(\p)$, the subexponential-time algorithm to compute $A_\p$ is much slower than Algorithm~\ref{alg:frobtriples}, which computes the Frobenius triples $(\det A_\p\bmod\ell, \tr A_\p\bmod\ell,\dim_1 (A_\p\bmod \ell))$ for primes $\ell\in S$ in time polynomial in $\log q$.
However, when $S$ is large (say on the order of $(\log N_E)^{1+o(1)}$) and $q$ is relatively small (say $\log q$ is polynomial in $\log N_E$), the running times are essentially the same, and computing~$A_\p$ gives us more information; in particular, it allows us to distinguish the conjugacy classes of $\diagmat{x}{x}$ and $\smallmat{x}{1}{0}{x}$ in $\GL_2(\ell)$ even when $x\ne 1$, which is not possible with just the Frobenius triple; We shall make use of this in \S\ref{sec:montecarlo}

\subsection{A Las Vegas algorithm}\label{sec:lasvegas}

We now give a Las Vegas algorithm to compute $G_E(\ell)$ up to local conjugacy for all primes $\ell$ up to a given bound $L$ by computing images of Frobenius elements $\Frob_\p$ with $N(\p)$ up to a given bound $P$.
Using the GRH-based bounds of Section \ref{sec:GRHbounds} to determine $L$ and $P$ yields an algorithm whose expected running time is polynomial in $\log \Vert f\Vert$, where $y^2=f(x)$ is an integral defining equation for $E/K$ with $f\in \Z[\alpha][x]$ and $\Vert f\Vert$ is the maximum of the absolute values of the norms of the $\Z[\alpha]$-coefficients of~$f$ (which may also be defined in terms of the integer coefficients of $f$ and $\disc(F)$, where $K=\Q[\alpha]/(F(\alpha))$).

\begin{algorithm}\label{alg:lvfull}
Given an elliptic curve $E\colon y^2=f(x)$ over $K=\Q[\alpha]/(F(\alpha))$ with integral coefficients and bounds $L$ and $P$, compute for each prime $\ell\le L$ a group $G_\ell\subseteq\GL_2(\ell)$ that is locally conjugate to a subgroup of $G_E(\ell)$ and contains a representative of $\rho_{E,\ell}(\Frob_\p)$ for all primes $\p$ of $K$ prime to $\ell\disc(F)$ and of good reduction for~$E$ with $N(\p)\le P$ as follows:
\begin{enumerate}[{\bf 1.}]
\setlength{\itemsep}{4pt}
\item Let $S$ be the set of primes $\ell\le L$, and for each $\ell\in S$ initialize the quantities $s_\ell\leftarrow\{\}$, $c_\ell\leftarrow 0$, $z_\ell\leftarrow 0$.
\item Compute the norm $\Delta_E\in \Z$ of the discriminant of $E$ and the discriminant $d_F\in \Z$ of the polynomial $F$.
\item For each rational prime $p\le P$ that does not divide $\Delta_E$ or $d_F$:
\begin{enumerate}[{\bf a.}]
\item Factor $F(\alpha)\bmod p$ into irreducible $G_1(\alpha),\ldots, G_r(\alpha)\in \Fp[\alpha]$.
\item For each $G_i$ with $\deg G_i\le \log P/\log p$:
\begin{enumerate}[{\bf i.}]
\item Use Algorithm~\ref{alg:frobtriples} to compute the Frobenius triples
\[
\tau_{\ell,\p}:=\bigl(\det\rho_{E,\ell}(\Frob_\p),\ \tr\rho_{E,\ell}(\Frob_\p),\ \dim_1\rho_{E,\ell}(\Frob_\p)\bigr)
\]
for the prime $\p$ of $K$ with residue field $\Fp[\alpha]/G_i(\alpha)$ and each prime $\ell \in S - \{p\}$.
\item For each prime $\ell \in S-\{p\}$ update $s_\ell\leftarrow s_\ell\cup\{\tau_{\ell,\p}\}$ and $c_\ell\leftarrow c_\ell+1$.
\item If $\tr\rho_{E,\ell}(\Frob_\p)=0$ then update $z_{\ell}\leftarrow z_{\ell}+1$.
\end{enumerate}
\end{enumerate}
\item For each prime $\ell\in S$, use Algorithm~\ref{alg:grpsig} to construct generators for a subgroup $G_\ell$ of $\GL_2(\ell)$ with $\sig(G_\ell)=s_\ell$ and $|z(G_\ell)-z_\ell/c_\ell|<1/8$ (if Algorithm 3 fails, report that $P$ is too small and terminate).
\item Output the groups $G_\ell$ (specified by generators) and terminate.
\end{enumerate}
\end{algorithm}

Failure in step 4 can conceivably occur if $s_\ell$ and $z_\ell/c_\ell$ do not actually correspond to a subgroup of $\GL_2(\ell)$, in which case the input to Algorithm~\ref{alg:grpsig} is invalid and this may cause it to fail (an event that can be easily detected), even though it is guaranteed operate correctly on all valid inputs.
This could happen if $P$ is too small for every conjugacy class in $G_E(\ell)$ to be realized as the image of $\Frob_\p$ with $N(\p)\le P$.
The bounds in Section \ref{sec:GRHbounds} allow us to choose $P$ so that such a failure would disprove the GRH.

\begin{theorem}\label{thm:lvfull}
Assume the GRH and let $K=\Q[\alpha]/(F(\alpha))$ be a fixed number field.
There is a Las Vegas algorithm that, given an elliptic curve $E/K$ in integral form $y^2=f(x)$ with $f\in \Z[\alpha][x]$ that does not have complex multiplication, determines for every prime $\ell$ a subgroup $G_\ell\subseteq\GL_2(\ell)$ locally conjugate to $G_E(\ell)$.
The algorithm outputs a bound $L$ for which $G_E(\ell) = \GL_2(\ell)$ for all primes $\ell > L$, and a list of generators for $G_\ell$ for each prime $\ell\le L$.
The expected running time of the algorithm is bounded by
\[
(\log\Vert f\Vert)^{11+o(1)}.
\]
\end{theorem}
\begin{proof}
Under the GRH, Proposition~\ref{prop:GRHellbound} guarantees that we have $G_E(\ell)=\GL_2(\ell)$ for all primes $\ell$ larger than $c_K(\log N_E)(\log\log N_E)^3$, where the constant $c_K$ is effective and $N_E$ is the absolute value of the norm of the conductor of $E$.
By Ogg's formula \cite{ogg67}, $N_E$ is bounded by the norm of the discriminant of $E$, which can be expressed as a polynomial of bounded degree in terms of the coefficients of $f$.
It follows that $\log N_E = O(\log\Vert f\Vert)$, where the implied constant is effective and depends only on $K$.
We may thus take $L=(\log\Vert f\Vert)^{1+o(1)}$ as a bound on the primes $\ell$ that we need to consider.

Since $K$ is fixed, we have $\deg F=O(1)$ and $\log q = O(\log p)$, and all the integers and finite field elements that arise in the algorithm have $O(\log p)$ bits.
Using fast arithmetic, we can assume the cost of each arithmetic operation in~$\Z$, $\F_\p$, $\F_p$ is $(\log p)^{1+o(1)}$; see \cite{gg13}, for example.
Using the Cantor-Zassenhaus algorithm~\cite{cz81}, step 3a takes $O((\log p)^{2+o(1)})$ expected time, by \cite[Thm.~14.14]{gg13}, and the time to reduce $E$ to $E_\p$ is $(\log \Vert f\Vert)^{1+o(1)}$.
The time for step 3b is $O((\log p)^{5+o(1)})$; this follows from \cite[Cor.\ 11]{ss15}, which also applies to the constant degree extension $\F_\p/\F_p$.

For the bound $P$, Corollary~\ref{cor:GRHsamplebounds} implies that we can take $P= (\log \Vert f\Vert)^{10+o(1)}$, where the implied constants are again effective.
Note that by Lemma~\ref{lem:zdiff}, we only need to determine $z(G(E))$ to within $\epsilon < 1/8$.
The running time of step 3 of Algorithm~\ref{alg:lvfull} is then bounded by
\[
(\log\Vert f\Vert)^{10+o(1)}\left((\log\Vert f\Vert)^{1+o(1)}+(\log P)^{5+o(1)}\right) = (\log \Vert f\Vert)^{11+o(1)},
\]
which dominates the cost of the other steps, including the time to determine the primes $\ell \le L$ and $p\le P$.
\end{proof}

\subsection{A Monte Carlo algorithm}\label{sec:montecarlo}

We now give a more efficient Monte Carlo algorithm to solve the same problem.
Although it has a negligible impact on the worst-case asymptotic complexity that we can prove under the GRH, for practical purposes it is better to split the problem into two stages: (1) determine the primes~$\ell$ for which $G_E(\ell)\ne \GL_2(\ell)$, and (2) compute $G_E(\ell)$ up to local conjugacy for each of these primes.
If one assumes that Serre's question has an affirmative answer, meaning that the largest $\ell$ for which $G_E(\ell)\ne \GL_2(\ell)$ is bounded by a constant depending only on $K$, then the exceptional primes $\ell$ are bounded by $O(1)$ for any fixed $K$, but we do not want the correctness of the algorithm to depend on this, so we will typically consider many more primes $\ell$ in stage (1) (up to the GRH bound given by Corollary~\ref{prop:GRHellbound}) than in stage (2).
The key difference is that if $G_E(\ell)=\GL_2(\ell)$, we can unequivocally determine this after computing the image of just $O(1)$ random Frobenius elements, whereas computing $G_E(\ell)\subsetneq \GL_2(\ell)$ up to local conjugacy requires us to compute the image of $O(\ell)$ random Frobenius elements in the worst case.

\begin{proposition}\label{prop:maximal}
Let $\ell>7$ be prime.
A subgroup $G$ of $\GL_2(\F_\ell)$ contains $\SL_2(\F_\ell)$ if and only if it contains elements $g_1,g_2,g_3$ with nonzero trace such that
\begin{enumerate}[{\rm(1)}]
\setlength{\itemindent}{16pt}
\item $\chi(g_1)=+1$;
\item $\chi(g_2)=-1$;
\item $u(g_3)\not\in \{1,2,4\}\text{ and } u(g_3)^2-3u(g_3)+1\ne 0$ $($equivalently, $g_3^e\not\in Z(\ell)$ for $e\le 5)$.
\end{enumerate}
where $\chi(g)=\left(\frac{\tr(g)^2-4\det(g)}{\ell}\right)\in\{0,\pm 1\}$ and $u(g)=\tr(g)^2/\det(g)\in \F_\ell$.
\end{proposition}
\begin{proof}
The reverse implication appears in \cite[Prop.\ 19]{serre72} and follows from Proposition~\ref{prop:subgroups}; (1) and (2) together imply that no conjugate of $G$ lies in $C_s^+(\ell),C_{ns}^+(\ell)$, or $B(\ell)$, and (3) rules out the exceptional cases.
Conversely, for $\ell>7$ there exist $g_1,g_2,g_3\in \SL_2(\F_\ell)$ satisfying conditions (1), (2), (3), respectively.
\end{proof}

Up to constant factors the following proposition is implied by \cite[Thm.\ 5.1]{kz12} (and its proof), but here we give a slightly more precise statement.

\begin{proposition}\label{prop:prob}
Let $\ell> 7$ be prime and let $G$ be subgroup of $\GL_2(\F_\ell)$ containing $\SL_2(\F_\ell)$.
Let $X_1,X_2,\dots$ be a sequence of independent random variables uniformly distributed over $G$.
Let $X$ be the integer random variable for which the event $X=r$ occurs if $r$ is the least integer for which $\{X_1,\dots,X_r\}$ include $g_1,g_2,g_3$ of nonzero trace that satisfy the three criteria of Proposition~\ref{prop:maximal}.
The expected value $\Exp[X]$ of $X$ satisfies $\Exp[X] < 8$, and $\Exp[X]\to 3$ as $\ell\to\infty$.
\end{proposition}
\begin{proof}
We consider the waiting times for each of the conditions (1), (2), (3) in Proposition~\ref{prop:maximal} to be satisfied.
From Table~\ref{table:conjclasses} we see that $\SL_2(\F_\ell)$ contains $(\ell-1)(\ell^2+\ell)/2$ elements 
$g_1$ for which $\chi(g_1)=+1$, of which at most $\ell^2+\ell$ have trace zero.
The same is true of every coset of $\SL_2(\F_\ell)$ in $G$; applying $\#\SL_2(\F_\ell)=\ell^3-\ell$ yields
\[
{}\ \ \frac{\#\{g\in G:\chi(g)=1, \tr(g)\ne 0\}}{\#G}\ge \frac{\ell-3}{2\ell-2}\longrightarrow\frac{1}{2}\qquad\text{as}\qquad\ell\to\infty,
\]
and we note that the LHS is never less than $2/5$ for $\ell \ge 11$.
A similar argument shows that
\[
\frac{\#\{g\in G:\chi(g)=-1, \tr(g)\ne 0\}}{\#G}\ge \frac{\ell-3}{2\ell+2}\longrightarrow\frac{1}{2}\qquad\text{as}\qquad\ell\to\infty,
\]
and the LHS is at least $1/3$ for $\ell \ge 11$.
The events represented by these ratios are disjoint, so with probability approaching $1$ as $\ell\to\infty$, one of them occurs for $X_1$, and the expected waiting time for both to occur approaches~$3$ as $\ell\to\infty$.

The images of $C_s(\ell)\cap \SL_2(\F_\ell)$ and $C_{ns}(\ell)\cap\SL_2(\F_\ell)$ in $\PSL_2(\F_\ell)$ are cyclic groups of order $(\ell-1)/2$ and $(\ell+1)/2$, respectively, and the same applies to their conjugates.
In each of these groups there are only $10$ elements of order at most $5$, hence these occur with probability approaching $0$ as $\ell \to\infty$.
Switching to a coset of $\SL_2(\F_\ell)$ and considering images in $\PGL_2(\F_\ell)$ can only decrease the probability of getting an element of order at most $5$.
On the other hand, every $g\in G$ with $\chi(g)=\pm 1$ lies in a conjugate of $C_s(\ell)$ or $C_{ns}(\ell)$, and we have already noted that the probability that $X_1$ is such an element approaches $1$ as $\ell\to\infty$.
Thus with probability approaching $1$ as $\ell\to\infty$, condition (3) is satisfied by $X_1$ and this implies $\Exp[X]\to 3$.

A direct calculation shows that for $\ell > 7$ the probability that $X_1$ satisfies both conditions (2) and (3) is never less than $1/6$, and since (1) and (2) are disjoint, the expected waiting time for either (1) or both (2) and (3) to be satisfied is bounded by $30/17 < 2$, and this implies $\Exp[X] < 2+6= 8$.
\end{proof}

For $\ell\le 7$ we rely on the following proposition.

\begin{proposition}\label{prop:probsmall}
Let $G$ be a subgroup of $\GL_2(\ell)$.
For $\ell=2$ the group $G$ contains $\SL_2(2)$ if and only if it contains $g_1,g_2$ with $\tr(g_1)=1$ and $\dim_1(g_2)=1$.
For $\ell>2$ the group $G$ contains $\SL_2(\ell)$ if and only if it contains $g_1,g_2$ with $\chi(g_1)=-1$, $\chi(g_2)=0$ and $\dim_1(g_2)=1$.
\end{proposition}
\begin{proof}
The case $\ell=2$ is easily checked, so we assume $\ell>2$.
For the ``if" direction, we note that the criteria for~$g_1$ ensure that $G$ is not contained in a Borel group or in the normalizer of a split Cartan.  For $\ell > 5$ the fact that $g_2$ has projective order divisible by $\ell$ rules out exceptional subgroups, and for $\ell=3,5$ every exceptional subgroup containing an element of order $\ell$ also contains $\SL_2(\ell)$.
For the ``only if" direction, we note that $\SL_2(\ell)\cap C_{ns}(\ell)$ has order $\ell+1$ and thus contains a non-scalar element $g_1$ with $\chi(g_1)=-1$, and $\SL_2(\ell)\cap B(\ell)$ has order divisible by $\ell$ and contains a non-scalar element $g_2$ with $\chi(g_2)=0$ and $\dim_1(g_2)=1$.
\end{proof}

If one defines the integer random variable $X$ as in Proposition~\ref{prop:prob} using the criterion that $\{X_1,\ldots,X_r\}$ contains $g_1,g_2$ as in Proposition~\ref{prop:probsmall}, it is easy to show that $\Exp[X]<\ell+2$.
In particular, $\Exp[X] < 9$ for $\ell \le 7$.

With these results in hand we now give a Monte Carlo algorithm for determining the set of primes $\ell$ for which $G_E(\ell)$ does not contain $\SL_2(\ell)$.
Note that when $G_E(\ell)$ contains $\SL_2(\ell)$ we can determine $G_E(\ell)$ exactly by computing the intersection of $K$ with the cyclotomic field $\Q(\zeta_\ell)$, a computation that does not depend on $E$ and takes negligible time for any fixed number field $K$.

\begin{algorithm}\label{alg:mcxell}
Given an elliptic curve $E\colon y^2=f(x)$ over $K=\Q[\alpha]/(F(\alpha))$ with integral coefficients and bounds $P>L\ge 5$, attempt to determine the set of primes $\ell\le L$ for which $\SL_2(\ell)\not\subseteq G_E(\ell)$ as follows:
\end{algorithm}
\begin{enumerate}[{\bf 1.}]
\setlength{\itemsep}{4pt}
\item Initialize $S\leftarrow \{\ell\le L \text{ prime}\}$ and create a table $T$ with boolean entries $T_{\ell,1,},T_{\ell,2},T_{\ell,3}$ set to $0$ for each $\ell\in S$, then set $T_{\ell,3}\leftarrow 1$ for $\ell\le 7$.
\item Compute the norm $\Delta_E\in \Z$ of the discriminant of $E$ and the discriminant $d_F\in \Z$ of the polynomial $F$.
\item Repeat the following $27\lceil1+\log_3 M\rceil$ times, where $M=\#\{\ell\le L\text{ prime}\}$:
\begin{enumerate}[{\bf a.}]
\item Pick a random prime $p\in [P,2P]$ that does not divide $\Delta_E$ or $d_F$ and a random prime $\p$ of $K$ lying above $p$ and use Algorithm~\ref{alg:frobtriples} to compute Frobenius triples
\[
\tau_{\ell,\p}:=\left(\det \rho_{E,\ell}(\Frob_\p),\ \tr \rho_{E,\ell}(\Frob_\p),\ \dim_1\rho_{E,\ell}(\Frob_\p)\right)
\]
for each prime $\ell\in S$.
\item For each prime $\ell\in S$, set $T_{\ell,i}\leftarrow 1$ if $\tau_{\ell,\p}$ matches the conjugacy class of some $g_i\in \GL_2(\ell)$ satisfying~(i) of Proposition~\ref{prop:maximal} (for $\ell > 7$) or Proposition~\ref{prop:probsmall} (for $\ell\le 7$); if $T_{\ell,1},T_{\ell,2},T_{\ell,3}=1$, remove $\ell$ from~$S$.
\end{enumerate}
\item Output the set $S$ and terminate.
\end{enumerate}

\begin{remark}\label{rem:lvmc}
As written this is not (strictly speaking) a Monte Carlo algorithm, since it uses Algorithm~\ref{alg:frobtriples}, which is a Las Vegas algorithm (meaning that is running time is potentially unbounded, even though its expected running is bounded by Proposition~\ref{prop:frobtriples}).
This distinction has no practical relevance, but for the sake of staying consistent with our terminology, let us assume that Algorithm~\ref{alg:mcxell} automatically terminates Algorithm~\ref{alg:frobtriples} if its actual running time exceeds its expected running time by an unreasonable factor, and terminates with failure in this case.
Doing so decreases the probability of success only negligibly and we can easily keep it above $2/3$.
\end{remark}

\begin{theorem}\label{thm:mcxell}
Assume the GRH and let $K=\Q[\alpha]/(F(\alpha))$ be a fixed number field.
There is a Monte Carlo algorithm with one-sided error that, given a non-CM elliptic curve $E/K$ in integral form $y^2=f(x)$ with $f\in \Z[\alpha][x]$, determines the set $S_E$ of primes $\ell$ for which $G_E(\ell)$ does not contain $\SL_2(\ell)$ with probability greater than $2/3$.
The running time of the algorithm is bounded by
\[
(\log\Vert f\Vert)^{1+o(1)},
\]
and the set $S$ it outputs always contains $S_E$.
\end{theorem}
\begin{proof}
We use Algorithm~\ref{alg:mcxell} with the modification indicated in Remark~\ref{rem:lvmc}.
Under the GRH we may take $L=(\log\Vert f\Vert)^{1+o(1)}$, by Proposition~\ref{prop:GRHellbound}, and we may choose $P$ so that $\log P=O(\log L$).
It is clear from Propositions~\ref{prop:maximal} and~\ref{prop:probsmall} that the set $S$ output by Algorithm~\ref{alg:mcxell} always contains $S_E$.
Each call to Algorithm~\ref{alg:frobtriples} in step 3a then takes $O((\log\Vert f\Vert)^{1+o(1)})$ time, and these calls dominate the total running time.
After 27 iterations in step 3, for each prime $\ell\le L$ not in $S_E$, the probability that $\ell$ remains in $S$ is less than $1/3$ (this follows from Proposition~\ref{prop:prob} and the remark following Proposition~\ref{prop:probsmall}, since we always have $\Exp[X]<9$).
After all $27\lceil 1+\log_3 M\rceil$ iterations, this probability is less than $1/(3M)$, and a union bound shows that the probability that any prime $\ell\le L$ not in $S_E$ (of which there at most $M$) remains in $S$ is less than $1/3$.
\end{proof}

\begin{remark}
To amplify the success probability of Algorithm~\ref{alg:mcxell} we run it repeatedly and take the intersection of all the sets $S$ output by the algorithm as our final result.
\end{remark}

We now give a Monte Carlo algorithm to compute $G_E(\ell)$ up to local conjugacy for a given set of primes~$\ell$.
Rather than attempting to compute the full signature $s$ of each $G_E(\ell)$, we rely on the fact that $s$ can be compactly represented by a subset $\overline{s}$ containing at most 11 triples, as explained in \S\ref{sec:grpsig}.
Since we are sampling elements of $s$ randomly, we have no way of knowing \emph{a priori} whether a given triple necessarily belongs to $\overline{s}$.
Instead, we dynamically construct an approximation to $\overline{s}$ that we update whenever we find a triple that does belong to the minimal signature compatible with our current approximation; for example, whenever we find a triple whose projective order exceeds $m(s)=\max\{|\pi(g)|:g\in G_E(\ell)\}$ or does not divide $\lambda(s)=\lcm\{|\pi(g)|:g\in G_E(\ell)\}$.
When doing so we simultaneously remove any triples that are no longer necessary.
Depending on the order in which we find elements, it may happen that the cardinality of our approximation to $\overline{s}$ temporarily exceeds 11, but its cardinality is always bounded by $O(\log \ell)$ and will eventually be no greater than 11.

\begin{algorithm}\label{alg:mcfull}
Given an elliptic curve $E\colon y^2=f(x)$ over $K=\Q[\alpha]/(F(\alpha))$ with integral coefficients, a bound $P$, and a nonempty set $S$ of primes less than $P$, attempt to compute $G_E(\ell)$ up to local conjugacy for each prime $\ell\in S$ as follows:
\begin{enumerate}[{\bf 1.}]
\setlength{\itemsep}{4pt}
\item Initialize variables $\overline{s}_\ell\leftarrow\{\}$, $c_\ell=0$, $z_\ell\leftarrow 0$ for each $\ell\in S$.
\item Compute the norm $\Delta_E\in \Z$ of the discriminant of $E$ and the discriminant $d_F\in \Z$ of the polynomial $F$.
\item Repeat the following $9\max(S)\lceil 1+\log\#S\rceil$ times:
\begin{enumerate}[{\bf a.}]
\item Pick a random prime $p\in [P,2P]$ that does not divide $\Delta_E$ or $d_F$, a random prime $\p$ of $K$ above~$p$, and compute the matrix $A_\p$ as in Theorem~\ref{thm:algAp}.
\item For each prime $\ell\in S$ dividing $(\tr A_\p)^2-4N(\p)$, determine whether the order of $A_\p\bmod\ell$ is divisible by $\ell$ and if so, add the triple $(1,2,1)$ to $\overline{s}_\ell$.
\end{enumerate}
\item Repeat the following $9\lceil 60 + 2\lceil 1+\log\log(1+\max(S))\rceil \lceil 1+ \log\#S\rceil$ times:
\begin{enumerate}[{\bf a.}]
\item Pick a random prime $p\in [P,2P]$ not in $S$, a random prime $\p$ of $K$ above $p$, and compute the integer matrix $A_\p$ as in Theorem~\ref{thm:algAp}.
\item For each prime $\ell \in S$:
\begin{enumerate}[{\bf i.}]
\item Compute $A=A_\p\bmod \ell \in \GL_2(\ell)$, set $A\leftarrow A^\ell$, update $\overline{s}_\ell$ to reflect the triple $(\det A,\tr A,\dim_1\! A)$.
\item Increment $c_\ell$, and if $\tr A= 0$ then increment $z_\ell$.
\item Set $A\leftarrow A^{|\pi(A)|}$ and update $\overline{s}_\ell$ to reflect the triple $(\det A,\tr A,\dim_1 A)$.
\end{enumerate}
\end{enumerate}
\item If the cardinality of any of the sets $\overline{s}_\ell$ exceeds 11, return to step 3.
\item For each prime $\ell\in S$, use Algorithm~\ref{alg:grpsig} to construct generators for a subgroup $G_\ell$ of $\GL_2(\ell)$ for which $s:=\sig(G_\ell)$ satisfies $\overline{s}=\overline{s}_\ell$ and $|z(G_\ell)-z_\ell/c_\ell|<1/8$ (if this fails for any reason, return to step 2).
\item Output the groups $G_\ell$ (specified by generators) and terminate.
\end{enumerate}
\end{algorithm}

\begin{remark}
The constants in steps 3 and 4 are larger than necessary, and for practical implementation we note that steps 3 and 4 can be combined; we have written the algorithm this way in order to simplify the complexity analysis below.
We also assume that Algorithm~\ref{alg:mcfull} is modified as in Remark~\ref{rem:lvmc} to terminate the Las Vegas algorithm used to compute $A_\p$ if its running time exceeds its expected running time by an unreasonable factor; this ensures that the running time of Algorithm~\ref{alg:mcfull} is bounded.
\end{remark}

\begin{theorem}\label{thm:mcfull}
Assume the GRH.
Let $K=\Q[\alpha]/(F(\alpha))$ be a fixed number field, let $E/K$ be an elliptic curve in integral form $y^2=f(x)$ with $f\in \Z[\alpha][x]$, let~$S$ be a set of primes $\ell\le L$ that contains $S_E$, with $L=(\log N_E)^{1+o(1)}$ as in Proposition~\ref{prop:GRHellbound}, and let $P = (\log N_E)^{10+o(1)}$ be as in Corollary~\ref{cor:GRHsamplebounds}.
Given inputs $E$, $P$, and $S$, Algorithm~\ref{alg:mcfull} correctly determines $G_E(\ell)$ up to local conjugacy for all $\ell\in S_E$ with probability greater than $2/3$, and its running time is bounded by
\[
(\log\Vert f\Vert)^{1+o(1)}.
\]
\end{theorem}
\begin{proof}
As argued in the proof of Theorem~\ref{thm:lvfull}, we have $\log N_E=O(\log \Vert f\Vert)$, and this implies $\log P=O(\log\log \Vert f\Vert)$.
It follows from Theorem~\ref{thm:algAp} that the time to compute $A_\p$ for any prime $\p$ with $N(\p)\in [P,2P]$ is bounded by $(\log \Vert f\Vert)^{o(1)}$.
The number of primes dividing $(\tr A_\p)^2-4N(\p)$ is bounded by $\log P=O(\log\log \Vert f\Vert)$, and it follows that the total time for step 3 is bounded by $O((\log \Vert f\Vert)^{1+o(1)})$, and this also applies to step 4.
The cost of updating $\overline{s}_\ell$ is negligible because the cardinality of $\overline{s}_\ell$ is bounded by a constant factor of $\log\ell \le \log P=(\log\log\Vert f\Vert)$, and computing $A^\ell$ can be accomplished in time $O(\M(\log\ell)\log\ell)$, which is also polynomial in $\log\log\Vert f\Vert$.
The time for the check in step 5 is quasi-linear in $\#S=O((\log\Vert f\Vert)^{1+o(1)})$, the time for step 6 is bounded by $O(\#S(\log P)^{1+o(1)})=O((\log\Vert f\Vert)^{1+o(1)})$, by Proposition~\ref{prop:algexceptime}, and this also bounds the time for step 7.
This addresses the bound on the running time of Algorithm~\ref{alg:mcfull}, it remains only to show that its output is correct with probability greater than 2/3.

Let $\ell\in S$ be a prime greater than $5$ for which $G_E(\ell)$ has order divisible by $\ell$.
The proportion of elements of $G_E(\ell)$ of order divisible by $\ell$ is at least $1/\ell$, since $G_E(\ell)$ does not contain $\SL_2(\ell)$ and must therefore either lie in a Borel group or be an exceptional group whose image in $\PGL_2(\ell)$ has order divisible by $\ell=3,5$ (the claim holds in either case).
After $3\max S$ iterations of step 3 the probability that $(1,2,1)\not\in s_\ell$ is less than $1/10$,
and after $9\max S\lceil 1+\log\#S\rceil$ iterations the probability that $(1,2,1)\not\in s_\ell$ for any $\ell\in S$ for which $G_E(\ell)$ has order divisible by $\ell$ is less than $1/10$.

The fact that step 4.b.iii is executed at least $18\lceil 1+ \log \#S\rceil$ times ensures that the probability that for some $\ell\in S$ the set $\overline{s}_\ell$ does not contain the triple of a generator for the scalar subgroup of $G_E(\ell)$ is very small, say less than $1/1000$.
The same comment applies to the probability that $\overline{s}_\ell$ does not contain a triple whose determinant generates $\det(G_E(\ell))$ for some $\ell\in S$.

For each $\ell\in S$, after $3\cdot 60\cdot \lceil 1+\log \#S\rceil$ iterations of step 4 the probability that we have not encountered representative $A$ in step 4.b.i for the projective image of every element of $G_E(\ell)$ in the case that $G_(\ell)$ is an exceptional subgroup is less than $1/10$, and after the completion of step 4 the probability that this is true for any $\ell\in S$ is less than $1/10$.
Similarly, for each $\ell\in S$, after $6\lceil 1+\log\log(1+\max(S))\rceil$ iterations of step~4 the probability that we have not encountered an $A$ in step 4.b.i that has maximal projective order in the image of $G_E(\ell)$ under the $\ell$-power map is less than $1/10$, and after the completion of step~4 the probability that this is true for any $\ell\in S$ is less than $1/10$.

Additionally, after the completion of step 4 the probability that for some $\ell\in S$ for which $G_E(\ell)$ has dihedral projective image the set $s_\ell$ does not contain the signature of some $h\in G_E(\ell)$ whose projective image is not contained in the subgroup generated by some $g\in G_E(\ell)$ of maximal projective order whose signature lies in $s_\ell$ is negligible, say less than $1/1000$.
Finally, we note that the probability that $|z(G_E(\ell))-z_\ell/c_\ell|\ge 1/8$ for any $\ell\in S$ after the completion of step 3 is also negligible, say less than $1/1000$.

Taking a union bound, it follows that the probability that at the end of step~3 any of the sets $\overline{s}_\ell$ does not satisfy all the criteria listed in \S\ref{sec:grpsig} for a suitable representative subset of $s=\sig(G_E(\ell))$ is less than $0.304<1/3$, and this also bounds the probability that any $\overline{s}_\ell$ has cardinality greater than 11.
Thus we expect to return to step 4 in step 5 just $O(1)$ times, and when we reach step 6 we will compute subgroups $G_\ell$ that are locally conjugate to $G_E(\ell)$ for all $\ell \in S$ with probability greater than 2/3.
\end{proof}

Unlike the Las Vegas algorithm given in \S\ref{sec:lasvegas}, our Monte Carlo algorithm explicitly relies on the use of a compact representation $\overline{s}_\ell$ of the signature of $G_E(\ell)$ that contains only a bounded number of triples (at most~11, as noted in \S\ref{sec:grpsig}), and on the fact that we can compute $A_\p$ in subexponential time; both are crucial to obtaining a quasi-linear running time.

\subsection{Distinguishing locally conjugate subgroups}\label{sec:distinguish}

As written, our algorithms cannot distinguish non-conjugate subgroups $G$ and $G'$ of $\GL_2(\ell)$ that are locally conjugate.
However, as noted in Remark~\ref{rem:onlycase}, up to conjugacy the only case in which this can occur is when $G$ and $G'$ are subgroups of the form $G=\langle H,t\rangle$ and $G'=\langle H',t\rangle$, where $t=\smallmat{1}{1}{0}{1}$ and $H$ and $H'$ are subgroups of the split Cartan group $C_s(\ell)$ that are conjugate via $s=\smallmat{0}{1}{1}{0}$ (so $H'$ is $H$ with the diagonal entries swapped).
As proved in Theorem~\ref{thm:locconjisog}, if $G=G_E(\ell)$ for some elliptic curve $E/K$, then $G'=G_{E'}(\ell)$ for an elliptic curve $E'/K$ isogenous to $E$ that we can obtain by following a uniquely determined path of $\ell$-isogenies with $E$ and $E'$ as endpoints.
In most cases the curves $E$ and $E'$ are distinguished by the degrees of the minimal extensions of $K$ over which they acquire a rational point of order~$\ell$.
In terms of the groups $G:=G_E(\ell)$ and $G':=G_{E'}(\ell)$, these are precisely the indices $d_1(G)$ and $d_1(G')$ of the largest subgroups of~$G$ and $G'$ that stabilize a nonzero vector; these indices necessarily divide $\ell-1$, and in most cases they are distinct.
In this section we give a Monte Carlo algorithm to compute $d_1(G)$ that runs in quasi-cubic time, using the fact that $E$ admits a unique rational isogeny of degree $\ell$.

\begin{remark}
Even when $d_1(G)=d_1(G')$, after twisting $E$ and $E'$ appropriately (as described in \S\ref{sec:twists}), we may obtain a pair of elliptic curves $\tilde E$ and $\tilde E'$ for which $\tilde G := G_{\tilde E}(\ell)$ and $\tilde G':= G_{\tilde E'}(\ell)$ are again locally conjugate, but with $d_1(\tilde G)\ne d_1(\tilde G')$.
We are then able to distinguish $G$ and $G'$ by computing $d_1(\tilde G)$ and~$d_1(\tilde G')$.
This technique allowed us to distinguish every pair of locally conjugate groups that we encountered in our computations (see \S\ref{sec:results}), but we note that there are subgroups $G$ and $G'$ of $\GL_2(\ell)$ to which it cannot be applied (the smallest example with surjective determinants occurs when $\ell=29$).
\end{remark}


We begin with a general result that was mentioned in the introduction.
Recall that for each elliptic curve $E\colon y^2=x^3+Ax+B$ and integer $m$ there is a square-free polynomial $f_{E,m}(x)$ with coefficients in $\Z[A,B]$ whose roots are the $x$-coordinates $x(P)$ of the nonzero points $P\in E[m]$, called the $m$-\emph{division polynomial} of~$E$.  For even integers $m$ the factor $x^3+Ax+B$ is typically removed from $f_{E,m}(x)$, in which case its roots are the $x$-coordinates of the points $P\in E[m]-E[2]$.
More generally, one can remove the factor $f_{E,m'}(x)$ for each maximal proper divisor $m'$ of $m$.
We refer to the resulting polynomial $g_{E,m}(x)$ as the \emph{primitive} $m$-\emph{division polynomial} of $E$, which we note has the same splitting field as $f_{E,m}(x)$; the roots of $g_{E,m}(x)$ are the $x$-coordinates of the points in $E[m]$ of order $m$.
The polynomials $f_{E,m}$ and $g_{E,m}$ can be efficiently computed using well-known recursive formulas \cite{McKee94}.

\begin{lemma}\label{lem:quadsplit}
Let $E$ be an elliptic curve over a number field $K$, let $m>2$ be an integer, and let $L$ be the splitting field of the $m$-division polynomial $f_{E,m}(x)$ over $K$.
If $G_E(m)$ contains~$-1$ then $K(E[m])$ is a quadratic extension of $L$, and otherwise $K(E[m])=L$.
\end{lemma}
\begin{proof}
We first note that $\rho_{E,m}$ induces an isomorphism $\Gal(K(E[m])/K)\overset{\sim}{\longrightarrow} G_E(m)$ by restricting each $\sigma\in \Gal(\Kbar/K)$ to $K(E[m])\subseteq \Kbar$.
Let $\{P,Q\}$ be a basis for $E[m]$ as a $\Z/m\Z$-module and consider the subgroup $H\subseteq G_E(m)$ corresponding to the inclusion of Galois groups
\[
\Gal(K(E[m])/L)\subseteq \Gal(K(E[m])/K).
\]
For each $\sigma \in H$ we have $\sigma(P)\in E[m]$ and $x(\sigma(P))=x(P)$, and similarly for $Q$ and $P+Q$.
This implies $\sigma(P)=\pm P$, $\sigma(Q)=\pm Q$,  and $\sigma(P)+\sigma(Q)=\sigma(P+Q)=\pm (P+Q)$, and therefore $\rho_{E,m}(\sigma)=\pm 1$; so $H\subseteq \{\pm 1\}$.
If $-1\in G_E(m)$ then $H=\{\pm 1\}$, since $\rho_{E,m}^{-1}(-1)$ fixes $L$, and otherwise $H$ is trivial.
\end{proof}

\begin{corollary}\label{cor:X1degneg1}
Let $E$ be an elliptic curve over a number field $K$, let $m>2$ be an integer, let $g_{E,m}(x)$ be the primitive $m$-division polynomial of $E$, and let $d$ be the minimal degree of a factor of $g_{E,m}(x)$ in $K[x]$.
If $G_E(m)$ contains $-1$ then $d_1(G_E(m))=2d$.
\end{corollary}
\begin{proof}
We assume $E\colon y^2=x^3+Ax+B$ is in short Weierstrass form.
Let $P\in E[m]$ be a point of order~$m$ whose $x$-coordinate $x(P)$ is a root of a minimal degree factor of $g_{E,m}(x)$.
Then $[K(x(P))\!:\!K]= d$, and $[K(P)\!:\!K(x(P))]\le 2$ since $y(P)^2\in K(x(P))$.
If $-1\in G_E(m)$ then $[K(P)\!:\!K(x(P))]=2$, since $\sigma\rho_{E,m}^{-1}$ fixes $K(x(P))$ but acts non-trivially on $K(P)$ (indeed, $\sigma(y(P))=y(-P)=-y(P)\ne y(P)$ for $m>2$).
\end{proof}

\begin{example} The converse of Corollary~\ref{cor:X1degneg1} is false; the curve \href{http://www.lmfdb.org/EllipticCurve/Q/14a3}{\texttt{14a3}} gives a counterexample with $m=3$. \end{example}

Locally conjugate subgroups of $\GL_2(\ell)$ necessarily have the same scalar subgroups, so having determined $G_E(\ell)$ up to local conjugacy, we know whether or not it contains $-1$.
As noted above, we are specifically interested in the case where $G_E(\ell)$ is a Borel subgroup (so $E$ admits a rational isogeny of degree $\ell$).

In what follows, the \emph{degree} of a point $P\in E[m]$ is the degree of the extension $K(P)/K$ obtained by adjoining the coordinates of $P$ to $K$; equivalently, it is the degree of the minimal extension $L/K$ for which $P\in E[m](L)$.
In terms of $G_E(m)\subseteq\Aut(E[m])$, the degree of $P$ is the index of its stabilizer in $G_E(m)$.
The quantity $d_1(G_E(m))$ is simply the minimal degree of a point of order $m$.

\begin{lemma}\label{lem:isogd1}
Let $E$ be an elliptic curve over a number field $K$ that admits a unique rational isogeny $\varphi$ of prime degree $\ell$.
The points in $E[\ell]$ of degree $d_1(G_E(\ell))$ all lie in the kernel of $\varphi$.
\end{lemma}
\begin{proof}
We may assume that $G_E(\ell)$ lies in the Borel group $B(\ell)$ and contains $\smallmat{1}{1}{0}{1}$; it cannot lie in the split Cartan $C_s(\ell)$ because $E$ admits only one rational isogeny of degree $\ell$ (up to composition with an isomorphism).
The kernel of $\varphi$ consists of the points $P\in E[\ell]$ whose stabilizer in $G_E(\ell)$ contains $\smallmat{1}{1}{0}{1}$.
The orbit of any $P\in\ker\varphi$ under the action of $G_E(\ell)$ has cardinality at most $\ell-1$, since $\ker\varphi$ is Galois-stable and contains only $\ell-1$ nonzero points; the stabilizer of $P$ therefore has index at most $\ell-1$, and it follows that $d_1(G)\le \ell-1$, since $\ker\varphi$ contains points of order $\ell$.  The stabilizer of any $P\in E[\ell]$ of degree less than~$\ell$ must contain $\smallmat{1}{1}{0}{1}$, otherwise its index would be at least $\ell$, so every point of degree $d_1(G)$ is in $\ker\varphi$.
\end{proof}

For a rational isogeny $\varphi$ of prime degree $\ell>2$, let $h_\varphi\in K[x]$ denote the \emph{kernel polynomial} whose roots are the distinct $x$-coordinates $x(P)$ of the points $P\in\ker\varphi\subseteq E[\ell]$; it is a divisor of the $\ell$-division polynomial $f_{E,\ell}(x)$.
The kernel polynomials $h_\varphi$ play a key role in Elkies' improvement to Schoof's algorithm \cite{elkies95,schoof95}; the degree of $h_\varphi(x)$ is just $(\ell-1)/2$, compared to $(\ell^2-1)/2$ for $f_{E,\ell}(x)$.

\begin{corollary}\label{cor:isogd1}
Let $E$ be an elliptic curve over a number field $K$ that admits a unique rational isogeny $\varphi$ of prime degree $\ell>2$, and let $d$ be the minimal degree appearing of a factor of $h_\varphi(x)$ in $K[x]$.
Then $d_1(G_E(\ell))\in \{d,2d\}$, and if $G_E(\ell)$ contains~$-1$ then $d_1(G_E(\ell))=2d$.
\end{corollary}
\begin{proof}
The kernel of $\varphi$ has prime order $\ell$, hence it is generated by any nonzero $P\in \ker\varphi$.
By the previous lemma, these $P$ all have degree $d_1(G_E(\ell))$; let us pick one.
The cyclic group $\langle P\rangle$ is invariant under the action of $\Gal(K(E[\ell])/K)$, so $K(P)/K$ is a cyclic Galois extension, and it contains the splitting field of $h_\varphi(x)$ over $K$, which must be equal to $K(x(P))$, an extension of degree $d$.
Thus
\[
d_1(G_E(\ell))=[K(P)\!:\!K]=[K(P)\!:\!K(x(P))]\cdot [K(x(P))\!:\!K]
\]
is either $d$ or $2d$, depending on whether $y(P)$ lies in $K(x(P))$, or a quadratic extension of $K(x(P))$.
If $G_E(\ell)$ contains $-1$ then the latter must hold, by Corollary~\ref{cor:X1degneg1}.
\end{proof}

The kernel polynomial $h_\varphi(x)$ can be computed via Elkies' algorithm (see \cite[Alg.\ 27]{galbraith12}), which uses the classical modular polynomial $\Phi_\ell\in \Z[X,Y]$ that is a canonical model for the modular curve $X_0(\ell)$.
Under the GRH the polynomial $\Phi_\ell(X,Y)$ can be computed in $\ell^{3+o(1)}$ expected time \cite{bls12}.
By Proposition~\ref{prop:GRHellbound}, for elliptic curves $E$ without complex multiplication, we may assume that $\ell$ is bounded by $(\log \Vert f\Vert)^{1+o(1)}$, where $y^2=f(x)$ is an integral equation for $E/K$.
This yields a reasonably efficient algorithm to compute $h_\varphi(x)$, but factoring $h_\varphi(x)$ in $K[x]$ may be much more time-consuming; the complexity bounds in \cite{landau85} for factoring polynomial in $\O_K[x]$ give a running time of $(\log\Vert f\Vert)^{11+o(1)}$.

We can do much better than this by instead working modulo random primes $\p$ of $K$.
As noted in the proof of Corollary~\ref{cor:isogd1}, the Galois group $\Gal(L/K)$ of the splitting field $L$ of $h_\varphi(x)$ over $K$ is cyclic, and this implies that we can compute the degree $L/K$ by computing $h_\varphi(x)$ modulo several random primes $\p$ and factoring the result over $\F_\p:=\O_K/\p$ (and we can restrict to degree-1 primes $\p$); taking the least common multiple of the degrees of the factors will yield the degree of $L/K$ with high probability (by the Chebotarev density theorem).
Under the GRH it suffices to use $\p$ with $\log N(\p)$ on the order of $\log\Vert f\Vert^{1+o(1)}$; with probability greater than 1/2 just two primes~$\p$ are already enough to determine $[L\!:\!K]$.

The algorithm in \cite{sut13b} gives an efficient method to directly compute instantiated modular polynomials $\Phi_\ell(j(E),Y)$ modulo $\p$, as well as instantiated derivatives of~$\Phi_\ell(X,Y)$ that are required by Elkies' algorithm, allowing us to perform all our computations in finite fields $\F_\p$.
The expected time to compute the reduction of $h_\varphi$ in $\F_\p[x]$ is then bounded by $(\log \Vert f\Vert)^{3+o(1)}$, which also bounds the expected time to factor it in $\F_\p[x]$ using standard probabilistic algorithms (see \cite[Thm.\ 14.14]{gg13}).

Having computed $d=[L\!:\!K]$, it remains only to determine whether $d_1(G_E(\ell))$ is equal to $d$ or $2d$.  If $-1\in G_E(\ell)$ then Corollary~\ref{cor:isogd1} immediately implies the latter, and otherwise it suffices to determine whether the algebraic integer $f(\alpha)$ is a square in $\O_L$, where $\alpha$ is a root of the monic polynomial $h_\varphi(x)$; this computation can be efficiently accomplished via Hensel lifting and is dominated by the time to compute~$h_\varphi(x)$.
The following proposition summarizes our discussion.

\begin{proposition}\label{prop:dist}
Let $E\colon y^2=f(x)$ be a non-CM elliptic curve over a number field, and suppose that $E$ admits a unique rational isogeny of degree $\ell$.
Under the GRH there is a Monte Carlo algorithm to compute $d_1(G_E(\ell))$ whose running time is bounded by $(\log \Vert f\Vert)^{3+o(1)}$.
\end{proposition}

\begin{remark}
We can easily determine ahead of time whether or not computing $d_1(G_E(\ell))$ will distinguish two locally conjugate possibilities $G$ and $G'$ for $G_E(\ell)$.
As noted above, we may assume that $G$ and~$G'$ lie in the Borel group $B(\ell)$ and are thus upper triangular, in which case $d_1(G)$ can be computed more efficiently.
\end{remark}

\subsection{Quadratic twists}\label{sec:twists}

Recall that if $E/K$ is an elliptic curve and $F$ is a quadratic extension of $K$, an elliptic curve $E'/K$ whose base change to $F$ is isomorphic to that of $E$ is a \emph{quadratic twist} of $E$.
Up to $K$-isomorphism, for each quadratic extension $F/K$ there is a unique elliptic curve $E^F$ that is not $K$-isomorphic to~$E$.
Concretely, if $E$ is defined by the equation $y^2=f(x)$ and $F=K(\sqrt{d})$, then $dy^2=f(x)$ is an equation for $E^F$; we assume throughout this section that $E$ and $E^F$ are defined by equations of this form.

We wish to consider the relationship between the Galois images $G_E(\ell)$ and $G_{E^F}(\ell)$.
For $\ell=2$ we always have $G_E(\ell)=G_{E^F}(\ell)$, since $E[2]=E^F[2]$, so we assume $\ell>2$.
Most of our results in fact apply to any integer $m>2$, so we will work in this generality.
The $m$-torsion points of $E$ and $E^F$ differ only in their $y$-coordinates, thus the splitting fields of the $m$-division polynomials $f_{E,m}(x)$ and $f_{E^F,m}(x)$ are identical; let $L$ denote this field.
It follows from Lemma~\ref{lem:quadsplit} that either the $m$-torsion fields $K(E[m])$ and $K(E^F[m])$ are both quadratic extensions of $L$ (the generic case), one is equal to $L$ and the other is a quadratic extension, or both are equal to $L$.
Which case occurs depends on whether both, one, or neither of the groups $G_E(m)$ and $G_{E^F}(m)$ contain $-1$.

\begin{lemma}\label{lem:twistfields}
Let $E$ be an elliptic curve over a number field $K$, let $F$ be a quadratic extension of $K$, let $m>2$ be an integer, and let $L$ be the splitting field of the $m$-division polynomial of $E$.
Then $-1\not\in G_{E^F}(m)$ if and only if $K(E[m])$ is the compositum of $F$ and $L$.
\end{lemma}
\begin{proof}
Let $F=K(\sqrt{d})$, $E\colon y^2=f(x)$, and $E^F\colon dy^2=f(x)$, and let $\varphi\colon (x_0,y_0)\mapsto (x_0,y_0/\sqrt{d})$ be the isomorphism between the base changes of $E$ and $E^F$ to $F$.
We first suppose that $K(E[m])$ is the compositum of $F$ and $L$ and show that $-1\not\in G_{E^F}(m)$.
If $F\subseteq L$ then $K(E[m])=L$ and the base changes of $E$ and $E^F$ to $L$ are isomorphic, hence $K(E^F[m])=L$ and $G_{E^F}(m)$ does not contain $-1$, by Lemma~\ref{lem:quadsplit}.
If $F\not\subseteq L$, then $K(E[m])=L(\sqrt{d})$ and $K(E^F[m])\subseteq L(\sqrt{d})$, and we claim that in fact $K(E^F[m])=L$.
Let $\sigma$ be the non-trivial element of $\Gal(L(\sqrt{d})/L)$, corresponding to $-1\in G_E(m)$.
Then $\sigma(\sqrt{d})=-\sqrt{d}$ and $\sigma(y_0)=-y_0$ for any nonzero $P=(x_0,y_0)\in E[m]$; it follows that $\sigma$ fixes $\varphi(P)$, thus $K(E^F[m])=L$ and $-1\not\in G_{E^F}(m)$.

We now suppose that $K(E[m])$ is not the compositum of $F$ and $L$ and show that $-1\in G_{E^F}(m)$.
If $F\subseteq L$ then $K(E[m])$ is a quadratic extension of $L$ and the base changes of $E$ and $E^F$ to $L$ are isomorphic; we cannot have $K(E^F[m])= L$, since this would imply $K(E[m])= L$.
If $F\not\subseteq L$ then $F\not\subseteq K(E[m])$ and we cannot have $K(E^F[m])=L$, since this would imply $\sqrt{d}$ and therefore $F$ is contained in $K(E[m])$.
Thus in either case $K(E^F[m])\ne L$, and this implies $-1\in G_{E^F}(m)$, by Lemma~\ref{lem:quadsplit}.
\end{proof}

\begin{corollary}\label{cor:twistfields}
Let $E$ be an elliptic curve over a number field $K$, let $F$ be a quadratic extension of $K$, let~$m>2$ be an integer, let $L$ be the splitting field of the $m$-division polynomial of $E$, and let $G:=\langle G_E(m),-1\rangle$.
\begin{enumerate}
\item[{\rm (a)}] If $-1\in G_E(m)$ then $G_{E^F}(m)$ is conjugate in $\GL_2(m)$ to either $G$ or an index $2$ subgroup of $G$ that does not contain $-1$; the latter occurs precisely when $F$ is a subfield of $K(E[m])$ not contained in $L$.
\item[{\rm (b)}] If $-1\not\in G_E(m)$ then $G_{E^F}(m)$ is conjugate in $\GL_2(m)$ to either $G$ or an index $2$ subgroup of $G$ that does not contain $-1$; the latter occurs precisely when $F$ is a subfield of $L$.
\end{enumerate}
\end{corollary}
\begin{proof}
Let $F$, $E$, $E^F$, and $\varphi$ be as in the previous lemma, and let us fix bases for $E[m]$ and $E^F[m]$ as $\Z/m\Z$-modules that are compatible with $\varphi$ after base change.
As an element of $\GL_2(m)$, the action of any $\sigma\in \Gal(\Kbar/K)$ on $E[\ell]$ and $E^F[\ell]$ with respect to our chosen bases can differ only up to sign, thus we may assume $G_E(m)/(G_E(m)\cap \{\pm 1\}) = G_{E^F}(m)/(G_{E^F}(m)\cap \{\pm 1\})$.

We first consider (a), with $-1\in G_E(m)$.
In this case $K(E[m])$ is a quadratic extension of $L$, by Lemma~\ref{lem:quadsplit}.
If $K(E[m])$ is not the compositum of $F$ and $L$, then $G_{E^F}(m)$ contains $-1$ and $K(E^F[m])$ is also a quadratic extension of $L$ (by the previous lemma), and therefore contains $-1$; we thus have $G_{E^F}(m)$ conjugate to $G_E(m)=G$, and either $F$ does not lie in $K(E[m])$ or it is contained in $L$.
If $K(E[m])$ is the compositum of $F$ and $L$, then $K(E[m])=L(\sqrt{d})$ and the previous lemma implies that $-1\not\in G_{E^F}(m)$ and therefore $K(E^F[m])=L$.
The actions of $\Gal(L(\sqrt{d})/K(\sqrt{d}))$ on $E[m]$ and $\Gal(L/K)$ on $E^F[m]$ with respect to our chosen bases commute  with the isomorphism $\varphi$, and it follows that $G_{E^F}(m)$ is conjugate to the index~$2$ subgroup of $G_E(m)=G$ corresponding to $\Gal(L(\sqrt{d})/K(\sqrt{d}))=\Gal(K(E[m])/F)$, which does not contain~$-1$, and this occurs only when $F$ is a subfield of $K(E[m])$ not contained in $L$.

We now consider (b), with $-1\not\in G_E(m)$. in which case $K(E[m]=L$ is a subfield of $K(E^F[m])$, by Lemma~\ref{lem:quadsplit}.
If $F\not\subseteq L$ then $K(E[m])$ is not the compositum of $F$ and $L$ and $-1\in G_{E^F}(m)$, by the previous lemma; by the same argument used above, this implies that $G_E(m)$ is conjugate to an index $2$ subgroup of $G_{E^F}(m)$, and we must have $G_{E^F}(m)$ conjugate to $G$.
If $F\subseteq L$ then $K(E[m])=L$ and $-1\not\in G_E(m)$, and since $K(E[m])$ is the compositum of $F$ and $L$, we also have $-1\not\in G_{E^F}(m)$, by the previous lemma.
So $G_{E^F}(m)=G_E(m)$ is an index 2 subgroup of $G$ not containing $-1$, and this occurs only when $F\subseteq L$.
\end{proof}

In case (b) of Corollary~\ref{cor:twistfields}, when $F$ is a subfield of $L$ the $\ell$-torsion fields of $E$ and its twist $E^F$ coincide, but $E[m]$ and $E^F[m]$ are typically not isomorphic as Galois modules, and $G_E(m)$ and $G_{E^F}(m)$ need not be conjugate (or even locally conjugate) in $\GL_2(m)$, as shown by the following example.

\begin{example}
Let $E/\Q$ be the elliptic curve $y^2+y=x^3-x^2-10x-20$ with Cremona label \href{http://www.lmfdb.org/EllipticCurve/Q/11a1}{\texttt{11a1}}, which we may also write as $y^2=x^3-13392x-1080432$.
Its quadratic twist by $F=\Q(\sqrt{5})$ has Cremona label \href{http://www.lmfdb.org/EllipticCurve/Q/275b2}{\texttt{275b2}}.
The torsion field $\Q(E[5])$ can be written as $\Q[a]/(a^4-a^3+a^2-a+1)$ and is equal to the splitting field $L$ of the 5-division polynomial of $E$.  The field $\Q(E[5])$ contains $F$, so $\Q(E^F[5])=\Q(E[5])$, and $G_E(\ell)$ and $G_{E^F}(\ell)$ are both index 2 subgroups of $G=\langle G_E(\ell),-1\rangle$, but they are not conjugate.
Indeed, one finds that $G_E(\ell)\simeq \langle \diagmat{1}{2}\rangle$ and $G_{E^F}(\ell)\simeq\langle \diagmat{3}{4}\rangle$ are non-conjugate cyclic groups of order 4.
If we instead twist~$E$ by a quadratic field $F'$ not contained in $L$, say $F'=\Q(\sqrt{-3})$, we obtain the elliptic curve with Cremona label \href{http://www.lmfdb.org/EllipticCurve/Q/99d2}{\texttt{99d2}} and find that $G_{E^{F'}}(\ell)$ is conjugate to both $\langle\pm\diagmat{1}{2}\rangle$ and $\langle\pm\diagmat{3}{4}\rangle$.
\end{example}

In the previous example we obtained three non-conjugate subgroups of $\GL_2(5)$ as images of Galois representations arising in a family of quadratic twists of single elliptic curve $E$.
The following lemma shows that for $m=\ell$ prime, up to conjugacy, three is maximal and can occur only when $G_E(\ell)$ lies in a Borel group.

\begin{lemma}\label{lem:3twists}
Let $E$ be an elliptic curve over a number field $K$, let $\ell$ be a prime, and let $n$ be the number of non-conjugate subgroups of $\GL_2(\ell)$ that arise as $G_{E^F}(\ell)$ for some quadratic twist $E^F$ of $E$.
Then $n\le 3$; the case $n=3$ can occur only when $G_E(\ell)$ lies in a Borel group, and the case $n=2$ can occur only when $G_E(\ell)$ lies in either a Borel group or the normalizer of a Cartan group.
\end{lemma}
\begin{proof}
For $\ell=2$ we always have $n=1$, so we assume that $\ell$ is odd and put $G:=\langle G_E(\ell),-1\rangle$.
It follows from Corollary~\ref{cor:twistfields} that $n$ is at most one more than the number of index 2 subgroups of $G$ that do not contain $-1$.
Thus if $G_E(\ell)=G$ contains $-1$ and has no index 2 subgroups that do not contain $-1$, then $n=1$; this applies whenever $G_E(\ell)$ contains $\SL_2(\ell)$ or has projective image isomorphic to $\alt{4}$, $\sym{4}$, or $\alt{5}$ (by Lemma~\ref{lem:excep}).
By Proposition~\ref{prop:subgroups}, we may now assume that $G_E(\ell)$ (and therefore $G$) is contained in either a Borel group or the normalizer of Cartan group (possibly both).

Let us first suppose that the image of $G$ in $\PGL_2(\ell)$ is dihedral; then $G$ is a subgroup of the normalizer $C^+$ of a Cartan group $C$.
If $G_2$ is an index 2 subgroup of $G$ that does not contain $-1$, then $G_2$ also has dihedral image in $\PGL_2(\ell)$.
If we put $H:=G\cap C$ and $H_2:=G_2\cap C$ and apply Lemma~\ref{lem:dihedral}, we must be in case (2a) of the lemma, since $H_2$ does not contain $-1$, and $H_2$ is an index 2 subgroup of $H$ that is normal in $C^+$.
It follows from Corollaries~\ref{cor:nonsplitdihedral} and~\ref{cor:splitdihedral} that $H_2$ determines $G_2$, and $H$ has at most one index~2 subgroup that does not contain $-1$ and is normal in $C^+$, so there is at most one possible $G_2$; thus $n\le 2$.

If $G$ lies in a non-split Cartan group $C_{ns}$ then it has at most one index 2 subgroup, since $C_{ns}$ is cyclic, and we again have $n\le 2$.  Otherwise $G$ lies in a Borel group $B$, which we now assume.
The group $G$ and its index~2 subgroups are uniquely determined by their intersections with the split Cartan group $C_s$ contained in $B$; these are abelian groups, each of which can be written as a product of at most two cyclic groups.
It follows that $G\cap C_s$ has at most three subgroups of index 2.
If it has three, then at least one of them must contain $-1$, since if $H_1$ and $H_2$ are distinct index 2 subgroups of $G\cap C_s$ that do not contain $-1$ then $\langle H_1\cap H_2,-1\rangle$ is an index 2 subgroup that contains $-1$.
Thus $G$ has at most two index 2 subgroups that do not contain $-1$, and we therefore have $n\le 3$.
\end{proof}

\begin{remark}
Lemma~\ref{lem:3twists} does not apply to composite integers $m$.
Indeed, for $m=8$ there may be as many as 20 non-conjugate $G_{E^F}(m)$ that arise as $F$ ranges over quadratic extensions of $K$; see \cite{rzb14} for examples.
\end{remark}

For any subgroup $G$ of $\GL_2(\ell)$ we refer to $\langle G,-1\rangle$ and its index~2 subgroups that do not contain $-1$ as \emph{twists} of $G$ (so $G$ is always a twist of itself).
If $G=G_E(\ell)$ for some elliptic curve $E/K$ then the twists of~$G$ are precisely the subgroups that arise as $G_{E^F}(\ell)$ for some quadratic twist $E^F$ (up to conjugacy in $\GL_2(\ell)$).
Quadratic twists $E^F$ that realize every possibility for $G_{E^F}(\ell)$ can be efficiently constructed using the results in this section.
It suffices to determine the quadratic fields that lie in $K(E[\ell])$ (of which there are at most~3), and to determine which of these quadratic fields lies in the splitting field $L$ of the $\ell$-division polynomial of $E$.
The discriminants of these quadratic fields must divide the discriminant of $K(E[\ell])$, whose prime divisors include only $\ell$ and the primes of bad reduction for $E$.
Provided we can factor the discriminant of $E$, these fields can be determined by simply testing candidate fields $F$ with suitable discriminants by computing $G_{E^F}(\ell)$; in practice this is much faster than attempting to explicitly compute the torsion field $K(G_E(\ell))$ and the quadratic extensions $F/K$ it contains.

\begin{remark}
If $G$ is locally conjugate to $G'$, then each of its twists $H$ is locally conjugate to a corresponding twist $H'$ of $G'$.
If $G=G_E(\ell)$ for some elliptic curve $E/K$, then the twists of $G$ and the twists of any locally conjugate $G'$ all arise as images of Galois representations of elliptic curves defined over $K$.
Thus the discovery of a subgroup $G$ of $\GL_2(\ell)$ that arises as $G_E(\ell)$ may lead directly to as many as 5 other non-conjugate subgroups~$G'$ that arise as the image of Galois representations of curves that are twists of either $E$ or the elliptic curve~$E'$ isogenous to~$E$ given by Theorem~\ref{thm:locconjisog}.
\end{remark}

\begin{example}
Consider the elliptic curve $E/\Q$ with Cremona label \href{http://www.lmfdb.org/EllipticCurve/Q/11a3}{\texttt{11a3}}, which has
$G_E(5)=\langle \diagmat{1}{2} \smallmat{1}{1}{0}{1}\rangle$.
The group $G_E(5)$ has three twists, including itself.
The other two are $\langle G_E(5),-1\rangle$ and its index two subgroup $\langle \diagmat{4}{3}\smallmat{1}{1}{0}{1}\rangle$, which can be obtained as Galois images by twisting $E$ by $\Q(\sqrt{-3})$ and $\Q(\sqrt{5})$, which yields curves with Cremona labels \href{http://www.lmfdb.org/EllipticCurve/Q/99d1}{\texttt{99d1}} and \href{http://www.lmfdb.org/EllipticCurve/Q/275b1}{\texttt{275b1}}, respectively.
The group $G_E(5)$ is locally conjugate to $G_{E'}(\ell)=\langle \diagmat{2}{1}\smallmat{1}{1}{0}{1}\rangle$, where $E'$ has Cremona label \href{http://www.lmfdb.org/EllipticCurve/Q/11a2}{\texttt{11a2}}.
Twisting $E'$ by $\Q(\sqrt{-3})$ and $\Q(\sqrt{5})$ yields curves with Cremona labels \href{http://www.lmfdb.org/EllipticCurve/Q/99d3}{\texttt{99d3}} and \href{http://www.lmfdb.org/EllipticCurve/Q/275b3}{\texttt{275b3}}, respectively, whose Galois images realize the corresponding twists of $G_{E'}(\ell)$.
The six subgroups of $\GL_2(5)$ in this example are non-conjugate and listed in Table~\ref{table:Qpart1} under the labels \texttt{5b.1.1}, \texttt{5B.1.2}, \texttt{5B.1.3}, \texttt{5B.1.4}, \texttt{5B.4.1}, and \texttt{5B.4.2} (the curves listed in Table~\ref{table:Qpart1} for these groups are not all the same as those in this example, some have smaller conductor).
\end{example}

\section{Computational Results}\label{sec:results}

We implemented the algorithms described in Section 5 using the C programming language (as noted earlier, Magma scripts implementing the algorithms in Section 3 are available at \cite{sut15}).
For the computation of Frobenius triples in Algorithm~\ref{alg:frobtriples}, at primes up to $2^{40}$ we relied on the \texttt{smalljac} software library \cite{smalljac} based on the algorithms described in \cite{ks08}, and for larger primes we used the implementation of the SEA algorithm described in \cite{sut13b}.
For the computation of the matrices $A_\p$ described in \S\ref{sec:frobconj} we used a modified version of the algorithm in \cite{bs11} that was optimized for smaller primes, using techniques described in \cite[\S 4]{sut10} and \cite{sut13b}.

As a key practical optimization, we precomputed tables of Frobenius triples for every elliptic curve $E/\Fp$, for primes $p\le 2^{16}$.
This allows us to compute Frobenius triples for the reductions of an elliptic curve $E$ over a number field $K$ at degree-1 primes $\p$ of $K$ with $N(\p) \le 2^{16}$ by simply doing a table lookup; this is particular useful when computing Galois images for large families of elliptic curves.
While $2^{16}$ is typically much smaller than the $(\log N_E)^{10+o(1)}$ bound given by the GRH-based Chebotarev bounds of Corollary~\ref{cor:GRHsamplebounds}, in the typical case where $\rho_{E,\ell}$ is surjective we can usually obtain an unconditional proof of this fact by computing Frobenius triples for just a handful of small primes of good reduction; typically just ten or twenty primes suffice.
This optimization dramatically improves the practical efficiency of our algorithms because it allows us to very quickly determine a small set of primes $S$ that we know contains the set of exceptional primes $S_E$ (the primes $\ell$ for which $G_E(\ell)$ does not contain $\SL_2(\ell)$); this is the main motivation for treating Algorithms~\ref{alg:mcxell} and~\ref{alg:mcfull} separately.

We have applied our algorithms to several large databases of elliptic curves, including:
\begin{itemize}
\setlength{\itemsep}{0pt}
\item \href{http://homepages.warwick.ac.uk/~masgaj/ftp/data/}{Cremona's Elliptic Curve Data} \cite{cremona}, which includes all elliptic curves over~$\Q$ of conductor less than $350,000$ (about 2 million curves);
\item the \href{http://modular.math.washington.edu/Tables/}{Stein-Watkins Table of Elliptic Curves} \cite{sw02}, which includes a large proportion of the elliptic curves over $\Q$ of conductor up to $10^8$, and of prime conductor up to $10^{10}$ (about 140 million curves);
\item the \href{http://beta.lmfdb.org}{$L$-functions and Modular Forms Database (LMFDB)} \cite{lmfdb,lmfdbbeta}, which includes Cremona's tables as well as some 150,000 elliptic curves of small conductor over quadratic and cubic fields.
\end{itemize}
We also analyzed more than $10^9$ elliptic curves of bounded height over $\Q$ and ten quadratic fields (the five real and five imaginary quadratic fields of least absolute discriminant).
In addition to these, we analyzed elliptic curves in families parameterized by various modular curves, including:
\begin{itemize}
\item the modular curves $X_H$ of genus 0 described in~\cite{zywina15};
\item the modular curve $X_{\sym{4}}(7)$ of genus 0 over $\Q(\sqrt{-7})$, using the model in~\cite{kenku85};
\item the modular curve $X_s^+(11)$ of genus 2, using the model in~\cite{banwait13};
\item the modular curve $X_{ns}^+(11)$ of genus 1, using the model in~\cite{cc04};
\item the (isomorphic) modular curves $X_s^+(13)$ and $X_{ns}^+(13)$ of genus 3, using the models given in \cite{baran14};
\item the modular curves $X_0(\ell)$ for primes $11\le \ell \le 61$ of genus up to $5$, using the models provided by the Magma \cite{magma} function \texttt{SmallModularCurve}, as well as quadratic points on these curves found in \cite{bn14}.
\end{itemize}

We restricted our attention to elliptic curves without complex multiplication and used our Monte Carlo algorithm to compute $G_E(\ell)$ up to local conjugacy.
In cases where we were not able to unconditionally prove $G_E(\ell)=\GL_2(\ell)$ we repeated the algorithm 200 times, thereby ensuring (under the GRH) that the probability of error is less than $3^{-200}$.

Having computed $G_E(\ell)$ up to local conjugacy, in each case with two non-conjugate groups $G$ and $G'$ locally conjugate to $G_E(\ell)$ we computed $d_1(G_E(\ell))$ via Proposition~\ref{prop:dist}, and in cases with $d_1(G)\ne d_1(G')$ used this information to determine $G_E(\ell)$ up to conjugacy.
We encountered only one case with $d_1(G)=d_1(G')$, arising for the groups labeled \texttt{11B.10.4} and \texttt{11B.10.5} in Table~\ref{table:Qpart1}, but in this case $G$ and $G'$ have twists that are not locally conjugate, and by twisting $E$ appropriately we were able to determine $G_E(\ell)$ up to conjugacy, as described in \S\ref{sec:twists}.

\begin{remark}
Thanks to recent work by Zywina \cite{zywina15}, for the elliptic curves $E/\Q$ that we found to have exceptional Galois images $G_E(\ell)$, we were able to independently verify our results using his explicit models of modular curves $X_H/\Q$ of prime level that include every subgroup $H$ of $\GL_2(\ell)$ that is known to arise for a non-CM elliptic curve over $\Q$; in no instance did we find an error in our computations.
\end{remark}

\subsection{Results over \texorpdfstring{$\Q$}{\bf Q}}

In total we found 63 exceptional Galois images $G_E(\ell)$ for non-CM elliptic curves $E/\Q$.
These are listed in Tables~\ref{table:Qpart1} and~\ref{table:Qpart2}, along with an elliptic curve of minimal conductor that realizes $G_E(\ell)$.
In collaboration with John Cremona, our results for elliptic curves of conductor up to 350,000 have now been incorporated into Cremona's tables and the LMFDB.

\begin{remark}
Although we analyzed a total of more than $10^{10}$ elliptic curves $E/\Q$, every exceptional $G_E(\ell)$ that we found already occurs for a curve in Cremona's tables; indeed the largest conductor needed to obtain every exceptional $G_E(\ell)$ that we found is $232,544$, which is the conductor of curve listed for the group labeled \texttt{11Nn}.
\end{remark}

\subsection{Results over quadratic fields for elliptic curves with rational \texorpdfstring{$j$}{{\it j}}-invariants}

It follows from Conjecture \ref{conj:Qexcep} that the exceptional Galois images $G_E(\ell)$ that do not contain $\SL_2(\ell)$ that can arise when $E$ is the base change of a non-CM elliptic curve over $\Q$ to a quadratic field  are, up to conjugation in $\GL_2(\ell)$, the 63 exceptional $G_E(\ell)$ that arise over $\Q$ and their subgroups of index 2.
Using Algorithm~\ref{alg:enum}, we can easily enumerate these groups, and we find that up to conjugacy in $\GL_2(\ell)$, there are are 68 groups $G_E(\ell)$ that arise for base changes from $\Q$ to a quadratic field but not over $\Q$.

An elliptic curve $E$ over a quadratic field $K$ whose $j$-invariant lies in $\Q$ is either the base change of an elliptic curve over $\Q$, or a twist of such a curve.
As we are only concerned with elliptic curves without complex multiplication, we can assume $j(E)\not\in \{0,1728\}$ and only need to consider quadratic twists.
It follows from Corollary~\ref{cor:twistfields} that the groups $G_E(\ell)$ that can arise when $E$ is an elliptic curve over a quadratic field with $j(E)\in \Q$ are the groups $G$ that arise for base changes from $\Q$ and their \emph{twists}, as defined in \S\ref{sec:twists}: these are the groups $\langle G,-1\rangle$ and its index 2 subgroups that do not contain $-1$.
A computation shows that, up to conjugation in $\GL_2(\ell)$ and assuming Conjecture \ref{conj:Qexcep}, there are 23 such twists that do not arise for the base change of an elliptic curve over $\Q$,
We thus obtain the following result.

\begin{theorem}\label{thm:basechange}
Assume Conjecture~\ref{conj:Qexcep}.
Up to conjugation in $\GL_2(\ell)$ there are $160$ Galois images $G_E(\ell)$ that do not contain $\SL_2(\ell)$ and arise for non-CM elliptic curves $E$ over quadratic fields with $j(E)\in \Q$ and primes $\ell$; these are listed in Tables~\ref{table:Qpart1}-\ref{table:Qbasechangetwist2}.
Of these, $63$ arise over $\Q$, $68$ arise for base changes of elliptic curves over $\Q$ but not over $\Q$, and $29$ arise only for elliptic curves that are not base changes from $\Q$.
\end{theorem}

Of the 68 exceptional groups that arise for base changes $E_K$ of elliptic curves $E/\Q$ to quadratic fields $K$ (but not over $\Q$), 23 have surjective determinant map (these are listed in Table~\ref{table:Qbasechange1}) and 45 do not (these are listed in Table~\ref{table:Qbasechange2}).
Along with each group we list an elliptic curve $E/\Q$ and the discriminant $D$ of a quadratic field $K$ for which $G_{E_K}(\ell)$ is conjugate to the group listed.
In each case $K$ is a subfield of $\Q(E[\ell])$; taking $D=\bigl(\frac{-1}{\ell}\bigr)\ell$ to be the discriminant of the quadratic subfield of the cyclotomic field $\Q(\zeta_\ell)$) yields the subgroup of $G_E(\ell)$ with square determinants, while any other quadratic subfield $K$ of $\Q(E[\ell])$ yields a group whose determinant map is surjective.

The 29 elliptic curves listed in Tables~\ref{table:Qbasechangetwist1} and~\ref{table:Qbasechangetwist2} are quadratic twists $E_K^F$ of base changes of elliptic curves $E/\Q$ to quadratic fields $K$ by quadratic subextensions $F/K$ of $K(E_K[\ell])/K$ that were computed using the methods described in \S\ref{sec:twists}.

\subsection{Results over quadratic and cubic fields}\label{sec:nfresults}

As noted above, the LMFDB includes tables of elliptic curves of small conductor over various quadratic and cubic fields, including the five real and five imaginary quadratic fields of least absolute discriminant, as well as the cubic field of discriminant $-23$.
The enumeration of modular elliptic curves over the five imaginary quadratic fields $\Q(\sqrt{-1})$, $\Q(\sqrt{-2})$,$\Q(\sqrt{-3})$, $\Q(\sqrt{-7})$, and $\Q(\sqrt{-11})$ was originally addressed by Cremona in \cite{cremona84a,cremona84b} who constructed tables for elliptic curves of conductor norm up to 500; these results have recently extended to conductor norm 10,000 by Cremona and his student Warren Moore.
The tabulation of elliptic curves over the real quadratic field $\Q(\sqrt{5})$ described in~\cite{bdklorss13} has  been extended to conductor norm 5,000, and the LMFDB also contains data for elliptic curves over $\Q(\sqrt{2})$ and $\Q(\sqrt{3})$ to conductor norm 5,000, and over $\Q(\sqrt{13})$ and $\Q(\sqrt{17})$ to conductor norm 2,000 and 1,000, respectively (as of this writing).
In addition, elliptic curves over the cubic field $\Q[a]/(a^3-a^2+1)$ of discriminant $-23$ of conductor norm up to 10,000 are included in the LMFDB, based on the work in \cite{dgky14}.

In total, we computed $G_E(\ell)$ for 115,894 non-CM elliptic curves over these fields that are listed in the LMFDB, as well families of elliptic curves of bounded height, and curves parameterized by points of bounded height on the modular curves listed above.
Tables~\ref{table:Quadratic1}-\ref{table:Quadratic3} list the exceptional groups $G_E(\ell)$ that we found for non-CM elliptic curves over the ten quadratic fields noted above that are not already listed in Tables~\ref{table:Qpart1}--\ref{table:Qbasechangetwist2}.
It follows from \cite{zywina15} and the results of \S\ref{sec:twists} that these groups cannot arise for non-CM elliptic curves over quadratic fields that have rational $j$-invariants (we do not require Conjecture~\ref{conj:Qexcep} here because these groups all lie in the Borel group).

Table~\ref{table:c23} lists the exceptional groups $G_E(\ell)$ that we found for non-CM elliptic curves over the cubic field of discriminant $-23$ that do not already appear in Tables~\ref{table:Qpart1}--\ref{table:Quadratic3}.

\begin{remark}
Unlike the results listed in Tables~\ref{table:Qpart1}-\ref{table:Qbasechangetwist2}, which are complete under Conjecture~\ref{conj:Qexcep}, Tables~\ref{table:Quadratic1}-\ref{table:Quadratic3} are known to be incomplete.
In particular, it follows from \cite[Prop.\ 4.4.8.1]{ligozat77} that there are infinitely many elliptic curves over each of the ten quadratic fields that we consider with $G_E(11)$ conjugate to a subgroup of \texttt{11S4}, but none are listed in our tables.
\end{remark}

\begin{remark}\label{localglobal}
The elliptic curves listed in Table~\ref{table:Qpart1} for the groups labeled \texttt{7Ns.2.1} and \texttt{7Ns.3.1} both have $j$-invariant $2268945/128$ and represent the unique $\Qbar$-isomorphism class of elliptic curves $E/\Q$ that are exceptions to the local-global principle for isogenies \cite{sut12}: each admits a rational 7-isogeny locally everywhere (modulo every prime of good reduction), but not globally (over $\Q$).
The elliptic curve listed in Table~\ref{table:Qbasechange1} for the group labeled \texttt{13A4.1[2]} is the base change of the elliptic curve over $\Q$ listed in Table~\ref{table:Qpart2} for the group labeled \texttt{13S4} to $\Q(\sqrt{13})$; it represents one of five $\Qbar$-isomorphism classes of elliptic curves over $\Q(\sqrt{13})$ that are exceptions to the local-global principle for 13-isogenies \cite[Cor.\ 1.9]{bw14} (three have rational $j$-invariants and two do not).
The elliptic curve listed in Table~\ref{table:Qbasechange1} for the group labeled \texttt{5Ns[2]} is one of infinitely many examples of elliptic curves over $\Q(\sqrt{5})$ with distinct $j$-invariants that admit a 5-isogeny locally everywhere but not globally, as proved in \cite[Thm.\ 1.5]{bw14}, as is the curve listed in Table~\ref{table:Quadratic1} for the group labeled \texttt{5Ns.2.1[2]}.
These curves all have $G_E(5)$ conjugate to \texttt{5Ns[2]} or \texttt{5Ns.2.1[2]}; the former case may arise for the base change of an elliptic curve $E/\Q$ with $G_E(5)$ conjugate to \texttt{5Ns}, while the latter case can only arise only for elliptic curves $E/\Q(\sqrt{5})$ with $j(E)\not\in \Q$.
\end{remark}

\subsection{Group labels}
In the tables that follow conjugacy classes of subgroups $G$ of $\GL_2(\ell)$ are identified by labels of the form
\[
\ell S.a.b.c[d],
\]
where $\ell$ is a prime, $S$ is one of \texttt{G}, \texttt{B}, \texttt{Cs}, \texttt{Cn}, \texttt{Ns}, \texttt{Nn}, \texttt{A4}, \texttt{S4}, \texttt{A5}, while $a$, $b$, $c$ are (optional) nonnegative integers whose meaning depends on $S$, as described below, and $d$ is the index of $\det(G)$ in $\Z(\ell)^\times$; the suffix $[d]$ is omitted when $d=1$.
Let $r$ be the least positive integer that generates the index $d$ subgroup of $\Z(\ell)^\times$.
\begin{enumerate}
\setlength{\itemsep}{4pt}
\item[\texttt{G}:] $G$ contains $\SL_2(\ell)$; the label $\ell\hairspace\texttt{G}$ denotes $\GL_2(\ell)$ and $\ell\hairspace\texttt{G}[d]$ is used when $d=[\GL_2(\ell):G] > 1$.
\item[\texttt{B}:] $G$ is conjugate to a subgroup of $B(\ell)$ that contains an element of order $\ell$,
The label $\ell\hairspace\texttt{B}$ denotes $B(\ell)$ and $\ell\hairspace\texttt{B}[d]$ denotes $\ell\hairspace\texttt{B}\cap \ell\hairspace\texttt{G}[d]$.
The label $\ell\hairspace\texttt{B}.a.b[d]$ denotes the subgroup generated by
\[
\begin{pmatrix}a&0\\0&1/a\end{pmatrix},\ \begin{pmatrix}b&0\\0&r/b\end{pmatrix},\ \begin{pmatrix}1&1\\0&1\end{pmatrix},
\]
where the integers $a,b>0$ are both as small as possible.

\item[\texttt{Cs}:] $G$ is conjugate to a subgroup of $C_s(\ell)$ (including subgroups of $Z(\ell)\subseteq C_s(\ell)$)
The label $\ell\hairspace\texttt{Cs}$ denotes $C_s(\ell)$ and $\ell\hairspace\texttt{Cs}[d]$ denotes $\ell\hairspace\texttt{Cs}\cap \ell\hairspace\texttt{G}[d]$.
The label $\ell\hairspace\texttt{Cs}.a.b[d]$ denotes the subgroup generated by
\[
\begin{pmatrix}a&0\\0&1/a\end{pmatrix},\ \begin{pmatrix}b&0\\0&r/b\end{pmatrix},
\]
with $a,b>0$ minimal.

\item[\texttt{Cn}:] $G$ is conjugate to a subgroup of $C_{ns}(\ell)$ that does not lie in $C_s(\ell)$.
For $\ell=2$ this is the index~2 subgroup of $\GL_2(2)$, which is denoted \texttt{2Cn}.
For $\ell > 2$ the label $\ell\hairspace\texttt{Cn}$ denotes $C_{ns}(\ell)$, and $\ell\hairspace\texttt{Cn}[d]$ denotes $\ell\hairspace\texttt{Cn}\cap\ell\hairspace\texttt{G}[d]$.
The label$\ell\hairspace\texttt{Cn}.a.b[d]$ denotes the subgroup generated by
\[
\begin{pmatrix}a&\varepsilon b\\b&a\end{pmatrix},
\]
with the integers $b>0$, $a\ge 0$ chosen to make $(a,b)$ lexicographically minimal.

\item[\texttt{Ns}:] $G$ is conjugate to a subgroup of $C_s^+(\ell)$ with dihedral projective image.
The label $\ell\hairspace\texttt{Ns}$ denotes $C_s^+(\ell)$, the label $\ell\hairspace\texttt{Ns}[d]$ denotes $\ell\hairspace\texttt{Ns}\cap\ell\hairspace\texttt{G}[d]$, and $\ell\hairspace\texttt{Ns}.a.b[d]$ denotes the subgroup of $C_s^+(\ell)$ generated by
\[
\begin{pmatrix}a&0\\0&1/a\end{pmatrix},\ \begin{pmatrix}0&b\\-r/b&0\end{pmatrix},
\]
with $a$ and $b$ minimal, and $\ell\hairspace\texttt{Ns}.a.b.c[d]$ denotes the subgroup generated by
\[
\begin{pmatrix}a&0\\0&1/a\end{pmatrix},\ \begin{pmatrix}0&b\\-1/b&0\end{pmatrix},\ \begin{pmatrix}0&c\\-r/c&0\end{pmatrix}
\]
with $a,b,c>0$ minimal.

\item[\texttt{Nn}:] $G$ is conjugate to a subgroup of $C_{ns}^+(\ell)$ with dihedral projective image and not conjugate to any subgroup of $C_s^+(\ell)$.
The label $\ell\hairspace\texttt{Nn}$ denotes $C_{ns}^+(\ell)$ and $\ell\hairspace\texttt{Nn}[d]$ denotes $\ell\hairspace\texttt{Nn}\cap\ell\hairspace\texttt{G}[d]$.
The label $\ell\hairspace\texttt{Nn}.a.b[d]$ denotes the subgroup generated by
\[
\begin{pmatrix}a&\varepsilon b\\b&a\end{pmatrix},\ \begin{pmatrix}1&0\\0&-1\end{pmatrix}.
\]
with $(a,b)$ lexicographically minimal, and $\ell\hairspace\texttt{Nn}a.b.c[d]$ denotes the subgroup generated by
\[
\begin{pmatrix}a&\varepsilon b\\b&a\end{pmatrix},\ \begin{pmatrix}1&0\\0&-1\end{pmatrix}\delta^c,
\]
where $\delta=\smallmat{x}{\varepsilon y}{y}{x}$ is any generator for $C_{ns}(\ell)$ and $c=[Z(\ell):G\cap Z(\ell)]$ as in Corollary~\ref{cor:nonsplitdihedral}.

\item[\texttt{A4}:] $G$ has projective image isomorphic to $\alt{4}$ and does not contain $\SL_2(\ell)$.
This requires $d>1$.
The label $\ell$\texttt{A4}$.a[d]$ indicates $[\det(G):\det(Z(G)]=a$ (which must be 1 or 3, the latter only when $\ell\equiv 1\bmod 3$).
Algorithm~\ref{alg:excep} can be used to obtain an explicit set of generators.

\item[\texttt{S4}:] $G$ has projective image isomorphic to $\sym{4}$ and does not contain $\SL_2(\ell)$.
The label $\ell$\texttt{S4} indicates $Z(G)=Z(\ell)$ and $d=1$, while $\ell$\texttt{S4}$[d]$ is used for $d>1$ when $[\det(G):\det(Z(G))]=2$, and $\ell$\texttt{S4}.$1[d]$ is used when $[\det(G):\det(Z(G))] = 1$ (which implies $d>1$).
See Lemma~\ref{lem:excep} for a list of the cases that can occur.
Algorithm~\ref{alg:excep} can be used to obtain an explicit set of generators.

\item[\texttt{A5}:] $G$ has projective image isomorphic to $\alt{5}$.
This requires $\ell\equiv\pm 1\bmod 5$ and $d>1$.
The label $\ell$\texttt{A5}.$[d]$ indicates $[\det(G):\det(Z(G))]=1$ (the only possible case, by Lemma~\ref{lem:excep}).
Algorithm~\ref{alg:excep} can be used to obtain an explicit set of generators.
\end{enumerate}

A magma script that will compute the label of any subgroup of $\GL_2(\ell)$ is available at \cite{sut15}; it also includes a procedure to construct a subgroup based on its label, with generators as above.

\subsection{Tables of exceptional Galois images}
Each of the tables that follow lists the following data:

\begin{itemize}
\setlength{\itemsep}{2pt}
\item the first column lists the label of a group $G\subseteq\GL_2(\ell)$, as defined above, the second columns lists its index in $\GL_2(\ell)$, and the third lists the generators for $G$ as indicated by the label;
\item the column``$-1$" indicates whether the group $G$ contains the scalar matrix $-1$ or not;
\item $t$ is the number of twists the group has (as defined in \S\ref{sec:twists}), equivalently, the number of non-conjugate $G_{E'}(\ell)$ that arise among the twists $E'$ of $E$ (defined over the same field $K$).
\item $d_0$ is the index of the largest subgroup of $G$ that fixes a linear subspace of $\Z(\ell)^2$; equivalently, the degree of the minimal extension over which $E$ admits a rational $\ell$-isogeny.
\item $d_1$ is the index of the largest subgroup of $G$ that fixes a nonzero vector in $\Z(\ell)^2$;  equivalently, the degree of the minimal extension over which $E$ has a rational point of order~$\ell$.
\item $d$ is the order of $G$; equivalently, the degree of the minimal extension $L/K$ for which $E[\ell]\subseteq E(L)$.
\item the curve column lists the Weierstrass coefficients $[a_1,a_2,a_3,a_4,a_6]$ of an integral equation
\[
y^2+a_1xy+a_3y = x^3+a_2x^2+a_4x+a_6
\]
that defines an elliptic curve $E/K$.  When $K \ne \Q$, these may be polynomials in $a\in\O_K$ with minimal polynomial $f(a)$, in which case the curve is listed as $[a_1,a_2,a_3,a_4,a_6]/(f(a))$.
Curves are linked to their entry in the LMFDB, when available

\item for elliptic curves $E$ over quadratic fields with $j(E)\in \Q$ that are not base changes we list $j(E)$.
\item $N$ is the absolute norm of the conductor of the elliptic curve $E$ in factored form.
\item $D$ is the discriminant of the number field $K$ (not listed when $K=\Q$).
\end{itemize}

Pairs of locally conjugate groups are indicated by brackets on the left, and for each such pair the listed curves are related by a chain of $\ell$-isogenies, as in Theorem~\ref{thm:locconjisog}.
Recall from Section \ref{sec:notation} that we view elements of $\Aut(E[\ell])$ as $2\times 2$ matrices that act on column vectors on the \textbf{left} (this distinction is important because many of the groups are not conjugate to their transposes).

\FloatBarrier

\begin{table}
\begin{small}
\setlength{\extrarowheight}{0.8pt}
\begin{tabular}{@{\hskip -6pt}rlrlccrrrll}
&group&index&generators&$-1$&$t$&$d_0$&$d_1$&$d$&curve& $N$\\\toprule
&\texttt{2Cs} & 6 & $ $ & yes & 1 & 1 & 1 & 1 & \href{http://lmfdb.org/EllipticCurve/Q/15/a/5}{$[1,1,1,\smallminus 10,\smallminus 10]$}& $3^{1} 5^{1}$\\ 
&\texttt{2B} & 3 & $\smallmat{1}{1}{0}{1}$ & yes & 1 & 1 & 1 & 2 & \href{http://lmfdb.org/EllipticCurve/Q/14/a/6}{$[1,0,1,4,\smallminus 6]$}& $2^{1} 7^{1}$\\ 
&\texttt{2Cn} & 2 & $\smallmat{0}{1}{1}{1}$ & yes & 1 & 3 & 3 & 3 & \href{http://lmfdb.org/EllipticCurve/Q/196/a/2}{$[0,\smallminus 1,0,\smallminus 2,1]$}& $2^{2} 7^{2}$\\ 
\midrule
&\texttt{3Cs.1.1} & 24 & $\smallmat{1}{0}{0}{2}$ & no & 2 & 1 & 1 & 2 & \href{http://lmfdb.org/EllipticCurve/Q/14/a/6}{$[1,0,1,4,\smallminus 6]$}& $2^{1} 7^{1}$\\ 
&\texttt{3Cs} & 12 & $\smallmat{2}{0}{0}{2},\smallmat{1}{0}{0}{2}$ & yes & 2 & 1 & 2 & 4 & \href{http://lmfdb.org/EllipticCurve/Q/98/a/6}{$[1,1,0,220,2192]$}& $2^{1} 7^{2}$\\ 
\lb &\texttt{3B.1.1} & 8 & $\smallmat{1}{0}{0}{2},\smallmat{1}{1}{0}{1}$ & no & 3 & 1 & 1 & 6 & \href{http://lmfdb.org/EllipticCurve/Q/14/a/5}{$[1,0,1,\smallminus 1,0]$}& $2^{1} 7^{1}$\\ 
&\texttt{3B.1.2} & 8 & $\smallmat{2}{0}{0}{1},\smallmat{1}{1}{0}{1}$ & no & 3 & 1 & 2 & 6 & \href{http://lmfdb.org/EllipticCurve/Q/14/a/2}{$[1,0,1,\smallminus 171,\smallminus 874]$}& $2^{1} 7^{1}$\\
\bottomrule
\end{tabular}
\end{small}
\bigskip
\smallskip
\caption{Exceptional $G_E(\ell)$ for non-CM elliptic curves $E/\Q$ ($\ell\le 11$).}\label{table:Qpart1}
\end{table}

\begin{table}
\begin{small}
\setlength{\extrarowheight}{1pt}
\begin{tabular}{@{\hskip -6pt}rlrlrrrrrll}
&group&index&generators&$-1$&$t$&$d_0$&$d_1$&$d$& curve & $N$\\\toprule
&\texttt{13S4} & 91 & $\smallmat{3}{0}{12}{9},\smallmat{2}{0}{0}{2},\smallmat{9}{5}{0}{6}$ & yes & 1 & 6 & 72 & 288 & \href{http://lmfdb.org/EllipticCurve/Q/50700/z/1}{$[0,1,0,\smallminus 4788,109188]$}& $2^{2} 3^{1} 5^{2} 13^{2}$\\ 
\lb &\texttt{13B.3.1} & 56 & $\smallmat{3}{0}{0}{9},\smallmat{1}{0}{0}{2},\smallmat{1}{1}{0}{1}$ & no & 3 & 1 & 3 & 468 & \href{http://lmfdb.org/EllipticCurve/Q/147/c/2}{$[0,1,1,\smallminus 114,473]$}& $3^{1} 7^{2}$\\ 
&\texttt{13B.3.2} & 56 & $\smallmat{3}{0}{0}{9},\smallmat{2}{0}{0}{1},\smallmat{1}{1}{0}{1}$ & no & 3 & 1 & 12 & 468 & \href{http://lmfdb.org/EllipticCurve/Q/147/c/1}{$[0,1,1,\smallminus 44704,\smallminus 3655907]$}& $3^{1} 7^{2}$\\ 
\lb &\texttt{13B.3.4} & 56 & $\smallmat{3}{0}{0}{9},\smallmat{4}{0}{0}{7},\smallmat{1}{1}{0}{1}$ & no & 3 & 1 & 6 & 468 & \href{http://lmfdb.org/EllipticCurve/Q/24843/d/2}{$[0,1,1,\smallminus 19322,1116938]$}& $3^{1} 7^{2} 13^{2}$\\ 
&\texttt{13B.3.7} & 56 & $\smallmat{3}{0}{0}{9},\smallmat{7}{0}{0}{4},\smallmat{1}{1}{0}{1}$ & no & 3 & 1 & 12 & 468 & \href{http://lmfdb.org/EllipticCurve/Q/24843/d/1}{$[0,1,1,\smallminus 7555032,\smallminus 8001807082]$}& $3^{1} 7^{2} 13^{2}$\\ 
\lb &\texttt{13B.5.1} & 42 & $\smallmat{5}{0}{0}{8},\smallmat{1}{0}{0}{2},\smallmat{1}{1}{0}{1}$ & yes & 1 & 1 & 4 & 624 & \href{http://lmfdb.org/EllipticCurve/Q/2890/l/2}{$[1,\smallminus 1,0,\smallminus 139,965]$}& $2^{1} 5^{1} 17^{2}$\\
\bottomrule
\end{tabular}
\end{small}
\bigskip
\smallskip
\caption{Known exceptional $G_E(\ell)$ for non-CM elliptic curves $E/\Q$ $(\ell > 11)$.}\label{table:Qpart2}
\vspace{10pt}
\end{table}
\smallskip

\begin{table}
\begin{small}
\setlength{\extrarowheight}{1pt}
\begin{tabular}{@{\hskip -6pt}rlrlrrrrrlrr}
&group&index&generators&$-1$&$t$&$d_0$&$d_1$&$d$& curve & $D$\\\toprule
&\texttt{3Cn} & 6 & $\smallmat{1}{2}{1}{1}$ & yes & 1 & 4 & 8 & 8 & $[0,0,1,\smallminus 7,12]$ & $-7$\\ 
\midrule
&\texttt{5Cn.0.1} & 60 & $\smallmat{0}{2}{1}{0}$ & yes & 1 & 2 & 8 & 8 & $[0,0,1,\smallminus 2850,\smallminus 58179]$ & $-3$\\ 
&\texttt{5Cn} & 20 & $\smallmat{3}{2}{1}{3}$ & yes & 1 & 6 & 24 & 24 & $[1,\smallminus 1,1,\smallminus 5,2]$ & $-15$\\ 
&\texttt{5Nn.1.1.1} & 20 & $\smallmat{1}{4}{3}{4},\smallmat{1}{2}{1}{1}$ & yes & 1 & 6 & 24 & 24 & $[1,\smallminus 1,1,\smallminus 5,2]$ & $-3$\\ 
\midrule
&\texttt{7Cs} & 56 & $\smallmat{3}{0}{0}{5},\smallmat{1}{0}{0}{3}$ & yes & 2 & 1 & 6 & 36 & $[0,0,1,2580,549326]$ & $-3$\\ 
&\texttt{7Cn} & 42 & $\smallmat{6}{4}{6}{6}$ & yes & 1 & 8 & 48 & 48 & $[0,\smallminus 1,1,\smallminus 10158,804091]$ & $-23$\\ 
&\texttt{7Nn.1.3} & 42 & $\smallmat{1}{0}{0}{6},\smallmat{1}{2}{3}{1}$ & yes & 1 & 4 & 24 & 48 & $[0,\smallminus 1,1,\smallminus 10158,804091]$ & $161$\\
\bottomrule
\end{tabular}
\end{small}
\bigskip
\smallskip
\caption{Known exceptional $G_E(\ell)$ with surjective determinant for base changes of non-CM elliptic curves $E/\Q$ to quadratic fields $\Q(\sqrt{D})$.}\label{table:Qbasechange1}
\end{table}

\begin{table}
\begin{small}
\setlength{\extrarowheight}{1pt}
\begin{tabular}{@{\hskip -6pt}rlrlrrrrrlr}
&group&index&generators&$-1$&$t$&$d_0$&$d_1$&$d$& curve & $D$\\\toprule
&\texttt{3Cs.1.1[2]} & 48 & $ $ & no & 2 & 1 & 1 & 1 & $[1,0,1,4,\smallminus 6]$ & $-3$\\ 
&\texttt{3Cs[2]} & 24 & $\smallmat{2}{0}{0}{2}$ & yes & 2 & 1 & 2 & 2 & $[1,1,0,220,2192]$ & $-3$\\ 
&\texttt{3B.1.1[2]} & 16 & $\smallmat{1}{1}{0}{1}$ & no & 2 & 1 & 1 & 3 & $[1,0,1,\smallminus 1,0]$ & $-3$\\ 
&\texttt{3Cn[2]} & 12 & $\smallmat{0}{2}{1}{0}$ & yes & 1 & 2 & 4 & 4 & $[1,1,1,3,\smallminus 5]$ & $-3$\\ 
&\texttt{3B[2]} & 8 & $\smallmat{2}{0}{0}{2},\smallmat{1}{1}{0}{1}$ & yes & 2 & 1 & 2 & 6 & $[1,1,1,\smallminus 3,1]$ & $-3$\\ 
&\texttt{3Nn[2]} & 6 & $\smallmat{2}{2}{2}{1},\smallmat{0}{1}{2}{0}$ & yes & 1 & 4 & 8 & 8 & $[0,0,1,\smallminus 7,12]$ & $-3$\\ 
\midrule
&\texttt{5Cs.1.1[2]} & 240 & $\smallmat{1}{0}{0}{4}$ & no & 2 & 1 & 1 & 2 & $[0,\smallminus 1,1,\smallminus 10,\smallminus 20]$ & $5$\\
\bottomrule
\end{tabular}
\end{small}
\bigskip
\smallskip
\caption{Known exceptional $G_E(\ell)$ with non-surjective determinant for base changes of non-CM elliptic curves $E/\Q$ to quadratic fields $\Q(\sqrt{D})$.}\label{table:Qbasechange2}
\end{table}

\begin{table}
\begin{small}
\setlength{\extrarowheight}{0.7pt}
\begin{tabular}{@{\hskip -6pt}rlr@{\hskip 16pt}lrrrrr@{\hskip 16pt}lr@{\hskip 16pt}l}
&group&{\hskip 8pt}index&generators&{\hskip 8pt}$-1$&{\hskip 8pt}$t$&{\hskip 8pt}$d_0$&{\hskip 8pt}$d_1$&{\hskip 8pt}$d$& $j$-invariant &{\hskip 10pt}$D$&$N$\\&curve\\\toprule
&\texttt{7Cs.2.1} & 112 & $\smallmat{2}{0}{0}{4},\smallmat{1}{0}{0}{3}$ & no & 2 & 1 & 3 & 18 & $2^{18} 3^{3} 5^{1} 43^{3}41^{\smallminus 7}$ & $-3$ & $5^{4} 7^{2} 41^{2}$\\&\multicolumn{11}{l}{$[0,0,\smallminus 1,21500(3a\smallplus 5),152590625a\smallplus 129702031]/(a^2\smallminus a\smallplus 1)$}\vspace{1pt}\\ %
\midrule
\blb &\texttt{13B.1.1} & 168 & $\smallmat{1}{0}{0}{2},\smallmat{1}{1}{0}{1}$ & no & 3 & 1 & 1 & 156 & $\smallminus 3^{3} 131^{3}2^{\smallminus 13} 5^{\smallminus 1}$ & $17$ & $2^{2} 5^{2}$\\&\multicolumn{11}{l}{$[1,\smallminus 1,1,\smallminus 131a\smallminus 205,1758a\smallplus 2745]/(a^2\smallminus a\smallminus 4)$}\vspace{2pt}\\ 
&\texttt{13B.1.2} & 168 & $\smallmat{2}{0}{0}{1},\smallmat{1}{1}{0}{1}$ & no & 3 & 1 & 12 & 156 & $\smallminus 3^{3} 118691^{3}2^{\smallminus 1} 5^{\smallminus 13}$ & $17$ & $2^{2} 5^{2}$\\&\multicolumn{11}{l}{$[1,\smallminus 1,1,\smallminus 118691a\smallminus 185455,\smallminus 31941270a\smallminus 49878411]/(a^2\smallminus a\smallminus 4)$}\vspace{3pt}\\
\bottomrule
\end{tabular}
\end{small}
\bigskip
\caption{Known exceptional $G_E(\ell)$ for non-CM elliptic curves $E$ over quadratic fields $\Q(\sqrt{D})$ with $j(E)\in\Q$ that are not base changes from $\Q$ $(\ell\le 17)$.}\label{table:Qbasechangetwist1}
\end{table}

\begin{table}
\begin{small}
\setlength{\extrarowheight}{1pt}
\begin{tabular}{@{\hskip -6pt}rlr@{\hskip 16pt}lrrrrr@{\hskip 16pt}lr@{\hskip 16pt}l}
&group&{\hskip 8pt}index&generators&{\hskip 8pt}$-1$&{\hskip 8pt}$t$&{\hskip 8pt}$d_0$&{\hskip 8pt}$d_1$&{\hskip 8pt}$d$& $j$-invariant &{\hskip 10pt}$D$&$N$\\&curve\\\toprule
\blb &\texttt{37B.10.1} & 456 & $\smallmat{10}{0}{0}{26},\smallmat{1}{0}{0}{2},\smallmat{1}{1}{0}{1}$ & no & 3 & 1 & 3 & 3996 & $\smallminus 7^{1} 11^{3}$ & $5$ & $7^{4}$\\&\multicolumn{11}{l}{$[a\smallplus 1,a,a,\smallminus 78a\smallminus 78,418a\smallplus 333]/(a^2\smallminus a\smallminus 1)$}\vspace{4pt}\\ 
&\texttt{37B.10.2} & 456 & $\smallmat{10}{0}{0}{26},\smallmat{2}{0}{0}{1},\smallmat{1}{1}{0}{1}$ & no & 3 & 1 & 36 & 3996 & $\smallminus 7^{1} 137^{3} 2083^{3}$ & $5$ & $7^{4}$\\&\multicolumn{11}{l}{$[a,a\smallminus 1,a,2039213a\smallminus 4078427,2003653476a\smallminus 3506903387]/(a^2\smallminus a\smallminus 1)$}\vspace{4pt}\\ 
\blb &\texttt{37B.10.6} & 456 & $\smallmat{10}{0}{0}{26},\smallmat{6}{0}{0}{25},\smallmat{1}{1}{0}{1}$ & no & 3 & 1 & 12 & 3996 & $\smallminus 7^{1} 11^{3}$ & $185$ & $7^{4} 37^{2}$\\&\multicolumn{11}{l}{$[0,0,0,94230675(11a\smallminus 82),28183538250(1859a\smallminus 13602)]/(a^2\smallminus a\smallminus 46)$}\vspace{4pt}\\
\bottomrule
\end{tabular}
\end{small}
\bigskip
\smallskip
\caption{Known exceptional $G_E(\ell)$ for non-CM elliptic curves $E$ over quadratic fields $\Q(\sqrt{D})$ with $j(E)\in\Q$ that are not base changes from $\Q$ ($\ell > 17$).}\label{table:Qbasechangetwist2}
\end{table}

\begin{table}
\begin{small}
\setlength{\extrarowheight}{0.7pt}
\begin{tabular}{@{\hskip -6pt}rlr@{\hskip 20pt}lrrrrrr@{\hskip 20pt}l}
&group&{\hskip 10pt}index&generators&{\hskip 10pt}$-1$&{\hskip 10pt}$t$&{\hskip 10pt}$d_0$&{\hskip 10pt}$d_1$&{\hskip 10pt}$d$&{\hskip 10pt}$D$& $N$\\&curve\\\toprule
&\texttt{5Nn.2.2[2]} & 80 & $\smallmat{2}{4}{2}{2},\smallmat{1}{0}{0}{4}$ & no & 2 & 3 & 3 & 6 & $5$ & $31^{2}$ \\&\multicolumn{10}{l}{\href{http://beta.lmfdb.org/EllipticCurve/2.2.5.1/961.3/b/1}{$[0,a\smallminus 1,1,42a\smallminus 95,192a\smallminus 332]/(a^2\smallminus a\smallminus 1)$}}\vspace{3pt}\\ 
&\texttt{5Ns.2.1[2]} & 60 & $\smallmat{2}{0}{0}{3},\smallmat{0}{1}{1}{0}$ & yes & 1 & 2 & 4 & 8 & $5$ & $2^{2} 19^{2}$ \\&\multicolumn{10}{l}{\href{http://beta.lmfdb.org/EllipticCurve/2.2.5.1/1444.2/a/1}{$[1,0,1,2a,2a\smallplus 2]/(a^2\smallminus a\smallminus 1)$}}\vspace{3pt}\\ 
&\texttt{5Nn.3.2[2]} & 40 & $\smallmat{3}{4}{2}{3},\smallmat{1}{0}{0}{4}$ & yes & 2 & 3 & 6 & 12 & $5$ & $31^{2}$ \\&\multicolumn{10}{l}{\href{http://beta.lmfdb.org/EllipticCurve/2.2.5.1/961.2/a/1}{$[0,\smallminus a\smallminus 1,1,\smallminus 1,2a\smallplus 1]/(a^2\smallminus a\smallminus 1)$}}\vspace{1pt}\\
\bottomrule
\end{tabular}
\end{small}
\bigskip\vspace{-2pt}
\caption{Some exceptional $G_E(\ell)$ for non-CM elliptic curves $E$ over quadratic fields $(\ell \le 11)$.}\label{table:Quadratic1}
\end{table}

\begin{table}
\begin{small}
\setlength{\extrarowheight}{0.7pt}
\begin{tabular}{@{\hskip -6pt}rlr@{\hskip 20pt}lrrrrrr@{\hskip 20pt}l}
&group&{\hskip 10pt}index&generators&{\hskip 10pt}$-1$&{\hskip 10pt}$t$&{\hskip 10pt}$d_0$&{\hskip 10pt}$d_1$&{\hskip 10pt}$d$&{\hskip 10pt}$D$& $N$\\&curve\\\toprule
&\texttt{13Ns} & 91 & $\smallmat{2}{0}{0}{7},\smallmat{0}{12}{1}{0},\smallmat{1}{0}{0}{2}$ & yes & 1 & 2 & 24 & 288 & $8$ & $5^{2} 7^{2} 263^{2}$\\&\multicolumn{10}{l}{$[a,1,a\smallplus 1,14455a\smallminus 27951,2058670a\smallminus 3164816]/(a^2\smallminus 2)$}\vspace{1pt}\\ 
\midrule
&\texttt{17B.4.3[2]} & 144 & $\smallmat{4}{0}{0}{13},\smallmat{3}{0}{0}{12},\smallmat{1}{1}{0}{1}$ & yes & 1 & 1 & 16 & 544 & $17$ & $2^{8} 17^{2}$\\&\multicolumn{10}{l}{$[0,a\smallminus 1,a\smallplus 1,62a\smallminus 174,378a\smallminus 955]/(a^2\smallminus a\smallminus 4)$}\vspace{3pt}\\ 
\blb &\texttt{17B.4.1} & 72 & $\smallmat{4}{0}{0}{13},\smallmat{1}{0}{0}{3},\smallmat{1}{1}{0}{1}$ & yes & 1 & 1 & 4 & 1088 & $5$ & $2^{2} 29^{2}$ \\&\multicolumn{10}{l}{\href{http://beta.lmfdb.org/EllipticCurve/2.2.5.1/3364.2/e/2}{$[1,1,a,\smallminus 3a\smallminus 2,2a]/(a^2\smallminus a\smallminus 1)$}}\vspace{2pt}\\
\bottomrule
\end{tabular}
\end{small}
\bigskip\vspace{-2pt}
\caption{Some exceptional $G_E(\ell)$ for non-CM $E$ over quadratic fields $(13\le \ell\le 23)$.}\label{table:Quadratic2}
\end{table}

\begin{table}
\begin{small}
\setlength{\extrarowheight}{1pt}
\begin{tabular}{@{\hskip -6pt}rlr@{\hskip 20pt}lrrrrrr@{\hskip 20pt}l}
&group&{\hskip 10pt}index&generators&{\hskip 10pt}$-1$&{\hskip 10pt}$t$&{\hskip 10pt}$d_0$&{\hskip 10pt}$d_1$&{\hskip 10pt}$d$&{\hskip 10pt}$D$& $N$\\&curve\\\toprule
\blb &\texttt{29B.7.1} & 120 & $\smallmat{7}{0}{0}{25},\smallmat{1}{0}{0}{2},\smallmat{1}{1}{0}{1}$ & no & 3 & 1 & 7 & 5684 & $-4$ & $29^{2}$ \\&\multicolumn{10}{l}{\href{http://beta.lmfdb.org/EllipticCurve/2.0.4.1/[841,41,1]/a/2}{$[1,a,1,\smallminus 1,0]/(a^2\smallplus 1)$}}\vspace{2pt}\\ 
&\texttt{29B.7.2} & 120 & $\smallmat{7}{0}{0}{25},\smallmat{2}{0}{0}{1},\smallmat{1}{1}{0}{1}$ & no & 3 & 1 & 28 & 5684 & $-4$ & $29^{2}$ \\&\multicolumn{10}{l}{\href{http://beta.lmfdb.org/EllipticCurve/2.0.4.1/[841,41,1]/a/1}{$[1,a,1,2080a\smallminus 3751,73352a\smallminus 79386]/(a^2\smallplus 1)$}}\vspace{4pt}\\ 
\blb &\texttt{29B.7.4} & 120 & $\smallmat{7}{0}{0}{25},\smallmat{4}{0}{0}{15},\smallmat{1}{1}{0}{1}$ & no & 3 & 1 & 14 & 5684 & $-4$ & $29^{4}$\\&\multicolumn{10}{l}{$[\smallminus a,a\smallminus 1,\smallminus a,\smallminus 141a\smallminus 157,\smallminus 2591a\smallminus 5674]/(a^2\smallplus 1)$}\vspace{2pt}\\
\bottomrule
\end{tabular}
\end{small}
\bigskip
\smallskip
\caption{Some exceptional $G_E(\ell)$ for non-CM elliptic curves $E$ over quadratic fields $(\ell >23)$.}\label{table:Quadratic3}
\end{table}

\FloatBarrier

\begin{table}[ht]
\begin{small}
\setlength{\extrarowheight}{1pt}
\begin{tabular}{@{\hskip -6pt}rlr@{\hskip 20pt}lrrrrrr@{\hskip 20pt}l}
&group&{\hskip 10pt}index&generators&{\hskip 10pt}$-1$&{\hskip 10pt}$t$&{\hskip 10pt}$d_0$&{\hskip 10pt}$d_1$&{\hskip 10pt}$d$&{\hskip 10pt}$D$& $N$\\&curve\\\toprule
&\texttt{7Cs.1.1} & 336 & $\smallmat{1}{0}{0}{3}$ & no & 3 & 1 & 1 & 6 & $-23$ & $2^{3} 7^{2}$ \\&\multicolumn{10}{l}{\href{http://beta.lmfdb.org/EllipticCurve/3.1.23.1/392.2/A/2}{$[1,\smallminus a^2\smallplus a,a^2\smallplus a,\smallminus 3a\smallminus 2,\smallminus 2a\smallminus 2]$}}\vspace{4pt}\\ 
&\texttt{7Cs.1.4} & 336 & $\smallmat{4}{0}{0}{6}$ & no & 3 & 1 & 2 & 6 & $-23$ & $2^{3} 7^{6}$\\&\multicolumn{10}{l}{$[1,4a^2\smallplus 2a\smallplus 4,0,56a^2\smallminus 114a\smallminus 120,\smallminus 316a^2\smallplus 224a\smallplus 512]$}\vspace{4pt}\\ 
&\texttt{7Cs.6.1} & 168 & $\smallmat{6}{0}{0}{6},\smallmat{1}{0}{0}{3}$ & yes & 3 & 1 & 2 & 12 & $-23$ & $2^{3} 5^{2} 7^{2}$ \\&\multicolumn{10}{l}{\href{http://beta.lmfdb.org/EllipticCurve/3.1.23.1/9800.4/A/1}{$[a\smallplus 1,a\smallplus 1,a^2\smallplus a,12a^2\smallplus 20a\smallplus 5,72a^2\smallminus 49a\smallminus 47]$}}\vspace{2pt}\\
\bottomrule
\end{tabular}
\end{small}
\bigskip
\caption{Some exceptional $G_E(\ell)$ for non-CM elliptic curves $E$ over $\Q[a]/(a^3-a^2+1)$.}\label{table:c23}
\end{table}

\end{document}